\documentclass[oneside,british]{amsart}
\usepackage[latin9]{inputenc}
\usepackage{color}
\usepackage{babel}
\usepackage{mathrsfs}
\usepackage{amstext}
\usepackage{amsthm}
\usepackage{amssymb}
\usepackage{esint}
\usepackage[unicode=true,
 bookmarks=true,bookmarksnumbered=false,bookmarksopen=false,
 breaklinks=false,pdfborder={0 0 0},pdfborderstyle={},backref=false,colorlinks=true]
 {hyperref}

\makeatletter
\numberwithin{equation}{section}
\numberwithin{figure}{section}
\theoremstyle{plain}
\newtheorem{thm}{\protect\theoremname}
\theoremstyle{remark}
\newtheorem{rem}[thm]{\protect\remarkname}
\theoremstyle{plain}
\newtheorem{assumption}[thm]{\protect\assumptionname}
\theoremstyle{plain}
\newtheorem{lem}[thm]{\protect\lemmaname}
\theoremstyle{plain}
\newtheorem{prop}[thm]{\protect\propositionname}
\theoremstyle{plain}
\newtheorem{fact}[thm]{\protect\factname}
\theoremstyle{plain}
\newtheorem{cor}[thm]{\protect\corollaryname}

\usepackage[dvipsnames]{xcolor}
\usepackage{fancyhdr}
\usepackage{amsmath,bm}

\makeatother

\providecommand{\assumptionname}{Assumption}
\providecommand{\corollaryname}{Corollary}
\providecommand{\factname}{Fact}
\providecommand{\lemmaname}{Lemma}
\providecommand{\propositionname}{Proposition}
\providecommand{\remarkname}{Remark}
\providecommand{\theoremname}{Theorem}

\begin{document}
\title{Regularization and asymptotic behaviour of Ornstein-Uhlenbeck evolution operators in infinite dimension}
\author{Francesco Bartaloni}
\pagestyle{fancy}
\fancyhead{}
\fancyhead[C]{\small\textsc{Asymptotic behaviour of Ornstein-Uhlenbeck operators in infinite dimension}}

\hypersetup{linkcolor=red!85!green}

\hypersetup{citecolor=blue}

\maketitle

\begin{abstract}

We are concerned with the properties of Ornstein-Uhlenbeck evolution
operators acting on functions defined in a Hilbert space and $p$-integrable
with respect to a suitable Gaussian measure. These operators provide
solutions to the infinite dimensional, non-autonomous backward Kolmogorov
equation.

The first part of the paper focuses on some general regularization
properties, while the second part carries on a deep analysis of the
asymptotic behaviour in the periodic case. In particular, we identify
the optimal convergence rate of the Ornstein-Uhlenbeck operators and
we give an optimality criterion depending only on the drift term of
the Kolmogorov equation.

\end{abstract}

%\begin{keywords}
%Ornstein-Uhlenbeck equation, non-autonomous, infinite dimension, regularization, periodic case, rate of convergence
%\end{keywords}

In the last decades, the study of autonomous equations in infinite
dimension has been widely developed and now the theory is well-established.
Among all, the most famous operator which has been object of investigation
is the Ornstein-Uhlenbeck operator. The interest in such operator
follows from its importance in applications, in particular in stochastic
analysis. Further, the associated semigroup admits an explicit representation,
which allows direct computations and consequently to deduce the main
properties of both the operator and the semigroup. We refer to the
monographs \cite{DAP-ZAB, DAP-ZAB2} for an exhaustive survey about
the Ornstein-Uhlenbeck operator and its generalizations in infinite
dimension, both in the deterministic and the stochastic case.

The situation completely changes if one considers non-autonomous equations.
Indeed, some results have been proved in finite dimension for initial
data both in the space of continuous functions and in $L^{p}$ spaces
(see for instance \cite{Add13, AddAngLor17, AngLor141, AngLor142, AngLor16, AngLorLun13, DAP-LUN, GEISS-LUN_0, GEISS-LUN, KunLorLun10, KunLorRha14, KunLorRha16,Lor11, LorLunSch16, LorLunZam10,ROL});
on the other hand, the theory of non-autonomous equations in infinite
dimension is far to be complete. In our knowledge, the first steps
in this field were moved in \cite{BigDef23,Cer21,Def23,OuyRoc}, where
some features and functional inequalities related to the Ornstein-Uhlenbeck
evolution operator are studied, such as strong Feller property, hypercontractivity,
Harnack inequality and logarithmic Sobolev inequalities.

In this paper will focus on the the regularization properties and
the asymptotic behaviour of the Ornstein-Uhlenbeck operator in $L^{p}$
spaces with respect to an evolution system of measures. These are
families of measures that generalize the notion of invariant measure
to a non-autonomous context.

Further, we provide an example satisfying all the assumptions of the
paper as well as a criterion of optimality of the convergence rate
of the Ornstein-Uhlenbeck operator.

\vspace{4mm}

\section{Preliminaries}\label{sez: intro}

\subsection{Setting of the problem}\label{subsec: setting}
\leavevmode\label{subsect: setting}

We consider a separable Hilbert space $\left(X,\left\langle \cdot,\cdot\right\rangle _{X}\right)$
and we denote by $\mathcal{L}\left(X\right)$ the set of all linear
bounded operators from $X$ to itself. Let $\left\{ U\left(t,s\right)\mid-\infty<s\leq t<+\infty\right\} $
be the evolution operator in $X$ associated to a family of operators
$A\left(t\right):D\left(A\left(t\right)\right)\subseteq X\to X$,
$t\in\mathbb{R}$.

This means that $U\left(\cdot,\cdot\right)$ satisfies:
\begin{equation}
\begin{cases}
{\displaystyle \frac{\partial U\left(t,s\right)}{\partial t}}=A\left(t\right)U\left(t,s\right) & \forall s<t,\\
U\left(t,s\right)=U\left(t,r\right)U\left(r,s\right) & \forall s\leq r\leq t,\\
U\left(\tau,\tau\right)=I & \forall\tau\in\mathbb{R}.
\end{cases}\label{eq: evoluzione}
\end{equation}

We assume that $U\left(t,s\right)\in\mathcal{L}\left(X\right)$ for
every $s\leq t$ and that the function $U\left(\cdot,\cdot\right)x$
is continuous in the half space $\left\{ \left(t,s\right)\in\mathbb{R}^{2}\mid s\leq t\right\} $
for every $x\in X$, i.e. the evolution operator $\left\{ U\left(t,s\right)\mid-\infty<s\leq t<+\infty\right\} $
is \emph{strongly continuous}.

We consider also a bounded family $\left\{ B\left(t\right)\mid t\in\mathbb{R}\right\} \subseteq\mathcal{L}\left(X\right)$
such that every function $B\left(\cdot\right)x$ is measurable.

Denote by $C_{b}\left(X\right)$ the space of all bounded and continuous
functions defined on $X$, and by $C_{b}^{1}\left(X\right)$ the space
of bounded continuously differentiable functions with bounded gradient.

For $\alpha\in\left(0,1\right)$, we denote also by $C_{b}^{\alpha}\left(X\right)$
the space of bounded $\alpha$-H\"older continuous functions on $X$,
namely the bounded functions $f:X\to\mathbb{R}$ such that:\[
\left[f\right]_{C^{\alpha}\left(X\right)}:=\sup_{\substack{x,y\in X \\ x\neq y}}\frac{\left|f\left(x\right)-f\left(y\right)\right|}{\left|x-y\right|_{X}^{\alpha}}<+\infty.
\]

The \emph{Ornstein-Uhlenbeck evolution operator} is thus defined by
the following formula:
\begin{align}
\left[P_{s,t}\varphi\right]\left(x\right) & :=\int_{X}\varphi\left(y\right)\mathcal{N}\left(U\left(t,s\right)x,Q\left(t,s\right)\right)\left(dy\right)\label{eq: def O-U formula}\\
 & =\int_{X}\varphi\left(U\left(t,s\right)x+y\right)\mathcal{N}\left(0,Q\left(t,s\right)\right)\left(dy\right)\nonumber \\
 & \quad\forall-\infty<s\leq t<+\infty\text{ and }\varphi\in C_{b}\left(X\right),\nonumber 
\end{align}
where the expression $\mathcal{N}\left(U\left(t,s\right)x,Q\left(t,s\right)\right)$
denotes the Gaussian measure on $X$ with mean $U\left(t,s\right)x$
and covariance:
\begin{equation}
Q\left(t,s\right):=\int_{s}^{t}U\left(t,r\right)B\left(r\right)B\left(r\right)^{*}U\left(t,r\right)^{*}\text{d}r.\label{eq: def cov Q}
\end{equation}
Here the integral is in the sense of Bochner. Clearly, the operators
in $\text{\eqref{eq: def cov Q}}$ are linear, bounded, non-negative
self-adjoint operators in $X$, and the $\mathcal{L}\left(X\right)$-valued
functions $Q\left(t,\cdot\right)$ are continuous in $\left(-\infty,t\right]$.

We necessarily have to assume that $Q\left(t,s\right)$ is a \emph{trace
class operator} for every $s\leq t$, as it will be formalized later
on. This means that:
\[
trQ\left(t,s\right):=\sum_{j=1}^{+\infty}\left\langle e_{j},Q\left(t,s\right)e_{j}\right\rangle <+\infty,
\]
where
\[
\mathscr{F}:=\left\{ e_{j}\mid j\in\mathbb{N}\right\} 
\]
is any Hilbert basis of $X$, which is henceforth and throughout the
whole paper to be considered fixed. It is a well-known fact that this
condition guarantees the existence of the Gaussian measure $\mathcal{N}\left(U\left(t,s\right)x,Q\left(t,s\right)\right)$.
\begin{rem}
The Ornstein-Uhlenbeck operator has a natural stochastic meaning.
Take $U\left(t,s\right)$ and $Q$$\left(t,s\right)$ as in $\text{\eqref{eq: evoluzione}}$
and $\text{\eqref{eq: def cov Q}}$, respectively, and let $\left\{ W\left(t\right)\mid t\in\mathbb{R}\right\} $
be a $X$-valued cylindrical Wiener process. Then, for every $x\in X$,
the Gaussian measure $\mathcal{N}\left(U\left(t,s\right)x,Q\left(t,s\right)\right)$
is the law of the random variable
\[
X\left(t;s,x\right)=U\left(t,s\right)x+\int_{s}^{t}U\left(t,r\right)B\left(r\right)dW\left(r\right),
\]
which is the unique mild solution of the stochastic differential equation
\[
\begin{cases}
dX\left(t\right)=A\left(t\right)X\left(t\right)dt+B\left(t\right)dW\left(t\right) & \forall t>s\\
X\left(s\right)=x\in X.
\end{cases}
\]

So for every $t\geq s$, $x\in X$ and $\varphi\in C_{b}\left(X\right)$,
$\left[P_{s,t}\varphi\right]\left(x\right)$ is the expected value
of $\varphi\left(X\left(t;s,x\right)\right)$.

Furthermore, $u\left(s,x;t\right):=\left[P_{s,t}\varphi\right]\left(x\right)$,
in the time variable $s$ and space variable $x$, is the mild solution
of the Kolmogorov backward equation:
\[
\begin{cases}
\partial_{s}u\left(s,x;t\right)+\left[L\left(s\right)u\left(s,\cdot;t\right)\right]\left(x\right)=0 & \forall s<t,\ x\in X\\
u\left(t,x;t\right)=\varphi\left(x\right) & \forall x\in X,
\end{cases}
\]
where the \emph{Orstein-Uhlenbeck differential operator} $L\left(\cdot\right)$
is formally defined by
\begin{equation}
\left[L\left(s\right)v\right]\left(x\right):=\frac{1}{2}tr\left[B\left(s\right)B\left(s\right)^{*}\nabla^{2}v\left(x\right)\right]+\left\langle A\left(s\right)x,\nabla v\left(x\right)\right\rangle _{X}\quad\forall s\in\mathbb{R},\ x\in X.\label{eq: gen di O-U}
\end{equation}
\end{rem}

\vspace{2mm}

\textbf{Note.} We choose to distribute the assumptions along the paper,
introducing them were they are needed to the development of the theory.
This choice is aimed at showing the true relationship between assumptions
and theorems, by highlighting where the former are used to obtain
a specific result, and where they are not.

For the reader's convenience, we provide in advance a complete list
with a brief description.
\begin{itemize}
\item[$\triangleright$] Assumption $\text{\ref{assu: omega_0 neg}}$: existence and exponential
decay of $U\left(t,s\right)$\vspace{1mm}.
\item[$\triangleright$] Assumption $\text{\ref{assu: B(t) limitati}}$: boundedness of $\left\{ B\left(t\right)\mid t\in\mathbb{R}\right\} $
in $\mathcal{L}\left(X\right)$\vspace{1mm}.
\item[$\triangleright$] Assumption $\text{\ref{assu: tracce finite}}$: uniform boundedness
of $trQ\left(t,s\right)$ with respect to $t$\vspace{1mm}.
\item[$\triangleright$] Assumption $\text{\ref{ass: controllabilita}}$: controllability
condition\vspace{1mm}.
\item[$\triangleright$] Assumption $\text{\ref{assu: mu_t non degeneri}}$: non-degeneracy
of the measures $\mathcal{N}\left(0,Q\left(t,-\infty\right)\right)$\vspace{1mm}.
\item[$\triangleright$] Assumption $\text{\ref{assu: periodicita}}$: periodicity of the
system\vspace{1mm}.
\item[$\triangleright$] Assumption $\text{\ref{ass: U* ha autov di mod max}}$, which prescribes
that $U\left(T,0\right)^{*}$ has eigenvalues of maximum modulus.
\end{itemize}
\subsection{Structure of the paper and main results}
\leavevmode\label{subsect: struttura}

Our primary assumption is that the evolution operator $\left\{ U\left(t,s\right)\mid s\leq t\right\} $
has an exponential decay rate, meaning that there exist $M,-\omega>0$
such that 
\begin{equation}
\|U(t,s)\|_{\mathcal{L}(X)}\leq Me^{\omega(t-s)}\quad\forall s\leq t.\label{eq: decadim}
\end{equation}
This allows in particular to define the (bounded) operator $Q\left(t,-\infty\right)$,
which is indeed a trace class operator if we assume that, for every
$t\in\mathbb{R}$, ${\rm Tr}Q(t,s)<+\infty$, uniformly in $s\leq t$.

According to the properties of $Q\left(t,-\infty\right)$, the centered
Gaussian measure $\mu_{t}=\mathcal{N}(0,Q(t,-\infty))$ is well defined
and the family of measures $\{\mu_{t}\mid t\in\mathbb{R}\}$ satisfies
\[
\int_{X}P_{s,t}\varphi d\mu_{s}=\int_{X}\varphi d\mu_{t},\quad\forall-\infty<s<t<+\infty,\ \varphi\in C_{b}(X).
\]
Any such family is called an evolution system of measures with respect
to the evolution operator $\left\{ P_{s,t}\mid s\leq t\right\} $.
Thanks to this property, we will prove that for every $s<t$ and $p\in\left[1,+\infty\right)$,
the operator $P_{s,t}$ - which \emph{a priori} is defined on $C_{b}(X)$
- extends to a contraction operator on $L^{p}(X,\mu_{t})$ with values
in $L^{p}(X,\mu_{s})$. The theory is then developed in relation to
the extended operators, by proving various results.

The first group of results concerns the regularization properties
of $P_{s,t}$. These are significant in themselves, and at the same
time are preparatory to the asymptotic analysis of the periodic case,
that gathers the second group of results.

Let us briefly resume the conclusions of the asymptotic analysis.
Assume that the coefficients $A\left(\cdot\right)$ and $B\left(\cdot\right)$
are periodic with period $T$, and let $\omega_{0}$ be the decay
rate of $\left\{ U\left(t,s\right)\mid s\leq t\right\} $ - namely,
the least $\omega\in\mathbb{R}$ such that there exists $M>0$ satisfying
\eqref{eq: decadim}. By assumption $\omega_{0}<0$. Then, for every
$p\in\left(1,+\infty\right)$ and every $\omega\in\left(\omega_{0},0\right)$,
there exists a positive constant $C_{T,\omega}$ such that for every
$f\in L^{p}(X,\mu_{t})$: 
\begin{equation}
\left\Vert P_{s,t}f-\int_{X}fd\mu_{t}\right\Vert _{L^{p}(X,\mu_{s})}\leq C_{T,\omega}e^{\omega(t-s)}\|f\|_{L^{p}(X,\mu_{t})},\quad\forall s<t.\label{eq: intro_conv_asint}
\end{equation}
It is notable that the regularization properties of $P_{s,t}$ proved
in Section \ref{sez: regolarizz} imply that every such operator is
compact, and that the spectral radius of $P_{0,T}$ is at most $e^{\omega_{0}T}$
(the spectral radius of $U\left(T,0\right)$).

Furthermore, if all the eigenvalues of $P_{0,T}$ with maximum modulus
are semisimple, then there exists $C_{T,\omega_{0}}>0$ such that
for every $f\in L^{p}(X,\mu_{t})$:
\begin{equation}
\left\Vert P_{s,t}f-\int_{X}fd\mu_{t}\right\Vert _{L^{p}(X,\mu_{s})}\leq C_{T,\omega_{0}}e^{\omega_{0}(t-s)}\|f\|_{L^{p}(X,\mu_{t})},\quad\forall s<t.\label{app_opt_rate_conv}
\end{equation}
The above estimate is optimal in the sense that if $\omega<\omega_{0}$,
then there exists some $f\in L^{p}(X,\mu_{t})$ such that $\text{\eqref{eq: intro_conv_asint}}$
is not true, whatever is $C_{T,\omega}$.

We stress that, in general, it is not easy to verify that all the
eigenvalues of $P_{0,T}$ with maximum modulus are semisimple, so
as to obtain the optimal rate of convergence. Fortunately, we are
able to prove that\textcolor{red}{{} }if $U(T,0)^{*}$ has some eigenvalues
of maximum modulus $e^{\omega_{0}T}$ and every such eigenvalue is
semisimple, then the same holds for $P_{0,T}$ - and therefore the
optimal estimate $\text{\eqref{app_opt_rate_conv}}$ holds.
\begin{rem}
We stress that even though $P_{s,t}$ is defined upon two families
of operators $A\left(\cdot\right)$ and $B\left(\cdot\right)$ (through
operators $U\left(t,s\right)$ and $Q\left(t,s\right)$), the above
optimality criterion concerns only the properties of $A\left(\cdot\right)$.
Thus, the fact that $P_{s,t}$ actually reaches the maximum possible
speed of convergence depends only on the first order coefficients
of its generator $L\left(s\right)$ in $\text{\eqref{eq: gen di O-U}}$.
\end{rem}

Finally, we apply this optimality criterion to a concrete example.
Taking $X=L^{2}(\Omega)$, with $\Omega\subseteq\mathbb{R}^{N}$ being
a bounded domain, we consider a family $\{a_{ij}(\cdot,\cdot)\mid i,j=1,\ldots,N\}$
of matrix-valued functions defined in $\mathbb{R}\times\overline{\Omega}$,
H\"older continuous and periodic in time, such that $\left[a_{ij}\left(t,x\right)\right]{}_{i,j=1}^{N}$
is a uniformly elliptic symmetric matrix for every $t\in\mathbb{R}$
and $x\in\overline{\Omega}$. For every $t\in\mathbb{R}$ we denote
by $A(t)$ the operator associated to the bilinear form
\[
a\left(t\right)\left(u,v\right)=-\int_{\Omega}\left\{ \sum_{i,j=1}^{N}D_{i}u\left(x\right)a_{ij}\left(t,x\right)D_{j}v\left(x\right)+u\left(x\right)v\left(x\right)\right\} dx\quad\forall u,v\in H^{1}(\Omega).
\]
Further, for every $t\in\mathbb{R}$ we set
\[
B\left(t\right):=A\left(t\right)^{-\theta}
\]
for a suitable $\theta\in\left(0,1/2\right)$. It follows that $\left\{ A\left(t\right)\right\} _{t\in\mathbb{R}}$
generates an evolution operator $\left\{ U\left(t,s\right)\right\} _{t\geq s}$
satisfying $\left\Vert U\left(t,s\right)\right\Vert _{\mathcal{L}\left(X\right)}\leq e^{-\left(t-s\right)}$.
Moreover, $e^{-\left(t-s\right)}$ is the unique eigenvalue of $U\left(T,0\right)^{*}$
with maximum modulus and it is semisimple. Thus, the associated Ornstein-Uhlenbeck
transition operator $\left\{ P_{s,t}\right\} _{t\geq s}$ fulfills
(\ref{app_opt_rate_conv}).

\vspace{2mm}

The paper is thus organized as follows.

In Section \ref{sez: O-U Cb} we set up the basic theory of the Ornstein-Uhlenbeck
evolution operator, considered as acting on regular bounded functions.
We show that it preserves regularity and we give a first pointwise
convergence result.

Section \ref{sez: regolarizz} is devoted to the study of the regularization
properties of the operator in $L^{p}$ spaces, with respect to a family
of Gaussian measures having the aforementioned invariance-type property
(that of being an evolution system of measures). These are proven
to be the only ones with finite moments of some order. To prove that
the Ornstein-Uhlenbeck operators enhance regularity, we have to consider
both scalar and vector-valued Sobolev spaces: this drives the need
that the related measures are non-degenerate and, ultimately, that
of a proper definition of such spaces, which is not trivial in the
vector-valued case. As a consequence of the scalar regularization
property, we obtain a compactness result which is pivotal for the
subsequent spectral analysis (Proposition \ref{prop: operatore compatto}).
All the results up to Section \ref{sez: regolarizz} hold for a general,
not necessarily time-periodic Ornstein-Uhlenbeck evolution operator.

In Section \ref{sez: tasso conv}, we use the spectral theory of compact
operators to carry on an accurate analysis, in the time-periodic case,
of the rate of convergence of the operator with respect to the $L^{p}$
norm, as outlined above.

The example is discussed in Section \ref{sez: esempio}.

\section{The Ornstein-Uhlenbeck evolution operator in spaces of continuous bounded functions}\label{sez: O-U Cb}

We begin with the structural assumptions.
\begin{assumption}
\label{assu: omega_0 neg}i) For every $t\in\mathbb{R}$, $A\left(t\right):D\left(A\left(t\right)\right)\subseteq X\to X$
is a linear operator such that the family  $\left\{ A\left(t\right)\mid t\in\mathbb{R}\right\} $
generates a strongly continuous evolution operator $\left\{ U\left(t,s\right)\mid s<t\right\} $.

ii) The set
\begin{align*}
\Omega_{U} & :=\left\{ \omega\in\mathbb{R}\mid\exists M>0:\forall s\leq t:\left\Vert U\left(t,s\right)\right\Vert _{\mathcal{L}\left(X\right)}\leq Me^{\omega\left(t-s\right)}\right\} 
\end{align*}
is non-empty and the evolution operator $U\left(\cdot,\cdot\right)$
satisfies the following growth condition:
\[
\omega_{0}:=\inf\Omega_{U}<0.
\]
\end{assumption}

\begin{assumption}
\label{assu: B(t) limitati}For every $t\in\mathbb{R}$, $B\left(t\right)\in\mathcal{L}\left(X\right)$
and the family $\left\{ B\left(t\right)\mid t\in\mathbb{R}\right\} $
is bounded. Further, the function $B\left(\cdot\right)x$ is measurable
for every $x$.
\end{assumption}

These assumptions guarantee in particular that we can pass to the
limit for $s\to-\infty$ in $\text{\eqref{eq: def cov Q}}$. Set,
for every $t\in\mathbb{R}$:
\begin{equation}
Q\left(t,-\infty\right):=\int_{-\infty}^{t}U\left(t,r\right)B\left(r\right)B\left(r\right)^{*}U\left(t,r\right)^{*}dr.\label{eq:def Q(t,-infty)}
\end{equation}
Hence, $Q\left(t,-\infty\right)$ is the $\mathcal{L}\left(X\right)$-limit
of $Q\left(t,s\right)$ for $s\to-\infty$.

The (centered) Gaussian measures with covariance operator defined
in $\text{\eqref{eq:def Q(t,-infty)}}$ are of prominent importance
in this paper. Anyway, at this stage, one cannot be sure that any
operator of the form $Q\left(t,-\infty\right)$ is trace class - and
thus that the measure $\mathcal{N}\left(0,Q\left(t,-\infty\right)\right)$
actually exists. The following assumption ensures that this is the
case, and that every function $trQ\left(t,\cdot\right)$ is continuous
up to $-\infty$ included.

\vspace{2mm}
\begin{assumption}
\label{assu: tracce finite}For every $s\leq t$ the operator $Q\left(t,s\right)$
in $\text{\eqref{eq: def cov Q}}$ is trace class, and:
\[
\sup_{s\leq t}trQ\left(t,s\right)<+\infty\quad\forall t\in\mathbb{R}.
\]
\end{assumption}

As a consequence, we have:
\begin{lem}
\label{lem: Q(t,-infty) trace class}Let $t\in\mathbb{R}$. Then:

i) the operator $Q\left(t,-\infty\right)$ is trace class;

ii) the series of functions $\sum_{j=1}^{+\infty}\left\langle e_{j},Q\left(t,\cdot\right)e_{j}\right\rangle $
converges uniformly to $trQ\left(t,\cdot\right)$ in $\left(-\infty,t\right]$;

iii) the function $trQ\left(t,\cdot\right)$ is non-negative, decreasing
and continuous in $\left[-\infty,t\right]$, and
\begin{equation}
\lim_{s\to-\infty}trQ\left(t,s\right)=\sup_{s\leq t}trQ\left(t,s\right)=trQ\left(t,-\infty\right)\label{eq: sup tracce uguale traccia}
\end{equation}
\end{lem}

\begin{proof}
i) By $\text{\eqref{eq: def cov Q}}$, the function $\left\langle x,Q\left(t,\cdot\right)x\right\rangle _{X}$
is non-negative and decreasing in $\left(-\infty,t\right]$ for every
$x\in X$. Hence, for every $N\in\mathbb{N}$:
\begin{align}
\sum_{j=1}^{N}\left\langle e_{j},Q\left(t,-\infty\right)e_{j}\right\rangle _{X} & =\sum_{j=1}^{N}\lim_{s\to-\infty}\left\langle e_{j},Q\left(t,s\right)e_{j}\right\rangle _{X}=\sup_{s\leq t}\sum_{j=1}^{N}\left\langle e_{j},Q\left(t,s\right)e_{j}\right\rangle _{X}\nonumber \\
 & \leq\sup_{s\leq t}trQ\left(t,s\right).\label{eq: trQ(t,-infty) finita}
\end{align}
Therefore, the series $\sum_{j=1}^{\infty}\left\langle e_{j},Q\left(t,-\infty\right)e_{j}\right\rangle _{X}$
converges and $Q\left(t,-\infty\right)$ is trace class.

ii) By definition the series of functions $\sum_{j=1}^{+\infty}\left\langle e_{j},Q\left(t,\cdot\right)e_{j}\right\rangle _{X}$
converges to $trQ\left(t,\cdot\right)$ pointwise in $\left(-\infty,t\right]$;
by point i) we immediately obtain that the convergence is uniform,
because
\[
\sup_{s\leq t}\sum_{j=N}^{+\infty}\left\langle e_{j},Q\left(t,s\right)e_{j}\right\rangle _{X}\leq\sum_{j=N}^{+\infty}\left\langle e_{j},Q\left(t,-\infty\right)e_{j}\right\rangle _{X}\to0\text{ as }N\to\infty.
\]
iii) Clearly, $trQ\left(t,\cdot\right)$ is non-negative and decreasing
because the general term of the series $\left\langle e_{j},Q\left(t,\cdot\right)e_{j}\right\rangle _{X}$
is such.

By the uniform convergence in point ii), we have for every $s_{0}\in\left(-\infty,t\right]$:
\[
\lim_{s\to s_{0}}trQ\left(t,s\right)=\sum_{j=1}^{+\infty}\lim_{s\to s_{0}}\left\langle e_{j},Q\left(t,s\right)e_{j}\right\rangle _{X}=trQ\left(t,s_{0}\right),
\]
so $trQ\left(t,\cdot\right)$ is continuous $\left(-\infty,t\right]$.

Further, since, for every $s\leq t$, $Q\left(t,s\right)\leq Q\left(t,-\infty\right)$
in the operator sense, we have:
\[
\lim_{s\to-\infty}trQ\left(t,s\right)=\sup_{s\leq t}trQ\left(t,s\right)\leq trQ\left(t,-\infty\right).
\]
Since the opposite inequality also holds by $\text{\eqref{eq: trQ(t,-infty) finita}}$,
we obtain the validity of $\text{\eqref{eq: sup tracce uguale traccia}}$,
and $trQ\left(t,\cdot\right)$ is continuous also at $-\infty$.
\end{proof}
\begin{rem}
\label{rem: per Q(t,s) iniettivo} Assumption $\text{\ref{assu: tracce finite}}$
and Lemma $\text{\ref{lem: Q(t,-infty) trace class}}$ ensure that,
for every $s<t$, the centered Gaussian measures with covariance operator
$Q\left(t,s\right)$ and $Q\left(t,-\infty\right)$ are well defined.
We underline two basic facts about the non-degeneracy of these measures. 

First, $KerQ\left(t,-\infty\right)\subseteq KerQ\left(t,s\right)$,
so if $\mathcal{N}\left(0,Q\left(t,s\right)\right)$ is non-degenerate,
then also $\mathcal{N}\left(0,Q\left(t,-\infty\right)\right)$ is
non-degenerate.

Indeed, it is easy to check that any operator $T$ of the form:
\[
T=\int_{a}^{b}S\left(r\right)S\left(r\right)^{*}dr,
\]
with, say, $S\left(r\right)\in\mathcal{L}\left(X\right)$ for almost
every $-\infty\leq a<r<b$ , satisfies, for every $x\in X$:
\begin{equation}
x\in KerT\iff x\in KerS\left(r\right)^{*}\ \text{for a. e. }r\in\left(a,b\right).\label{eq: ker cov}
\end{equation}
Therefore $KerQ\left(t,-\infty\right)\subseteq KerQ\left(t,s\right)$.

Secondly, observe that if there exists $\epsilon>0$ such that the
operator $B\left(r\right)^{*}$ is injective for almost every $r\in\left(t-\epsilon,t\right)$,
then $Q\left(t,s\right)$ is injective, and so is $Q\left(t,-\infty\right)$.

Indeed, fix $x\in KerQ\left(t,s\right)$. Then, by $\text{\eqref{eq: ker cov}}$,
$x\in KerB\left(r\right)^{*}U\left(t,r\right)^{*}$, in particular,
for almost every $r\in\left(t-\epsilon,t\right)$; moreover, $KerB\left(r\right)^{*}=\left\{ 0\right\} $
and thus $KerB\left(r\right)^{*}U\left(t,r\right)^{*}\subseteq KerU\left(t,r\right)^{*}$
for almost every $r\in\left(t-\epsilon,t\right)$. Hence, $U\left(t,\cdot\right)^{*}x=0$
almost everywhere in $\left(t-\epsilon,t\right)$ which implies $x=U\left(t,t\right)x=0$
by the continuity of $U\left(t,\cdot\right)x$.
\end{rem}

We assume also that the so called ``controllability condition'' holds
for our system.
\begin{assumption}
\label{ass: controllabilita}The following condition is satisfied:
\[
ImU\left(t,s\right)\subseteq ImQ\left(t,s\right)^{\frac{1}{2}}
\]
for every $-\infty<s<t<+\infty$.
\end{assumption}

The latter is a rather standard assumption that will be used in Section
$\text{\ref{sez: regolarizz}}$.
\begin{rem}
\label{rem: controll}The reason why Assumption $\text{\ref{ass: controllabilita}}$
is called ``controllability condition'' can be explained in the following
way. Denote by $L_{s,t}:L^{2}\left(s,t;X\right)\to X$ the following
linear operator:
\[
L_{s,t}\varphi:=\int_{s}^{t}U\left(t,r\right)B\left(r\right)\varphi\left(r\right)dr\quad\forall\varphi\in L^{2}\left(s,t;X\right).
\]
Then $L_{s,t}$ is bounded, $B\left(\cdot\right)^{*}U\left(t,\cdot\right)^{*}x\in L^{2}\left(s,t;X\right)$
for every constant $x\in X$ and we have:
\begin{align*}
\left\langle L_{s,t}\varphi,x\right\rangle _{X} & =\left\langle \int_{s}^{t}U\left(t,r\right)B\left(r\right)\varphi\left(r\right)dr,x\right\rangle _{X}\\
 & =\int_{s}^{t}\left\langle \varphi\left(r\right),B\left(r\right)^{*}U\left(t,r\right)^{*}x\right\rangle _{X}dr\\
 & =\left\langle \varphi,B\left(\cdot\right)^{*}U\left(t,\cdot\right)^{*}x\right\rangle _{L^{2}\left(s,t;X\right)},
\end{align*}
for every $\varphi\in L^{2}\left(s,t;X\right)$. Thus the adjoint
operator $L_{s,t}^{*}:X\to L^{2}\left(s,t;X\right)$ is:
\[
L_{s,t}^{*}x=B\left(\cdot\right)^{*}U\left(t,\cdot\right)^{*}x\quad\forall x\in X.
\]
Therefore, for every $x,y\in X$:
\begin{align*}
\left\langle Q\left(t,s\right)x,y\right\rangle _{X} & =\left\langle \int_{s}^{t}U\left(t,r\right)B\left(r\right)B\left(r\right)^{*}U\left(t,r\right)^{*}xdr,y\right\rangle _{X}\\
 & =\left\langle L_{s,t}^{*}x,L_{s,t}^{*}y\right\rangle _{L^{2}\left(s,t;X\right)}\\
 & =\left\langle L_{s,t}L_{s,t}^{*}x,y\right\rangle _{X},
\end{align*}
which implies $Q\left(t,s\right)=L_{s,t}L_{s,t}^{*}$ and $\left|Q\left(t,s\right)^{\frac{1}{2}}z\right|_{X}=\left|L_{s,t}^{*}z\right|_{L^{2}\left(s,t;X\right)}$
for every $z\in X$. We deduce from Proposition $\text{\ref{prop: immagini op Hilbert}}$:
\begin{equation}
ImQ\left(t,s\right)^{\frac{1}{2}}=ImL_{s,t}.\label{eq: immagini uguali Lst}
\end{equation}
Consider the following abstract linear equation in the unknown $\mathrm{u}$:
\begin{equation}
\begin{cases}
\mathrm{u}'\left(t\right)=A\left(t\right)\mathrm{u}\left(t\right)+B\left(t\right)\varphi\left(t\right) & \forall t>s\\
\mathrm{u}\left(s\right)=u_{0}
\end{cases}\label{eq: prob contr}
\end{equation}
with $\varphi\in L^{2}\left(s,t;X\right)$ and $u_{0}\in X$. It may
be considered to be the state equation of a linear control problem
(in case that $B$ is unbounded, a boundary control problem). The
mild solution is
\[
u\left(t;s,u_{0}\right)=U\left(t,s\right)u_{0}+L_{st}\varphi,\quad\forall t\geq s.
\]
The space $ImU\left(t,s\right)$ thus contains the values at time
$t$ of the solutions $u\left(\cdot;s,u_{0}\right)$ of $\text{\eqref{eq: prob contr}}$
with $\varphi=0$ (namely, the ``uncontrolled'' solutions), under
change of the initial datum $u_{0}\in X$. As a result of the identity
in $\eqref{eq: immagini uguali Lst}$, the space $Q\left(t,s\right)^{\frac{1}{2}}$
contains the values at time $t$ of the solutions of $\text{\eqref{eq: prob contr}}$
with $u_{0}=0$, under change of the control $\varphi$. Therefore,
Assumption $\text{\ref{ass: controllabilita}}$ is equivalent to the
statement that every state $x\in X$ reachable from an initial datum
without control (any possible ``natural state'', one could say)
is actually reachable from $0$ if the system is controlled in a suitable
way.

In this sense the system in $\text{\eqref{eq: prob contr}}$ can be
said controllable, under Assumption $\text{\ref{ass: controllabilita}}$
.

We stress that the verification that the example in Section $\text{\ref{sez: esempio}}$
satisfies Assumption $\text{\ref{ass: controllabilita}}$ will be
made by exploiting relation $\text{\eqref{eq: immagini uguali Lst}}$.
\end{rem}

\begin{rem}
\textcolor{red}{\label{rem: mis gauss e trasf F}}Let $\nu_{1}$ and
$\nu_{2}$ be two Borel measures on $X$ and denote by $\mathcal{B}\left(X\right)$
the Borel $\sigma$-algebra. Since $X$ is separable, a basic and
well known result in abstract measure theory ensures that $\nu_{1}=\nu_{2}$
on $\mathcal{B}\left(X\right)$ if and only if $\widehat{\nu_{1}}=\widehat{\nu_{2}}$
on $X^{*}=X$, where the symbol $\widehat{\cdot}$ denotes the Fourier
transform of a measure.

We recall that, for any measure $\mu$ acting on $\mathcal{B}\left(X\right)$,
$\widehat{\mu}$ is defined by
\[
\widehat{\mu}\left(\xi\right):=\int_{X}e^{i\left\langle \xi,y\right\rangle _{X}}\mu\left(dy\right)\quad\forall\xi\in X.
\]

Thus, the Fourier transform $\widehat{\nu}$ fully characterizes the
measure $\nu$.

If $\nu$ is Gaussian, it is possible to identify the mean and the
covariance operator of $\nu$ through $\widehat{\nu}$. Precisely,
$\nu=\mathcal{N}\left(a,Q\right)$ if and only if
\[
\widehat{\nu}\left(\xi\right)=\exp\left(i\left\langle a,\xi\right\rangle _{X}-\frac{1}{2}\left\langle \xi,Q\xi\right\rangle _{X}\right)\quad\forall\xi\in X.
\]

The convolution between two Borel measures $\nu_{1}$ and $\nu_{2}$
on $X$ is defined as the image under summation of the product measure,
on $X\times X$, between $\nu_{1}$ and $\nu_{2}$. Namely
\[
\nu_{1}\star\nu_{2}\left(A\right):=\int_{X}\chi_{A}\left(v+w\right)\nu_{1}\left(dv\right)\nu_{2}\left(dw\right)\quad\forall A\in\mathcal{B}\left(X\right).
\]
It is a consequence of this definition that:
\[
\widehat{\nu_{1}\star\nu_{2}}=\widehat{\nu_{1}}\cdot\widehat{\nu_{2}}.
\]

Finally, we recall that a sequence of Borel measures $\left(\mu_{n}\right)_{n}$
is said to \emph{converge weakly} to the Borel measure $\mu$, and
we write $\mu_{n}\rightharpoonup\mu$, if and only if
\[
\lim_{n\to\infty}\int_{X}f\left(x\right)\mu_{n}\left(dx\right)=\int_{X}f\left(x\right)\mu\left(dx\right)\quad\forall f\in C_{b}\left(X\right).
\]
If $\mu_{n}\rightharpoonup\mu$, then $\widehat{\mu_{n}}\to\widehat{\mu}$
pointwise in $X$, and uniformly on the compact subsets of $X$.

The weak convergence of Gaussian measures can be fully characterized:
if $\mu_{n}=\mathcal{N}\left(m_{n},Q_{n}\right)$ and $\mu=\mathcal{N}\left(m,Q\right)$
it is necessary and sufficient for $\mu_{n}$ to converge weakly to
$\mu$ that:
\begin{align*}
i)\  & m_{n}\to m\text{ in }X,\\
ii)\  & Q_{n}\to Q\text{ in }\mathcal{L}\left(X\right),\\
iii)\  & \int_{X}\left|x\right|_{X}^{2}\mu_{n}\left(dx\right)\to\int_{X}\left|x\right|_{X}^{2}\mu\left(dx\right).
\end{align*}
For a proof of the latter result, see e.g. \cite[Example 3.8.15]{BOG;GM}.

In particular, if all the Gaussian measures $\mu_{n}$ and $\mu$
are centered, $\mu_{n}\rightharpoonup\mu$ is equivalent to the convergence
of $Q_{n}$ to $Q$ in $\mathcal{L}\left(X\right)$ plus the convergence
of $trQ_{n}$ to $trQ$.
\end{rem}

\vspace{2mm}
\begin{rem}
\label{rem: conv deb gaussiane}For every fixed $t\in\mathbb{R}$,
as observed after the definitions $\text{\eqref{eq: def cov Q}}$
and $\text{\eqref{eq:def Q(t,-infty)}}$, and as a consequence of
point iii) in Lemma $\text{\ref{lem: Q(t,-infty) trace class}}$,
we have for every $s_{0}\in\left[-\infty,t\right)$:
\begin{align*}
 & \lim_{s\to s_{0}}\left\Vert Q\left(t,s\right)-Q\left(t,s_{0}\right)\right\Vert _{\mathcal{L}\left(X\right)}=0,\\
 & \lim_{s\to s_{0}}trQ\left(t,s\right)=trQ\left(t,s_{0}\right).
\end{align*}
We deduce from the final part of Remark $\text{\ref{rem: mis gauss e trasf F}}$
that:
\[
\mathcal{N}\left(0,Q\left(t,s\right)\right)\rightharpoonup\mathcal{N}\left(0,Q\left(t,s_{0}\right)\right)\quad\text{for }s\to s_{0}.
\]
\end{rem}

The following Proposition states some basic properties of the Ornstein-Uhlenbeck
transition operators as defined by formula $\text{\eqref{eq: def O-U formula}}$. 

First, such family forms a backward evolution operator in $C_{b}\left(X\right)$,
which is the natural space onto which these operators act.

Secondly, $C_{b}^{1}\left(X\right)$ is an invariant space for every
operator $P_{s,t}$; in other words, the $P_{s,t}$'s preserve the
$C^{1}$-regularity. For any $f\in C_{b}^{1}\left(X\right)$ we can
derive an explicit formula for $\nabla\left(P_{s,t}f\right)$ in terms
of $\nabla f$. In order to do this, it is convenient to define the
vector-valued counterpart of $P_{s,t}$, which will appear in this
formula. 

To this aim we denote
\begin{equation}
\overrightarrow{P_{s,t}}\Phi:=\int_{X}\Phi\left(U\left(t,s\right)\left(\cdot\right)+y\right)\mathcal{N}\left(0,Q\left(t,s\right)\right)\left(dy\right),\quad\forall\Phi\in C_{b}\left(X;X\right),\label{eq: def O-U vett formula}
\end{equation}
where the above integral is again in the sense of Bochner.

We also prove a first pointwise convergence result about $P_{s,t}$,
for $t-s\to\infty$.

We recall that a bounded operator is called a \emph{contraction} when
its norm is less then or equal to $1$.
\begin{prop}
\label{prop: op evoluzione}i) Let $s<t$ and $\overrightarrow{P_{s,t}}$
be defined by $\text{\eqref{eq: def O-U vett formula}}$. Then $P_{s,t}\left(C_{b}\left(X\right)\right)\subseteq C_{b}\left(X\right)$,
$\overrightarrow{P_{s,t}}\left(C_{b}\left(X;X\right)\right)\subseteq C_{b}\left(X;X\right)$,
and the operators:
\begin{align*}
P_{s,t} & :C_{b}\left(X\right)\to C_{b}\left(X\right)\\
\overrightarrow{P_{s,t}} & :C_{b}\left(X;X\right)\to C_{b}\left(X;X\right)
\end{align*}
are bounded and indeed contractions.

Further, the family $\left\{ P_{s,t}\mid-\infty<s\leq t<+\infty\right\} $
forms a backward evolution operator in $C_{b}\left(X\right)$, meaning
that
\begin{equation}
P_{s,t}f=P_{s,\tau}P_{\tau,t}f\quad\forall s<\tau<t,\ f\in C_{b}\left(X\right).\label{eq: op evoluzione}
\end{equation}

ii) For every $s<t$, $P_{s,t}\left(C_{b}^{1}\left(X\right)\right)\subseteq C_{b}^{1}\left(X\right)$
and the following formula holds:
\begin{equation}
\left[\nabla\left(P_{s,t}\varphi\right)\right]\left(x\right)=U\left(t,s\right)^{*}\left[\overrightarrow{P_{s,t}}\left(\nabla\varphi\right)\right]\left(x\right)\quad\forall\varphi\in C_{b}^{1}\left(X\right),\ x\in X.\label{eq: formula grad C1}
\end{equation}

iii) For each $f\in C_{b}\left(X\right)$ and $x\in X$:
\begin{equation}
\lim_{s\to-\infty}\left[P_{s,t}f\right]\left(x\right)=\int_{X}f\left(y\right)\mathcal{N}\left(0,Q\left(t,-\infty\right)\right)\left(dy\right)\quad\forall t\in\mathbb{R},\label{eq: conv P_s,t s a -infty}
\end{equation}
and
\begin{equation}
\lim_{t\to+\infty}\left\{ \left[P_{s,t}f\right]\left(x\right)-\int_{X}f\left(y\right)\mathcal{N}\left(0,Q\left(t,-\infty\right)\right)\left(dy\right)\right\} =0\quad\forall s\in\mathbb{R}\label{eq: conv P_s,t t a +infty}
\end{equation}
if $f\in C_{b}^{\alpha}\left(X\right)$ for some $\alpha\in\left(0,1\right)$.
\end{prop}

\begin{proof}
See Appendix \ref{appendice dim}.
\end{proof}
The asymptotic behaviour of the Ornstein-Uhlenbeck evolution operator
on continuous bounded functions is qualitatively described by relations
$\text{\eqref{eq: conv P_s,t s a -infty}}$ and $\text{\eqref{eq: conv P_s,t t a +infty}}$.
The main goal of this paper is to estimate the rate of the convergences
$\text{\eqref{eq: conv P_s,t s a -infty}}$ and $\text{\eqref{eq: conv P_s,t t a +infty}}$
for an extension of $P_{s,t}$ defined on spaces of possibly unbounded
functions that are integrable with respect to a suitable family of
Gaussian measures, under the assumption that the system is time-periodic.

To do this, we prove some regularization result for the extended operators,
which, besides, are interesting in themselves.

\vspace{2mm}

When there will be no room for confusion, we will omit the square
brackets. In particular, from now on, we will write $P_{s,t}\varphi\left(x\right)$
instead of $\left[P_{s,t}\varphi\right]\left(x\right)$ and $\overrightarrow{P_{s,t}}\Phi\left(x\right)$
instead of $\left[\overrightarrow{P_{s,t}}\Phi\right]\left(x\right)$.

\section{Regularization properties of the Ornstein-Uhlenbeck evolution operator in $L^p$ spaces}\label{sez: regolarizz}

This section is devoted to the extension of $\left\{ P_{s,t}\mid s\leq t\right\} $
and of its vectorial counterpart to spaces of possibly unbounded functions
that are integrable with respect to a suitable Gaussian measure, and
to the regularization property of the extended operators.

First, we prove a regularization result for the original evolution
operator, the one defined in $C_{b}\left(X\right)$. We already know
from point ii) in Proposition $\text{\ref{prop: op evoluzione}}$
that $P_{s,t}$ preserves the $C_{b}^{1}$-regularity; in order to
improve this result and pass from $C_{b}\left(X\right)$ to $C_{b}^{1}\left(X\right)$,
we need to apply the Cameron-Martin theorem to the Gaussian measure
appearing in $P_{s,t}$. This way we will prove in Proposition $\text{\ref{prop: CM formula}}$
that $P_{s,t}\left(C_{b}\left(X\right)\right)\subseteq C_{b}^{1}\left(X\right)$
and that, for every $f\in C_{b}\left(X\right)$, the directional derivative
of $P_{s,t}f$ can be explicitly described by means of the isometry
between the reproducing kernel and the Cameron-Martin space of the
measure $\mathcal{N}\left(0,Q\left(t,s\right)\right)$. This result
exploits the ``controllability condition'' in Assumption $\text{\ref{ass: controllabilita}}$.

From now on, we denote
\begin{equation}
\gamma_{t,s}:=\mathcal{N}\left(0,Q\left(t,s\right)\right)\quad\forall s<t,\label{eq: def gamma_ts}
\end{equation}
for simplicity of notation.

In the subsequent part of the section, we identify a suitable family
$\left\{ \mu_{t}\mid t\in\mathbb{R}\right\} $ of Gaussian measures
that allows to extend $P_{s,t}$ {[}$\overrightarrow{P_{s,t}}${]}
to the space $L^{p}\left(X,\mu_{t}\right)$ {[}$L^{p}\left(X,\mu_{t};X\right)${]},
with values in $L^{p}\left(X,\mu_{s}\right)$ {[}$L^{p}\left(X,\mu_{s};X\right)${]}.
Then, also the extended operators are shown to have a suitable regularization
property (Propositions $\text{\ref{prop: operatore compatto}}$ and
$\ref{prop: reg vett}$). The regularization of the extended operators
is carried on in three different subsections because the vector-valued
function space case is not a trivial extension of the scalar function
space case: even though the second case exploits the results of the
first case, it is necessary to take some measures and prove a series
of specific results in order to obtain the regularization property
also for $\overrightarrow{P_{s,t}}:L^{p}\left(X,\mu_{t};X\right)\to L^{p}\left(X,\mu_{s};X\right)$.

All this work is preparatory to the asymptotic analysis that will
be carried upon in the next section.

\vspace{2mm}

Clearly, the definitions of scalar and vector-valued Sobolev space
we deal with play a crucial role in the theory. As far as scalar functions
are concerned, we use the following notion. If $\nu$ is a non-degenerate
Gaussian measure on $X$, then the space $W^{1,p}\left(X,\nu\right)$
is defined as the domain of the closure of the gradient operator
\[
\nabla:C_{b}^{1}\left(X\right)\subseteq L^{p}\left(X,\nu\right)\to L^{p}\left(X,\nu;X\right).
\]
This means that a function $f\in L^{p}\left(X,\nu\right)$ belongs
to $W^{1,p}\left(X,\nu\right)$ if and only if there exists a sequence
$\left(f_{n}\right)_{n}\subseteq C_{b}^{1}\left(X\right)$ such that
$f_{n}\to f$ in $L^{p}\left(X,\nu\right)$ and $\left(\nabla f_{n}\right)_{n}$
converges in $L^{p}\left(X,\nu;X\right)$. In this case, denoting
again by $\nabla$ the closure of the former operator, $\nabla f$
is defined as the $L^{p}\left(X,\nu;X\right)$-limit of $\nabla f_{n}$.
Clearly, in order that this definition is well-posed, i.e. that it
does not depend on the choice of the sequence $\left(f_{n}\right)_{n}$,
it is necessary and sufficient to know that the gradient operator
is closable in $L^{p}\left(X,\nu\right)\times L^{p}\left(X,\nu;X\right)$,
which is a well-known result (see for instance \cite{DAP-ZAB}, Chapter
9.2).

The definition of $W^{1,p}\left(X,\nu;X\right)$ is more complicated,
to the point that we dedicate the entire Subsection $\text{\ref{subsez: Sob vett}}$
to it.

\vspace{2mm}

Eventually we make a some clarifications about the non degeneracy
of the measures $\gamma_{t,s}$ and $\mu_{t}$ defined in $\text{\eqref{eq: def gamma_ts}}$
and $\text{\eqref{def misura inv}}$ respectively.

We stress that both the definition of $W^{1,p}\left(X,\nu\right)$
and the definition of $W^{1,p}\left(X,\nu;X\right)$ require that
the Gaussian measure $\nu$ is non-degenerate. Thus, we will assume
that the measures we are interested in as far as Sobolev spaces are
concerned - namely those defined in $\text{\eqref{def misura inv}}$
- are actually non-degenerate. Nevertheless, the results of Subsections
$\text{\ref{subsez: reg Cb}}$ and $\text{\ref{subsez: estensione Lp}}$,
where we do not treat Sobolev spaces, hold in more general cases.

Precisely, in Subsection $\text{\ref{subsez: reg Cb}}$, the characterization
of the Cameron-Martin space (Proposition $\text{\ref{prop: caratt}}$)
is given at the most general level - actually, at a \emph{pre} measure
theoretic level since the operator $Q$ is assumed to be just a non-negative
self-adjoint operator, while the Cameron-Martin Theorem (Theorem $\text{\ref{Teo CM}}$)
is stated for a Gaussian measure $\mathcal{N}\left(a,Q\right)$ where
$Q$ is non-negative, self-adjoint and trace-class. Hence, the consequent
regularization result for continuous functions (Proposition $\text{\ref{prop: CM formula}}$)
holds for any possibly degenerate measure $\gamma_{t,s}$ defined
as in $\text{\eqref{eq: def gamma_ts}}$.

The three results of Subsection $\text{\ref{subsez: estensione Lp}}$
involve the family of measures $\left\{ \mu_{t}\mid t\in\mathbb{R}\right\} $
defined in $\text{\eqref{def misura inv}}$, but do not concern Sobolev
spaces with respect of $\mu_{t}$. Hence, all of three of these results,
namely the invariance property of the family with respect to $P_{s,t}$
(Proposition $\text{\ref{prop: invarianza}}$), its characterization
(Proposition $\text{\ref{prop: caratt mu_t}}$) and the extension
of $P_{s,t}$ to $L^{p}\left(X,\mu_{t}\right)$ (Proposition $\text{\ref{prop: estensione P_s,t}}$)
hold without assuming that the measures $\mu_{t}$ are non-degenerate.

From Subsection $\text{\ref{subsez: reg Lp}}$ on, the measures $\mu_{t}$
are assumed to be non-degenerate since we will deal with Sobolev spaces
defined with respect to $\mu_{t}$.

Note that, in virtue of Remark $\text{\ref{rem: per Q(t,s) iniettivo}}$,
the non degeneracy of $\gamma_{t,s}$ would imply that of $\mu_{t}$,
but the converse is not true in general.

\subsection{Regularization of continuous functions}
\leavevmode\label{subsez: reg Cb}

Our first regularization result relies upon the Cameron-Martin theorem.
To state this theorem in the form that is most suitable to our purposes,
we need to do some preparatory work.

Recall that if $Y$ is any other Hilbert space, then for any $T\in\mathcal{L}\left(X;Y\right)$,
$\left(KerT\right)^{\perp}=\overline{ImT^{*}}$, and consequently
\[
X=KerT\oplus\overline{ImT^{*}}.
\]
The \emph{pseudoinverse} of $T$ is the operator
\[
T^{-1}:ImT\subseteq Y\to X
\]
defined by
\[
T^{-1}y:=\underset{T^{-1}\left\{ y\right\} }{\text{argmin}}\left|\cdot\right|_{X}\quad\forall y\in ImT.
\]
The definition is well-posed since, for $y\in ImT$, the set $T^{-1}\left\{ y\right\} $
is closed, convex and non-empty. An equivalent definition is that
$T^{-1}y$ is the unique element of $T^{-1}\left\{ y\right\} \cap\left(KerT\right)^{\perp}$,
which implies $ImT^{-1}\subseteq\overline{ImT^{*}}$. In other terms,
$T^{-1}$ is characterized by the following relation:
\begin{equation}
x=T^{-1}y\iff x\in T^{-1}\left\{ y\right\} \cap\overline{ImT^{*}}=T^{-1}\left\{ y\right\} \cap\left(KerT\right)^{\perp}.\label{eq: caratt pseudoinverso}
\end{equation}
It is an easy consequence of the above relation that $T^{-1}$ is
closed even though not necessarily continuous with respect to the
topologies of $Y$ and $X$. We stress that $D\left(T^{-1}\right)=ImT$
may not be a closed subspace of $Y$.

If $Y=X$ and $T$ is self-adjoint, then $T^{-1}:ImT\to\overline{ImT}$
is symmetric, meaning that:
\[
\left\langle x_{1},T^{-1}x_{2}\right\rangle =\left\langle x_{2},T^{-1}x_{1}\right\rangle \quad\forall x_{1},x_{2}\in ImT.
\]
We stress that this is different and weaker than saying that $T^{-1}$
is self-adjoint. First, the existence of $\left(T^{-1}\right)^{*}$
requires that $ImT$ is dense in $X$, which means that $KerT=\left\{ 0\right\} $
so that $T^{-1}$ is the algebraic inverse of $T$ - which is clearly
not the most general case. Secondly, $T^{-1}=\left(T^{-1}\right)^{*}$
would imply that $ImT$ coincides with the domain of $\left(T^{-1}\right)^{*}$,
but the latter is in general bigger than the former.

Anyway, the set which serves as the domain of $\left(T^{-1}\right)^{*}$
can be defined independently of the density of $ImT$ - namely, independently
of the existence of $\left(T^{-1}\right)^{*}$ itself. Corollary $\text{\ref{cor: caratt dominio aggiunto}}$
actually provides a useful characterization of this set.

\vspace{2mm}

If $T$ is also non-negative, then the square root $T^{\frac{1}{2}}$
of $T$ can be defined as the unique non-negative operator $S\in\mathcal{L}\left(X\right)$
such that $S^{2}=T$. It is a trivial exercise to verify that $T^{\frac{1}{2}}$
is self-adjoint, that $KerT^{\frac{1}{2}}=KerT$, and thus $\overline{ImT^{\frac{1}{2}}}=\overline{ImT}$.
The pseudoinverse $T^{-\frac{1}{2}}:ImT^{\frac{1}{2}}\to\overline{ImT}$
of $T^{\frac{1}{2}}$ satisfies:
\[
T^{-1}=T^{-\frac{1}{2}}T^{-\frac{1}{2}}\quad\text{in }ImT
\]
by direct check, using $\text{\eqref{eq: caratt pseudoinverso}}$.

\vspace{4mm}

For a bounded, self-adjoint, non-negative operator $Q$ in $X$, and
an element $x\in X$, consider the following subset of $\mathbb{R}$:
\begin{equation}
\mathcal{D}\left(x;Q\right):=\left\{ \left\langle x,z\right\rangle _{X}\mid z\in X\text{ and }\left|Q^{\frac{1}{2}}z\right|_{X}\leq1\right\} .\label{eq: def norma CM}
\end{equation}
In other terms, the latter is the image - under the dual map $\left\langle \cdot,x\right\rangle $
- of the counterimage $Q^{-\frac{1}{2}}\left(\overline{B\left(0,1\right)_{X}}\right)$.

The set
\[
H_{Q}:=\left\{ x\in X\mid\sup\mathcal{D}\left(x;Q\right)<+\infty\right\} 
\]
is the \emph{Cameron-Martin space} of a Gaussian measure $\mathcal{N}\left(a,Q\right)$,
$a\in X$.

With this definitions and facts we can prove the following important
characterization, without any compactness assumption.
\begin{prop}
\label{prop: caratt}Let $Q\in\mathcal{L}\left(X\right)$ be a non-negative,
self-adjoint operator. Then:
\[
H_{Q}=ImQ^{\frac{1}{2}}.
\]

Further, if $x\in ImQ^{\frac{1}{2}}$, then $\sup\mathcal{D}\left(x;Q\right)=\left|Q^{-\frac{1}{2}}x\right|_{X}$.
\end{prop}

\begin{proof}
See Appendix \ref{appendice CM}.
\end{proof}
\begin{rem}
\label{rem: caratt app ImQ^1/2}It is usually ``difficult'' to prove
that a vector is in the image of a linear operator without knowing
that it is a closed set. Indeed, the inclusion $\left(\subseteq\right)$
in the statement of Proposition $\text{\ref{prop: caratt}}$ is the
hardest to prove. Thanks to this inclusion, we have that, if $Q$
is as in the hypotheses, then for every $x\in X$:
\[
x\in ImQ^{\frac{1}{2}}\iff\exists M_{x}>0:\forall z\in X:\left\langle x,z\right\rangle _{X}\leq M_{x}\left|Q^{\frac{1}{2}}z\right|_{X}.
\]
\end{rem}

Now we can state the Cameron-Martin theorem, applied to the case of
our interest, that of a (possibly degenerate) Gaussian measure in
a Hilbert space. We refer to \cite{DAP-ZAB} for the proof of this
classical theorem.
\begin{thm}[Cameron-Martin]
\textcolor{red}{\label{Teo CM} }Let $a\in X$ and $Q\in\mathcal{L}\left(X\right)$
be a non-negative, self-adjoint trace-class operator, and set $\gamma:=\mathcal{N}\left(a,Q\right)$.
Fix $h\in ImQ^{\frac{1}{2}}$.

Then the measure $\gamma\left(\cdot-h\right)$ is equivalent to $\gamma$,
and the density of $\gamma\left(\cdot-h\right)$ with respect to $\gamma$
has the form
\[
x\to\exp\left(\widehat{h}\left(x\right)-\frac{1}{2}\left|Q^{-\frac{1}{2}}h\right|_{X}^{2}\right)\quad\forall x\in X,
\]
where $\widehat{h}$ is a suitable function in $L^{2}\left(X,\gamma\right)$
satisfying:
\begin{align*}
i)\  & \int_{X}\hat{h}\left(x\right)\gamma\left(dx\right)=0,\\
ii)\  & \left\Vert \widehat{h}\right\Vert _{L^{2}\left(X,\gamma\right)}=\left|Q^{-\frac{1}{2}}h\right|_{X},\\
iii)\  & \gamma\circ\widehat{h}^{-1}\sim\mathcal{N}\left(0,\left|Q^{-\frac{1}{2}}h\right|_{X}^{2}\right).
\end{align*}

Further, the map $h\mapsto\widehat{h}$ is linear.
\end{thm}

\begin{rem}
Note that the classical formulation of the previous theorem identifies
$H_{Q}$ as the set where $\gamma$ is quasi-translation invariant.
Thus, the formulation in terms of the set $ImQ^{\frac{1}{2}}$ (instead
of $H_{Q}$) is possible due to Proposition $\text{\ref{prop: caratt}}$.

Note also that the property in $iii)$ means that the function $\widehat{h}$,
seen as a real random variable of the measure space $\left(X,\gamma\right)$,
has law $\mathcal{N}\left(0,\left|Q^{-\frac{1}{2}}h\right|_{X}^{2}\right)$
.
\end{rem}

\begin{rem}
\label{rem:def Lambda}For every $s<t$, the linear operator from
$X$ to $X$
\[
\Lambda\left(t,s\right):=Q\left(t,s\right)^{-\frac{1}{2}}U\left(t,s\right),
\]
which is well defined as a consequence of Assumption $\text{\ref{ass: controllabilita}}$,
is continuous. Indeed, since $Q\left(t,s\right)^{\frac{1}{2}}$ is
self-adjoint, its pseudoinverse $Q\left(t,s\right)^{-\frac{1}{2}}$
is closed in its domain $ImQ\left(t,s\right)^{\frac{1}{2}}$; thus
$\Lambda\left(t,s\right)$ is a closed operator from $X$ in $X$
(because $U\left(t,s\right)$ is bounded), so $\Lambda\left(t,s\right)$
is bounded by the Closed Graph Theorem.
\end{rem}

\vspace{2mm}

With the help of Theorem $\text{\ref{Teo CM}}$ we can prove the first
regularization result.
\begin{prop}
\textcolor{red}{\label{prop: CM formula} }Let $-\infty<s<t<+\infty$.
Then the operator $P_{s,t}$ satisfies $P_{s,t}\left(C_{b}\left(X\right)\right)\subseteq C_{b}^{1}\left(X\right)$,
and, for every $f\in C_{b}\left(X\right)$, the following formula
holds:
\begin{equation}
\left\langle \left[\nabla P_{s,t}f\right]\left(x\right),h\right\rangle _{X}=\int_{X}f\left(y+U\left(t,s\right)x\right)\cdot\widehat{U\left(t,s\right)h}\left(y\right)\gamma_{t,s}\left(dy\right)\quad\forall x,h\in X,\label{eq: grad P_s,t funz C_0}
\end{equation}
where the function $\widehat{U\left(t,s\right)h}$ is taken as in
Theorem $\text{\ref{Teo CM}}$.
\end{prop}

\begin{proof}
We fix $s<t$, $x\in X$ and $f\in C_{b}\left(X\right)$, and we prove
that:
\begin{align}
 & P_{s,t}f\left(x+h\right)-P_{s,t}f\left(x\right)-\int_{X}f\left(y+U\left(t,s\right)x\right)\cdot\widehat{U\left(t,s\right)h}\left(y\right)\gamma_{t,s}\left(dy\right)=o\left(h\right)\label{eq: grad P_s,t per funz regolari}
\end{align}
for $h\to0$. Since, as a consequence of the last statement in Theorem
$\text{\ref{Teo CM}}$, the third quantity in the left-hand member
depends linearly on $h$, relation $\text{\eqref{eq: grad P_s,t per funz regolari}}$
will imply that formula $\text{\eqref{eq: grad P_s,t funz C_0}}$
holds.

For every $h\in X$, we can apply Theorem $\text{\ref{Teo CM}}$ with
$a=0$, $Q=Q\left(t,s\right)$ (and thus $\gamma=\gamma_{t,s}$) and
$h=U\left(t,s\right)h$, since $U\left(t,s\right)h\in ImQ\left(t,s\right)^{\frac{1}{2}}$
by Assumption $\text{\ref{ass: controllabilita}}$. By the density
formula in Theorem $\text{\ref{Teo CM}}$, we have:
\begin{align*}
P_{s,t}f\left(x+h\right) & =\int_{X}f\left(y+U\left(t,s\right)\left(x+h\right)\right)\gamma_{t,s}\left(dy\right)\\
 & =\int_{X}f\left(y+U\left(t,s\right)x\right)\\
 & \cdot\exp\left(\widehat{U\left(t,s\right)h}\left(y\right)-\frac{1}{2}\left|\Lambda\left(t,s\right)h\right|_{X}^{2}\right)\gamma_{t,s}\left(dy\right),
\end{align*}
where $\Lambda\left(t,s\right)$ is the operator defined in Remark
$\text{\ref{rem:def Lambda}}$.

The change of variable $\tau=\widehat{U\left(t,s\right)h}\left(y\right)$
formally implies by point iii) in Theorem $\text{\ref{Teo CM}}$:
\[
\gamma_{t,s}\left(dy\right)=\gamma_{t,s}\circ\widehat{U\left(t,s\right)h}^{-1}\left(d\tau\right)=\mathcal{N}\left(0,\left|\Lambda\left(t,s\right)h\right|_{X}^{2}\right)\left(d\tau\right);
\]
by applying such change of variable followed by $\sigma=\frac{\tau}{\left|\Lambda\left(t,s\right)h\right|_{X}}$
we obtain for a generic $h\in X$:
\begin{align*}
 & \left|P_{s,t}f\left(x+h\right)-P_{s,t}f\left(x\right)-\int_{X}f\left(y+U\left(t,s\right)x\right)\cdot\widehat{U\left(t,s\right)h}\left(y\right)\gamma_{t,s}\left(dy\right)\right|\\
\leq & \int_{X}\left|f\left(y+U\left(t,s\right)x\right)\right|\left|\exp\left(\widehat{U\left(t,s\right)h}\left(y\right)-\frac{1}{2}\left|\Lambda\left(t,s\right)h\right|_{X}^{2}\right)-1-\widehat{U\left(t,s\right)h}\left(y\right)\right|\gamma_{t,s}\left(dy\right)\\
\leq & \left\Vert f\right\Vert _{L^{\infty}\left(X\right)}\int_{\mathbb{R}}\left|\exp\left(\tau-\frac{1}{2}\left|\Lambda\left(t,s\right)h\right|_{X}^{2}\right)-1-\tau\right|\mathcal{N}\left(0,\left|\Lambda\left(t,s\right)h\right|_{X}^{2}\right)\left(d\tau\right)\\
=\, & \frac{\left\Vert f\right\Vert _{L^{\infty}\left(X\right)}}{\left|\Lambda\left(t,s\right)h\right|_{X}\sqrt{2\pi}}\int_{\mathbb{R}}\left|\exp\left(\tau-\frac{1}{2}\left|\Lambda\left(t,s\right)h\right|_{X}^{2}\right)-1-\tau\right|\exp\left(-\frac{\tau^{2}}{2\left|\Lambda\left(t,s\right)h\right|_{X}^{2}}\right)d\tau\\
=\, & \frac{\left\Vert f\right\Vert _{L^{\infty}\left(X\right)}}{\sqrt{2\pi}}\int_{\mathbb{R}}\left|\exp\left(\sigma\left|\Lambda\left(t,s\right)h\right|_{X}-\frac{1}{2}\left|\Lambda\left(t,s\right)h\right|_{X}^{2}\right)-1-\sigma\left|\Lambda\left(t,s\right)h\right|_{X}\right|e^{-\frac{\sigma^{2}}{2}}d\sigma.
\end{align*}
For every $\sigma\in\mathbb{R}$ and every small $\left|h\right|_{X}$,
the Mean value theorem leads to the inequality :
\begin{align*}
 & \left|\exp\left(\sigma\left|\Lambda\left(t,s\right)h\right|_{X}-\frac{1}{2}\left|\Lambda\left(t,s\right)h\right|_{X}^{2}\right)-1-\sigma\left|\Lambda\left(t,s\right)h\right|_{X}\right|\\
\leq & \left|\Lambda\left(t,s\right)h\right|_{X}^{2}\left\{ \frac{1}{2}+\left(\sigma+1\right)^{2}e^{\left|\sigma\right|\left\Vert \Lambda\left(t,s\right)\right\Vert _{\mathcal{L}\left(X\right)}}\right\} .
\end{align*}
Thus, $\text{\eqref{eq: grad P_s,t per funz regolari}}$ holds.
\end{proof}
\subsection{Extension to $L^p$ spaces}
\leavevmode\label{subsez: estensione Lp}

Now we define the extended Ornstein-Uhlenbeck evolution operator.
To define $P_{s,t}$ on integrable functions, the key tool tool will
be the non-autonomous analogous of the concept of invariant measure.
Denote:
\begin{equation}
\mu_{t}:=\mathcal{N}\left(0,Q\left(t,-\infty\right)\right)\quad\forall t\in\mathbb{R}.\label{def misura inv}
\end{equation}
The family of measures $\left\{ \mu_{t}\mid t\in\mathbb{R}\right\} $
has already appeared in the third part of the statement of Proposition
$\text{\ref{prop: op evoluzione}}$, but its role in the theory goes
far beyond that. Its most important property is the following.
\begin{prop}
\label{prop: invarianza}For every $s<t$, every $\varphi\in C_{b}\left(X\right)$
and every  $\Phi\in C_{b}\left(X;X\right)$:
\begin{align}
\int_{X}P_{s,t}\varphi\left(x\right)\mu_{s}\left(dx\right) & =\int_{X}\varphi\left(x\right)\mu_{t}\left(dx\right),\label{eq: muinv}\\
\int_{X}\overrightarrow{P_{s,t}}\Phi\left(x\right)\mu_{s}\left(dx\right) & =\int_{X}\Phi\left(x\right)\mu_{t}\left(dx\right),\label{eq: muinv vettoriale}
\end{align}
where $\overrightarrow{P_{s,t}}$ is the operator defined in $\text{\eqref{eq: def O-U vett formula}}$.
\end{prop}

\begin{proof}
See Appendix \ref{appendice dim}.
\end{proof}
Following \cite{DAP-LUN}, we call any family $\left\{ \nu_{t}\mid t\in\mathbb{R}\right\} $
satisfying $\text{\eqref{eq: muinv}}$, and thus $\text{\eqref{eq: muinv vettoriale}}$,
with $\mu_{t}=\nu_{t}$, an \emph{evolution system of measures} with
respect to $\left\{ P_{s,t}\mid s<t\right\} $.

\vspace{2mm}

It has been proven in \cite{KNA} that the family defined in $\text{\eqref{def misura inv}}$
is the unique evolution system of measures with respect to $\left\{ P_{s,t}\mid s<t\right\} $
if the maps $t\to B\left(t\right)$ and $\left(s,t\right)\to U\left(t,s\right)$
are periodic (which implies that also $t\to\mu_{t}$ is periodic with
the same period).

In the general case, we have the following partial characterization.
\begin{prop}
\label{prop: caratt mu_t}Let $\left\{ \nu_{r}\mid r\in\mathbb{R}\right\} $
be a family of probability Borel measures on $X$ forming an evolution
system of measures with respect to $\left\{ P_{s,t}\mid s<t\right\} $.
Fix $t\in\mathbb{R}$ and assume that there exists $K>0$ such that
\begin{equation}
\sup_{r\leq t}\int_{X}\left|x\right|_{X}^{K}\nu_{r}\left(dx\right)<+\infty.\label{eq: momenti finiti}
\end{equation}
Then $\nu_{t}=\mu_{t}$, where $\mu_{t}$ is defined in $\text{\eqref{def misura inv}}$.
\end{prop}

\begin{proof}
Let $t\in\mathbb{R}$ and $\xi\in X$; for every $s<t$, the computation
for $P_{s,t}e^{i\left\langle \xi,\cdot\right\rangle _{X}}$ in the
proof of Proposition $\text{\ref{prop: invarianza}}$ leads to
\[
\widehat{P_{s,t}^{*}\nu_{s}}\left(\xi\right)=e^{-\frac{1}{2}\left\langle \xi,Q\left(t,s\right)\xi\right\rangle _{X}}\widehat{\nu_{s}}\left(U\left(t,s\right)^{*}\xi\right).
\]
By assumption, $P_{s,t}^{*}\nu_{s}=\nu_{t}$; thus:
\begin{align}
\exists\lim_{s\to-\infty}\widehat{\nu_{s}}\left(U\left(t,s\right)^{*}\xi\right) & =\lim_{s\to-\infty}e^{\frac{1}{2}\left\langle \xi,Q\left(t,s\right)\xi\right\rangle _{X}}\widehat{\nu_{t}}\left(\xi\right)\label{eq: limite esiste}\\
 & =e^{\frac{1}{2}\left\langle \xi,Q\left(t,-\infty\right)\xi\right\rangle _{X}}\widehat{\nu_{t}}\left(\xi\right).\nonumber 
\end{align}

For our thesis to be true, it sufficies that the limit in the left
hand side of the above identity is equal to $1$.

We have for every $s<t$ we have, if $K\in\left(0,1\right]$:
\begin{align*}
\left|\widehat{\nu}_{s}\left(U\left(t,s\right)^{*}\xi\right)-1\right| & \leq\int_{X}\left|e^{i\left\langle U\left(t,s\right)^{*}\xi,y\right\rangle _{X}}-1\right|\nu_{s}\left(dy\right)\\
 & \leq2^{1-K}\left|U\left(t,s\right)^{*}\xi\right|_{X}^{K}\int_{X}\left|y\right|_{X}^{K}\nu_{s}\left(dy\right)\\
 & \leq2^{1-K}\left|U\left(t,s\right)^{*}\xi\right|_{X}^{K}\sup_{r\leq t}\int_{X}\left|y\right|_{X}^{K}\nu_{r}\left(dy\right),
\end{align*}
and, if $K>1$:
\begin{align*}
\left|\widehat{\nu}_{s}\left(U\left(t,s\right)^{*}\xi\right)-1\right| & \leq\int_{X}\left|e^{i\left\langle U\left(t,s\right)^{*}\xi,y\right\rangle _{X}}-1\right|\nu_{s}\left(dy\right)\\
 & \leq\left|U\left(t,s\right)^{*}\xi\right|_{X}\int_{X}\left|y\right|_{X}\nu_{s}\left(dy\right)\\
 & \leq\left|U\left(t,s\right)^{*}\xi\right|_{X}\left(\sup_{r\leq t}\int_{X}\left|y\right|_{X}^{K}\nu_{r}\left(dy\right)\right)^{\frac{1}{K}}.
\end{align*}
In any case $\left|\widehat{\nu}_{s}\left(U\left(t,s\right)^{*}\xi\right)-1\right|\to0$
as $s\to-\infty$, by Assumption $\text{\ref{assu: omega_0 neg}}$
and by $\text{\eqref{eq: momenti finiti}}$. Relation $\text{\eqref{eq: limite esiste}}$
thus implies:
\[
\widehat{\nu_{t}}\left(\xi\right)=e^{-\frac{1}{2}\left\langle \xi,Q\left(t,-\infty\right)\xi\right\rangle _{X}}=\widehat{\mu_{t}}\left(\xi\right),
\]
which implies $\nu_{t}=\mu_{t}$, since $\xi$ is generic.
\end{proof}
Thanks to the ``invariance'' property in Proposition $\text{\ref{prop: invarianza}}$,
the operator $P_{s,t}$ can be extended to $L^{p}\left(X,\mu_{t}\right)$,
for any $p\geq1$. 
\begin{prop}
\label{prop: estensione P_s,t}Take $\left\{ \mu_{t}\mid t\in\mathbb{R}\right\} $
as in $\text{\eqref{def misura inv}}$. For each $s<t$, and $p\in[1,+\infty)$,
the operator $P_{s,t}:C_{b}\left(X\right)\to C_{b}\left(X\right)$
has a bounded extension:
\[
P_{s,t}^{p}:L^{p}\left(X,\mu_{t}\right)\to L^{p}\left(X,\mu_{s}\right).
\]

Analogously the operator $\overrightarrow{P_{s,t}}:C_{b}\left(X;X\right)\to C_{b}\left(X;X\right)$
has a bounded extension:
\[
\overrightarrow{P_{s,t}^{p}}:L^{p}\left(X,\mu_{t};X\right)\to L^{p}\left(X,\mu_{s};X\right).
\]

Both these extended operators are contractions. Relations $\text{\eqref{eq: muinv}}$
and $\eqref{eq: muinv vettoriale}$ also hold for $P_{s,t}=P_{s,t}^{p}$
with $\varphi\in L^{p}\left(X,\mu_{t}\right)$ and $\overrightarrow{P_{s,t}}=\overrightarrow{P_{s,t}^{p}}$
with $\Phi\in L^{p}\left(X,\mu_{t};X\right)$, respectively.
\end{prop}

\begin{proof}
See Appendix \ref{appendice dim}.
\end{proof}
\subsection{Regularization of $L^p$ functions. Compactness}
\leavevmode\label{subsez: reg Lp}

From now on, \emph{we consider a fixed} $p>1$ and we set for simplicity
$P_{s,t}=P_{s,t}^{p}$, for every $s<t$, so that $P_{s,t}$ will
always be considered as an operator from $L^{p}\left(X,\mu_{t}\right)$
to $L^{p}\left(X,\mu_{s}\right)$. An analogous simplified notation
will be used for the vectorial operator $\overrightarrow{P_{s,t}^{p}}:L^{p}\left(X,\mu_{t};X\right)\to L^{p}\left(X,\mu_{s};X\right)$,
which we will denote by $\overrightarrow{P_{s,t}}$.

We are going to show the main result of this section: every extended
operator $P_{s,t}:L^{p}\left(X,\mu_{t}\right)\to L^{p}\left(X,\mu_{s}\right)$
has image contained in $W^{1,p}\left(X,\mu_{s}\right)$ and it is
continuous as an operator from $L^{p}\left(X,\mu_{t}\right)$ to $W^{1,p}\left(X,\mu_{s}\right)$.
Since the latter space is compactly embedded in $L^{p}\left(X,\mu_{s}\right)$,
this result will imply that $P_{s,t}$ is a compact operator.

For the reasons discussed at the beginning of the section, we assume
henceforth that the measures defined by $\text{\eqref{def misura inv}}$
are non-degenerate.
\begin{assumption}
\label{assu: mu_t non degeneri}The measures $\mu_{t}$ defined in
$\text{\eqref{def misura inv}}$ are non-degenerate.
\end{assumption}

\begin{prop}
\label{prop: operatore compatto}Let $-\infty<s<t<+\infty$. The image
of the operator
\[
P_{s,t}:L^{p}\left(X,\mu_{t}\right)\to L^{p}\left(X,\mu_{s}\right)
\]
is contained in $W^{1,p}\left(X,\mu_{s}\right)$ and $P_{s,t}$ is
continuous as an operator with values in $W^{1,p}\left(X,\mu_{s}\right)$.
Consequently, $P_{s,t}$ is compact as an operator with values in
$L^{p}\left(X,\mu_{s}\right)$.

\vspace{1mm}

In particular, the continuity estimate for the gradient has the form:
\begin{equation}
\left\Vert \nabla P_{s,t}f\right\Vert _{L^{p}\left(X,\mu_{s};X\right)}\leq C_{p}\left\Vert \Lambda\left(t,s\right)\right\Vert _{\mathcal{L}\left(X\right)}\left\Vert f\right\Vert _{L^{p}\left(X,\mu_{t}\right)}\quad\forall f\in L^{p}\left(X,\mu_{t}\right),\label{stima nabla Ps,t}
\end{equation}
where $\Lambda\left(t,s\right)$ is the operator defined in Remark
$\text{\ref{rem:def Lambda}}$.

Furthermore, relation $\text{\eqref{eq: formula grad C1}}$ also holds
for $\varphi\in W^{1,p}\left(X,\mu_{t}\right)$ and for $\mu_{s}$-almost
every $x\in X$.
\end{prop}

\begin{proof}
Let $s<t$. By the definition of $W^{1,p}\left(X,\mu_{s}\right)$
given at the beginning of the section, to prove that $P_{s,t}g\in W^{1,p}\left(X,\mu_{s}\right)$
for a $g\in L^{p}\left(X,\mu_{t}\right)$ we need to exhibit a sequence
$\left(h_{n}\right)_{n}\subseteq C_{b}^{1}\left(X\right)$ such that
$h_{n}\to P_{s,t}g$ in $L^{p}\left(X,\mu_{s}\right)$ and such that
$\left(\nabla h_{n}\right)_{n}$ is Cauchy with respect to the $L^{p}\left(X,\mu_{s};X\right)$
norm. We prove that for every sequence of the form $h_{n}=P_{s,t}f_{n}$
such that $\left(f_{n}\right)_{n}\subseteq C_{b}\left(X\right)$ and
$f_{n}\to g$ in $L^{p}\left(X,\mu_{t}\right)$, the sequence $\left(\nabla h_{n}\right)_{n}$
is as desired.

Consider a function $f\in C_{b}\left(X\right)$; by Proposition $\text{\ref{prop: CM formula}}$,
$P_{s,t}f\in C_{b}^{1}\left(X\right)$ and $\nabla P_{s,t}f$ satisfies
$\text{\eqref{eq: grad P_s,t funz C_0}}$. We estimate the $L^{p}\left(X,\mu_{s};X\right)$
norm of $\nabla P_{s,t}f$.

Applying H\"older's inequality in formula $\text{\eqref{eq: grad P_s,t funz C_0}}$,
and setting for simplicity $L^{r}=L^{r}\left(X,\gamma_{t,s}\right)$
for $r=p,p'$, we have for every $x\in X$:
\begin{align*}
\left|\nabla P_{s,t}f\left(x\right)\right|_{X} & =\sup_{\left|h\right|_{X}=1}\left|\left\langle \nabla P_{s,t}f\left(x\right),h\right\rangle _{X}\right|\\
 & \leq\sup_{\left|h\right|_{X}=1}\int_{X}\left|f\left(y+U\left(t,s\right)x\right)\cdot\widehat{U\left(t,s\right)h}\left(y\right)\right|\gamma_{t,s}\left(dy\right)\\
 & \leq\left(\int_{X}\left|f\left(y+U\left(t,s\right)x\right)\right|^{p}\gamma_{t,s}\left(dy\right)\right)^{\frac{1}{p}}\sup_{\left|h\right|_{X}=1}\left\Vert \widehat{U\left(t,s\right)h}\right\Vert _{L^{p'}}\\
 & =\left(P_{s,t}\left|f\right|^{p}\left(x\right)\right)^{\frac{1}{p}}\sup_{\left|h\right|_{X}=1}\left\Vert \widehat{U\left(t,s\right)h}\right\Vert _{L^{p'}}.
\end{align*}
Since any vector $U\left(t,s\right)h$ belongs to the Cameron-Martin
space of the measure $\gamma_{t,s}$, point $iii)$ in Theorem $\text{\ref{Teo CM}}$
implies
\[
\gamma_{t,s}\circ\widehat{U\left(t,s\right)h}^{-1}\sim\mathcal{N}\left(0,\left|\Lambda\left(t,s\right)h\right|_{X}^{2}\right).
\]
Thus, we have:
\begin{align*}
\int_{X}\left|\widehat{U\left(t,s\right)h}\left(y\right)\right|^{p'}\gamma_{t,s}\left(dy\right) & =\int_{\mathbb{R}}\left|\tau\right|^{p'}\left(\gamma_{t,s}\circ\widehat{U\left(t,s\right)h}^{-1}\right)\left(d\tau\right)\\
 & =\int_{\mathbb{R}}\left|\tau\right|^{p'}\mathcal{N}\left(0,\left|\Lambda\left(t,s\right)h\right|_{X}^{2}\right)\left(d\tau\right)\\
 & =\frac{1}{\sqrt{2\pi}\left|\Lambda\left(t,s\right)h\right|_{X}}\int_{\mathbb{R}}\left|\tau\right|^{p'}\exp\left(-\frac{\tau^{2}}{2\left|\Lambda\left(t,s\right)h\right|_{X}^{2}}\right)\left(d\tau\right)\\
 & =\frac{\left|\Lambda\left(t,s\right)h\right|_{X}^{p'}}{\sqrt{2\pi}}\int_{\mathbb{R}}\left|u\right|^{p'}\exp\left(-\frac{u^{2}}{2}\right)\left(du\right).
\end{align*}
Therefore, setting:
\[
C_{p}:=\left(\frac{1}{\sqrt{2\pi}}\int_{\mathbb{R}}\left|u\right|^{p'}\exp\left(-\frac{u^{2}}{2}\right)\left(du\right)\right)^{\frac{1}{p'}},
\]
we obtain $\left\Vert \widehat{U\left(t,s\right)h}\right\Vert _{L^{p'}}\leq C_{p}\left|\Lambda\left(t,s\right)h\right|_{X}$,
and substituting in the initial inequality:
\begin{align}
\left|\nabla P_{s,t}f\left(x\right)\right|_{X} & \leq\left(P_{s,t}\left|f\right|^{p}\left(x\right)\right)^{\frac{1}{p}}\sup_{\left|h\right|_{X}=1}\left\Vert \widehat{U\left(t,s\right)h}\right\Vert _{L^{p'}}\label{eq: stima norma X nabla Ps,t}\\
 & \leq C_{p}\left\Vert \Lambda\left(t,s\right)\right\Vert _{\mathcal{L}\left(X\right)}\left(P_{s,t}\left|f\right|^{p}\left(x\right)\right)^{\frac{1}{p}}.\nonumber 
\end{align}

Thus, integrating with respect to the measure $\mu_{s}$ defined in
$\text{\eqref{def misura inv}}$, and remembering $\text{\eqref{eq: muinv}}$:
\begin{align}
\int_{X}\left|\left[\nabla P_{s,t}f\right]\left(x\right)\right|_{X}^{p}\mu_{s}\left(dx\right) & \leq C_{p}^{p}\left\Vert \Lambda\left(t,s\right)\right\Vert _{\mathcal{L}\left(X\right)}^{p}\int_{X}P_{s,t}\left|f\right|^{p}\left(x\right)\mu_{s}\left(dx\right)\label{eq: stima grad C_b}\\
 & =C_{p}^{p}\left\Vert \Lambda\left(t,s\right)\right\Vert _{\mathcal{L}\left(X\right)}^{p}\int_{X}\left|f\left(x\right)\right|^{p}\mu_{t}\left(dx\right).\nonumber 
\end{align}
Now let $g\in L^{p}\left(X,\mu_{t}\right)$ and $\left(f_{n}\right)_{n}\subseteq C_{b}\left(X\right)$
such that $f_{n}\to g$ in $L^{p}\left(X,\mu_{t}\right)$. By $\text{\eqref{eq: stima grad C_b}}$
we have:
\[
\left\Vert \nabla P_{s,t}f_{n}-\nabla P_{s,t}f_{m}\right\Vert _{L^{p}\left(X,\mu_{s};X\right)}\leq C_{p}\left\Vert \Lambda\left(t,s\right)\right\Vert _{\mathcal{L}\left(X\right)}\left\Vert f_{n}-f_{m}\right\Vert _{L^{p}\left(X,\mu_{s}\right)}\quad\forall n,m\in\mathbb{N}.
\]
Thus the sequence $\left(\nabla P_{s,t}f_{n}\right)_{n}\subseteq L^{p}\left(X,\mu_{s};X\right)$
is Cauchy. Since $P_{s,t}f_{n}\to P_{s,t}g$ in $L^{p}\left(X,\mu_{s}\right)$
- because of the continuity of $P_{s,t}:L^{p}\left(X,\mu_{t}\right)\to L^{p}\left(X,\mu_{s}\right)$
shown in Proposition $\text{\ref{prop: estensione P_s,t}}$ - $\nabla P_{s,t}g$
exists and coincides with the $L^{p}\left(X,\mu_{s};X\right)$-limit
of the sequence $\left(\nabla P_{s,t}f_{n}\right)_{n}$. This proves
that
\[
P_{s,t}\left(L^{p}\left(X,\mu_{t}\right)\right)\subseteq W^{1,p}\left(X,\mu_{s}\right),
\]
as required. Considering relation $\text{\eqref{eq: stima grad C_b}}$
with $f_{n}$ in place of $f$ and then passing to the limit for $n\to\infty$,
we obtain, by definition of $\nabla P_{s,t}g$, that the same estimate
holds with $g$ in place of $f$, which proves formula $\text{\eqref{stima nabla Ps,t}}$
and - consequently - the continuity of $P_{s,t}$ as an operator from
$L^{p}\left(X,\mu_{t}\right)$ to $W^{1,p}\left(X,\mu_{s}\right)$.

The compactness of
\[
P_{s,t}:L^{p}\left(X,\mu_{t}\right)\to L^{p}\left(X,\mu_{s}\right)
\]
 follows, since $W^{1,p}\left(X,\mu_{s}\right)$ is compactly embedded
in $L^{p}\left(X,\mu_{s}\right)$ (see\textcolor{red}{{} }\cite{DAP-ZAB}).

As far as the gradient formula in $\text{\eqref{eq: formula grad C1}}$
for $W^{1,p}$ functions is concerned, the reasoning is similar. Consider
$g\in W^{1,p}\left(X,\mu_{t}\right)$ and take an approximation sequence
$\left(f_{n}\right)_{n}\subseteq C_{b}^{1}\left(X\right)$ for $g$
in $L^{p}\left(X,\mu_{t}\right)$; then $\nabla g$ is by definition
the $L^{p}\left(X,\mu_{t};X\right)$-limit of $\left(\nabla f_{n}\right)_{n}$,
while $\nabla P_{s,t}g$ is the $L^{p}\left(X,\mu_{s};X\right)$-limit
of $\left(\nabla P_{s,t}f_{n}\right)_{n}$ - again, because $P_{s,t}f_{n}\to P_{s,t}g$
in $L^{p}\left(X,\mu_{s}\right)$ . So we can consider $\text{\eqref{eq: formula grad C1}}$
with $\varphi=f_{n}$ and pass to the $L^{p}$-limit for $n\to\infty$,
thus obtaining the validity of the very same formula for $\varphi=g$
and for $\mu_{s}$-almost every $x\in X$, since $\overrightarrow{P_{s,t}}:L^{p}\left(X,\mu_{t};X\right)\to L^{p}\left(X,\mu_{s};X\right)$
is also continuous by Proposition $\text{\ref{prop: estensione P_s,t}}$.
\end{proof}
\begin{rem}
We stress that, as proven above, the following relation hold:
\begin{equation}
\left\Vert \widehat{U\left(t,s\right)h}\right\Vert _{L^{p'}\left(X,\gamma_{t,s}\right)}\leq C_{p}\left|\Lambda\left(t,s\right)h\right|_{X}\quad\forall t>s\ \forall h\in X.\label{eq: norma p' U(t,s) cappuccio}
\end{equation}
\end{rem}

\subsection{The vector-valued first order Sobolev space}
\leavevmode\label{subsez: Sob vett}

The Ornstein-Uhlenbeck operator $\overrightarrow{P_{s,t}}$ of a vector
valued, continuous bounded function has been introduced in $\text{\eqref{eq: def O-U vett formula}}$
and appeared in formula $\text{\eqref{eq: formula grad C1}}$; further,
in Subsection $\text{\ref{subsez: estensione Lp}}$, Proposition $\text{\ref{prop: estensione P_s,t}}$,
such operator has been proven to be extendible to a contraction operator
from $L^{p}\left(X,\mu_{t};X\right)$ to $L^{p}\left(X,\mu_{s};X\right)$.

In this and next subsection we deal with the mirror version, for the
extended $\overrightarrow{P_{s,t}}$, of the regularization properties
established in Proposition $\text{\ref{prop: operatore compatto}}$
relative to $P_{s,t}$.

The main use of these new results for spaces of vector valued functions
will be within the spectral analysis of the periodic case carried
on in the last section - in particular, as part of a second order
regularization property which is used to properly relate the eigenvalues
of the the evolution operator $U\left(t,s\right)$ (the datum) with
the eigenvalues of $P_{s,t}$, the object of our study. This will
be useful in order identify the highest possible rate of convergence
in $\text{\eqref{eq: conv}}$.

Note that while Proposition $\text{\ref{prop: operatore compatto}}$
holds for $p>1$, its vectorial counterpart, i.e. Proposition $\text{\ref{prop: reg vett}}$,
is proven in the case $p\geq2$.

\vspace{2mm}

Denote by $\nabla_{\mathcal{L}}:C_{b}^{1}\left(X;X\right)\to C_{b}\left(X;\mathcal{L}\left(X\right)\right)$
the gradient operator for functions in $C_{b}^{1}\left(X;X\right).$
Namely, $\nabla_{\mathcal{L}}$ is the Frech\'et derivative defined
by the relation:
\[
\Phi\left(x+h\right)-\Phi\left(x\right)-\nabla_{\mathcal{L}}\Phi\left(x\right)h=o_{h\to0}\left(h\right),\quad\forall x\in X,\ \Phi\in C_{b}^{1}\left(X;X\right).
\]

Define, for every $s<t$ and $\mathrm{F}\in C_{b}\left(X;\mathcal{L}\left(X\right)\right)$:
\begin{equation}
\left[\overrightarrow{\overrightarrow{P_{s,t}}}\mathrm{F}\right]\left(x\right):=\int_{X}\mathrm{F}\left(y\right)\mathcal{N}\left(U\left(t,s\right)x,Q\left(t,s\right)\right)\left(dy\right),\label{eq: def P_st operatoriale}
\end{equation}
where the integral is meant in the sense of Bochner.

\vspace{2mm}
\begin{prop}
For every $t>s$, the following inclusions hold:
\begin{align*}
 & \overrightarrow{\overrightarrow{P_{s,t}}}\left(C_{b}\left(X;\mathcal{L}\left(X\right)\right)\right)\subseteq C_{b}\left(X;\mathcal{L}\left(X\right)\right),\\
 & \overrightarrow{P_{s,t}}\left(C_{b}^{1}\left(X;X\right)\right)\subseteq C_{b}^{1}\left(X;X\right),
\end{align*}
and for every $\Phi\in C_{b}^{1}\left(X;X\right)$:
\begin{equation}
\nabla_{\mathcal{L}}\overrightarrow{P_{s,t}}\Phi\left(x\right)=\left[\overrightarrow{\overrightarrow{P_{s,t}}}\nabla_{\mathcal{L}}\Phi\right]\left(x\right)U\left(t,s\right).\label{eq: formula grad vettoriale}
\end{equation}
\end{prop}

\begin{proof}
The first inclusion is a consequence of the Dominated convergence
theorem, while the proof of the other two relations is essentially
the same as in point ii), Proposition $\text{\ref{prop: op evoluzione}}$.
\end{proof}
We also denote by
\[
\mathcal{L}_{2}\left(X\right):=\left\{ T\in\mathcal{L}\left(X\right)\mid\sum_{j=1}^{+\infty}\left|Te_{j}\right|_{X}^{2}<+\infty\right\} 
\]
the set of the linear \emph{Hilbert-Schmidt }operators on $X$. It
is well known that $\mathcal{L}_{2}\left(X\right)$ is a Hilbert space
endowed with the scalar product:
\[
\left\langle T,S\right\rangle _{\mathcal{L}_{2}\left(X\right)}:=\sum_{j=1}^{+\infty}\left\langle Te_{j},Se_{j}\right\rangle _{X}.
\]

The following are well known classes of special scalar and vector-valued
regular functions. We set
\[
\mathcal{F}C_{b}^{1}\left(X\right):=\left\{ F\left(f_{1},\dots,f_{N}\right)\mid N\in\mathbb{N},\,f_{1},\dots,f_{N}\in X^{*},\,F\in C_{b}^{1}\left(\mathbb{R}^{N}\right)\right\} 
\]
for the class of the \emph{scalar cylindrical functions}, and
\[
\mathcal{F}C_{b}^{1}\left(X;X\right)=\left\{ \sum_{j=1}^{N}v_{j}y_{j}\mid N\in\mathbb{N},\,v_{j}\in\mathcal{F}C_{b}^{1}\left(X\right),\,y_{j}\in X\ \forall j=1,\dots,N\right\} 
\]
for the class of the \emph{vector-valued cylindrical functions}.

\vspace{2mm}

So far the gradient $\nabla_{\mathcal{L}}\Phi\left(x\right)$ of a
$C_{b}^{1}\left(X;X\right)$ function $\Phi$ at a point $x\in X$
is just a bounded linear operator on $X$, having been defined as
a Frech\'et derivative. For our purposes, we need to deal with gradients
that belong to the Hilbert space $\mathcal{L}_{2}\left(X\right)$.
The following basic Lemma ensures in particular that $\nabla_{\mathcal{L}}\Phi\left(x\right)\in\mathcal{L}_{2}\left(X\right)$
if $\Phi$ is cylindrical, for every $x\in X$.
\begin{lem}
\label{lem: norma H-S conv unif}For every $\Phi\in\mathcal{F}C_{b}^{1}\left(X;X\right)$,
the series:
\[
\sum_{j=1}^{+\infty}\left|\nabla_{\mathcal{L}}\Phi\left(x\right)e_{j}\right|_{X}^{2}
\]
converges uniformly with respect to $x\in X$.
\end{lem}

\begin{proof}
First, fix $v\in\mathcal{F}C_{b}^{1}\left(X\right)$ and assume $v=F\left(f_{1},\dots,f_{N}\right)$
for some $F\in C_{b}^{1}\left(\mathbb{R}^{N}\right)$ and $f_{j}\in X^{*}$,
$j=1,\dots,N$. Observe that $\nabla f_{j}\left(x_{0}\right)=f_{j}$
for every $j$ and $x_{0}\in X$, and also $\nabla v\left(x_{0}\right)\in X^{*}$.

Using the compact notation $\nabla\left(f_{1},\dots,f_{N}\right)\left(x_{0}\right)z$
for the $N$-dimensional vector $\left(\nabla f_{1}\left(x_{0}\right)z,\dots,\nabla f_{N}\left(x_{0}\right)z\right)=\left(f_{1}z,\dots,f_{N}z\right)$,
we have for every $z\in X$:
\begin{align*}
\left\langle \nabla v\left(x_{0}\right),z\right\rangle _{X} & =\left\langle \nabla F\left(f_{1}x_{0},\dots,f_{N}x_{0}\right),\nabla\left(f_{1},\dots,f_{N}\right)\left(x_{0}\right)z\right\rangle _{\mathbb{R}^{N}}\\
 & =\sum_{k=1}^{N}\partial_{k}F\left(f_{1}x_{0},\dots,f_{N}x_{0}\right)f_{k}z.
\end{align*}
Set $f_{j}:=\left\langle v_{j},\cdot\right\rangle _{X}$ and $\partial_{j}F=\partial_{j}F\left(f_{1}x_{0},\dots,f_{N}x_{0}\right)$
for $j=1,\dots,N$. It follows from the above identity that, for every
$R\in\mathbb{N}$:
\begin{align}
\sum_{j\geq R}\left\langle \nabla v\left(x_{0}\right),e_{j}\right\rangle _{X}^{2}\leq\, & \sum_{j\geq R}\left\{ \left\Vert \nabla F\right\Vert _{L^{\infty}\left(\mathbb{R}^{N}\right)}^{2}\sum_{k=1}^{N}\left\langle v_{k},e_{j}\right\rangle _{X}^{2}\right.\label{eq: stima norma L_2 grad cil scal}\\
 & \left.+2\sum_{h,k=1}^{N}\partial_{k}F\partial_{h}F\left\langle v_{h},e_{j}\right\rangle _{X}\left\langle v_{k},e_{j}\right\rangle _{X}\right\} \nonumber \\
\leq\, & C_{v}\left(R\right),\nonumber 
\end{align}
where
\[
C_{v}\left(R\right):=\left\Vert \nabla F\right\Vert _{L^{\infty}\left(\mathbb{R}^{N}\right)}^{2}\left\{ \sum_{k=1}^{N}\sum_{j\geq R}\left\langle v_{k},e_{j}\right\rangle _{X}^{2}+2\sum_{h,k=1}^{N}\left|\sum_{j\geq R}\left\langle v_{h},e_{j}\right\rangle _{X}\left\langle v_{k},e_{j}\right\rangle _{X}\right|\right\} .
\]
Clearly the quantity $C_{v}\left(R\right)$ is finite for every $R$,
it vanishes for $R\to\infty$ and does not depend on $x_{0}$.

In a similar fashion we see that, if $v,w\in\mathcal{F}C_{b}^{1}\left(X\right)$,
then
\begin{equation}
\left|\sum_{j\geq R}\left\langle \nabla v\left(x_{0}\right),e_{j}\right\rangle _{X}\left\langle \nabla w\left(x_{0}\right),e_{j}\right\rangle _{X}\right|\leq C_{v,w}\left(R\right)\to0\text{ as }R\to\infty,\label{eq: stima prod scal L_2 grad cil scal}
\end{equation}
for some quantity $C_{v,w}\left(R\right)$ not depending on $x_{0}$.

Now take $\Phi\in\mathcal{F}C_{b}^{1}\left(X;X\right)$ with $\Phi=\sum_{i=1}^{M}v_{i}y_{i}$
for some $y_{i}\in X$ and $v_{i}\in\mathcal{F}C_{b}^{1}\left(X\right)$.

Recalling that $\nabla_{\mathcal{L}}\left[v\left(\cdot\right)y\right]\left(x_{0}\right)e_{j}=\left\langle \nabla v\left(x_{0}\right),e_{j}\right\rangle _{X}y$
for every $v\in C_{b}^{1}\left(X\right)$ and $y\in X$, it follows
from $\text{\eqref{eq: stima norma L_2 grad cil scal}}$ and $\text{\eqref{eq: stima prod scal L_2 grad cil scal}}$
that:
\begin{align*}
\sum_{j\geq R}\left|\nabla_{\mathcal{L}}\Phi\left(x_{0}\right)e_{j}\right|_{X}^{2}=\, & \sum_{j\geq R}\left|\sum_{i=1}^{M}\nabla_{\mathcal{L}}\left[v_{i}\left(\cdot\right)y_{i}\right]\left(x_{0}\right)e_{j}\right|_{X}^{2}\\
=\, & \sum_{j\geq R}\left|\sum_{i=1}^{M}\left\langle \nabla v_{i}\left(x_{0}\right),e_{j}\right\rangle _{X}y_{i}\right|_{X}^{2}\\
=\, & \sum_{j\geq R}\left\{ \sum_{i=1}^{M}\left|y_{i}\right|_{X}^{2}\left\langle \nabla v_{i}\left(x_{0}\right),e_{j}\right\rangle _{X}^{2}\right.\\
 & \left.+2\sum_{h,k=1}^{M}\left\langle \nabla v_{h}\left(x_{0}\right),e_{j}\right\rangle _{X}\left\langle \nabla v_{k}\left(x_{0}\right),e_{j}\right\rangle _{X}\left\langle y_{h},y_{k}\right\rangle _{X}\right\} \\
\leq & \max_{i=1,\dots,M}\left|y_{i}\right|_{X}^{2}\sum_{i=1}^{M}\sum_{j\geq R}\left\langle \nabla v_{i}\left(x_{0}\right),e_{j}\right\rangle _{X}^{2}\\
 & +2\max_{h,k=1,\dots,M}\left|\left\langle y_{h},y_{k}\right\rangle _{X}\right|\sum_{h,k=1}^{M}\left|\sum_{j\geq R}\left\langle \nabla v_{h}\left(x_{0}\right),e_{j}\right\rangle _{X}\left\langle \nabla v_{k}\left(x_{0}\right),e_{j}\right\rangle _{X}\right|\\
\leq & \max_{i=1,\dots,M}\left|y_{i}\right|_{X}^{2}\sum_{i=1}^{M}C_{v_{i}}\left(R\right)+2\max_{h,k=1,\dots,M}\left|\left\langle y_{h},y_{k}\right\rangle _{X}\right|\sum_{h,k=1}^{M}C_{v_{h},v_{k}}\left(R\right).
\end{align*}
Hence
\[
\lim_{R\to\infty}\sup_{x_{0}\in X}\sum_{j\geq R}\left|\nabla_{\mathcal{L}}\Phi\left(x_{0}\right)e_{j}\right|_{X}^{2}=0,
\]
and the thesis is proved.
\end{proof}
It follows in particular from Lemma $\ref{lem: norma H-S conv unif}$
that:
\[
\nabla_{\mathcal{L}}\left(\mathcal{F}C_{b}^{1}\left(X;X\right)\right)\subseteq C_{b}\left(X;\mathcal{L}_{2}\left(X\right)\right).
\]
Further, for every $\Phi\in\mathcal{F}C_{b}^{1}\left(X;X\right)$,
there exists $C_{\Phi}>0$ such that:
\[
\left\Vert \nabla_{\mathcal{L}}\Phi\left(x\right)\right\Vert _{\mathcal{L}_{2}\left(X\right)}\leq C_{\Phi}\quad\forall x\in X.
\]

These considerations show that cylindrical functions are a good basis
to define the vector-valued Sobolev space of order one.

Observe that, for every non-degenerate Gaussian measure $\nu$ on
$X$, $\mathcal{F}C_{b}^{1}\left(X;X\right)$ is dense in $L^{p}\left(X,\nu;X\right)$.
Further, denoting $\nabla_{\mathcal{L}_{2}}:=\nabla_{\mathcal{L}\mid\mathcal{F}C_{b}^{1}\left(X;X\right)}$
we have that
\[
\nabla_{\mathcal{L}_{2}}:\mathcal{F}C_{b}^{1}\left(X;X\right)\subseteq L^{p}\left(X,\nu;X\right)\to L^{p}\left(X,\nu;\mathcal{L}_{2}\left(X\right)\right)
\]
is closable. See again \cite{DAP-ZAB} for a proof of this basic fact.

We denote by $W^{1,p}\left(X,\nu;X\right)$ the domain of the closure
of $\nabla_{\mathcal{L}_{2}}$ in $L^{p}\left(X,\nu;X\right)\times L^{p}\left(X,\nu;\mathcal{L}_{2}\left(X\right)\right)$,
and we use again the symbol $\nabla_{\mathcal{L}_{2}}$ to identify
the closure operator. Thus we may write
\[
\nabla_{\mathcal{L}_{2}}:W^{1,p}\left(X,\nu;X\right)\to L^{p}\left(X,\nu;\mathcal{L}_{2}\left(X\right)\right).
\]
Due to this definition of the first order Sobolev space of vector-valued
functions, a function $\Phi\in L^{p}\left(X,\nu;X\right)$ belongs
to $W^{1,p}\left(X,\nu;X\right)$ if and only if there exists a sequence
$\left(\Phi_{n}\right)_{n}\subseteq\mathcal{F}C_{b}^{1}\left(X;X\right)$
such that $\Phi_{n}\to\Phi$ in $L^{p}\left(X,\nu;X\right)$ and $\nabla_{\mathcal{L}_{2}}\Phi_{n}$
converges in $L^{p}\left(X,\nu;\mathcal{L}_{2}\left(X\right)\right)$.
In this case, we have for every $x\in X$:
\begin{equation}
\nabla_{\mathcal{L}_{2}}\Phi\left(x\right)=\lim_{n\to\infty}\nabla_{\mathcal{L}_{2}}\Phi_{n}\left(x\right)\quad\text{in }L^{p}\left(X,\nu;\mathcal{L}_{2}\left(X\right)\right),\label{eq: def nabla per W1p vett}
\end{equation}
regardless of the choice of $\left(\Phi_{n}\right)_{n}$.

\vspace{2mm}

From now on, the symbol $\nabla_{\mathcal{L}_{2}}$ will refer to
the Frech\'et derivative of a regular cylindrical function as well
as to the weak derivative (in the sense of $\text{\eqref{eq: def nabla per W1p vett}}$)
of a function belonging to $W^{1,p}\left(X,\nu;X\right)$ for some
non-degenerate Gaussian measure $\nu$. In particular, if $\Phi\in W^{1,p}\left(X,\nu;X\right)$
then $\nabla_{\mathcal{L}_{2}}\Phi\left(x\right)\in\mathcal{L}_{2}\left(X\right)$
for $\nu$-almost every $x\in X$.

\subsection{Regularization of $X$-valued $L^p$ functions}
\leavevmode\label{subsez: reg Lp vett}

Our present purpose is that of proving the vectorial counterpart of
Proposition $\text{\ref{prop: operatore compatto}}$.

We introduce the following space:
\[
UC_{b}\left(X;\mathcal{L}_{2}\left(X\right)\right):=\left\{ F\in C_{b}\left(X;\mathcal{L}_{2}\left(X\right)\right)\mid\sum_{j=1}^{+\infty}\left|F\left(x\right)e_{j}\right|_{X}^{2}<+\infty\,\text{uniformly in }x\in X\right\} .
\]

\begin{rem}
\label{rem: Pst operatoriale su UCb}For every $F\in UC_{b}\left(X;\mathcal{L}_{2}\left(X\right)\right)$,
and $t>s$, it is immediate to verify that $\overrightarrow{\overrightarrow{P_{s,t}}}F\in C_{b}\left(X;\mathcal{L}_{2}\left(X\right)\right)$
and that the coordinates of $\overrightarrow{\overrightarrow{P_{s,t}}}F\left(x\right)$
in $\mathcal{L}_{2}\left(X\right)$ are of the form:
\[
\left\langle \overrightarrow{\overrightarrow{P_{s,t}}}F\left(x\right)e_{i},e_{j}\right\rangle _{X}=P_{s,t}\left\langle F\left(\cdot\right)e_{i},e_{j}\right\rangle _{X}\left(x\right).
\]
\end{rem}

\begin{lem}
\label{lem: grad e Ps,t di cilindriche} Fix $t>s$. Then:

\vspace{2mm}

i)

\vspace{-11mm}
\[
\overrightarrow{P_{s,t}}\left(\mathcal{F}C_{b}^{1}\left(X;X\right)\right)\subseteq\mathcal{F}C_{b}^{1}\left(X;X\right);
\]

ii) for every $\Phi\in\mathcal{F}C_{b}^{1}\left(X;X\right)$, $x_{0}\in X$:
\[
\left\Vert \nabla_{\mathcal{L}_{2}}\overrightarrow{P_{s,t}}\Phi\left(x_{0}\right)\right\Vert _{\mathcal{L}_{2}\left(X\right)}^{2}\leq\left\Vert U\left(t,s\right)\right\Vert _{\mathcal{L}\left(X\right)}^{2}P_{s,t}\left(\left\Vert \nabla_{\mathcal{L}_{2}}\Phi\left(\cdot\right)\right\Vert _{\mathcal{L}_{2}\left(X\right)}^{2}\right)\left(x_{0}\right),
\]
and in particular
\[
\left\Vert \nabla_{\mathcal{L}_{2}}\overrightarrow{P_{s,t}}\Phi\left(x_{0}\right)\right\Vert _{\mathcal{L}_{2}\left(X\right)}\leq\left\Vert U\left(t,s\right)\right\Vert _{\mathcal{L}\left(X\right)}C_{\Phi}.
\]
\end{lem}

\begin{proof}
Fix $t>s$, $v=F\left(\left\langle w_{1},\cdot\right\rangle _{X},\dots,\left\langle w_{N},\cdot\right\rangle _{X}\right)$
with $F\in C_{b}^{1}\left(\mathbb{R}^{N}\right)$, $w_{j}\in X$ and
let $y\in X$. For every $x\in X$:
\begin{align*}
\left(\overrightarrow{P_{s,t}}v\left(\cdot\right)y\right)\left(x\right) & =\left(P_{s,t}v\right)\left(x\right)y\\
 & =\int_{X}F\left(\left\langle w_{1},z\right\rangle _{X}+f_{1}x,\dots,\left\langle w_{N},z\right\rangle _{X}+f_{N}x\right)\gamma_{t,s}\left(dz\right)y,
\end{align*}

having set $f_{j}:=\left\langle U\left(t,s\right)^{*}w_{j},\cdot\right\rangle _{X}$,
$j=1,\dots,N$. Note that the function $\widetilde{v}:\mathbb{R}^{N}\to\mathbb{R}$
defined by
\[
\widetilde{v}\left(\xi\right):=\int_{X}F\left(\left\langle w_{1},z\right\rangle _{X}+\xi_{1},\dots,\left\langle w_{N},z\right\rangle _{X}+\xi_{N}\right)\gamma_{t,s}\left(dz\right)
\]
belongs to $C_{b}^{1}\left(\mathbb{R}^{N}\right)$, with
\[
\nabla\widetilde{v}\left(\xi\right)=\int_{X}\left(\nabla F\right)\left(\left\langle w_{1},z\right\rangle _{X}+\xi_{1},\dots,\left\langle w_{N},z\right\rangle _{X}+\xi_{N}\right)\gamma_{t,s}\left(dz\right);
\]
thus $\overrightarrow{P_{s,t}}v\left(\cdot\right)y=\widetilde{v}\left(f_{1},\dots,f_{N}\right)$
is cylindrical. Passing to linear combinations $\sum_{i=1}^{M}v_{i}\left(\cdot\right)y_{i}$,
the inclusion in i) follows.

Now take $\Phi\in\mathcal{F}C_{b}^{1}\left(X;X\right)$ and $x_{0}\in X$.
Note that $\nabla_{\mathcal{L}_{2}}\Phi\in UC_{b}\left(X;\mathcal{L}_{2}\left(X\right)\right)$
by Lemma $\text{\ref{lem: norma H-S conv unif}}$; thus, as observed
in Remark $\text{\ref{rem: Pst operatoriale su UCb}}$, $\overrightarrow{\overrightarrow{P_{s,t}}}\nabla_{\mathcal{L}_{2}}\Phi\in C_{b}\left(X;\mathcal{L}_{2}\left(X\right)\right)$.
Hence, by point i), we can pass to the Hilbert-Schmidt norms in $\text{\eqref{eq: formula grad vettoriale}}$
with our $\Phi$; taking into account the definition in $\text{\eqref{eq: def P_st operatoriale}}$
as well we obtain:
\begin{align*}
\left\Vert \nabla_{\mathcal{L}_{2}}\overrightarrow{P_{s,t}}\Phi\left(x_{0}\right)\right\Vert _{\mathcal{L}_{2}\left(X\right)}^{2} & \leq\left\Vert U\left(t,s\right)\right\Vert _{\mathcal{L}\left(X\right)}^{2}\left\Vert \left[\overrightarrow{\overrightarrow{P_{s,t}}}\nabla_{\mathcal{L}_{2}}\Phi\right]\left(x_{0}\right)\right\Vert _{\mathcal{L}_{2}\left(X\right)}^{2}\\
 & \leq\left\Vert U\left(t,s\right)\right\Vert _{\mathcal{L}\left(X\right)}^{2}\sum_{j=1}^{\infty}\int_{X}\left|\nabla_{\mathcal{L}_{2}}\Phi\left(y\right)e_{j}\right|_{X}^{2}\mathcal{N}\left(U\left(t,s\right)x_{0},Q\left(t,s\right)\right)\left(dy\right)\\
 & =\left\Vert U\left(t,s\right)\right\Vert _{\mathcal{L}\left(X\right)}^{2}\int_{X}\left\Vert \nabla_{\mathcal{L}_{2}}\Phi\left(y\right)\right\Vert _{\mathcal{L}_{2}\left(X\right)}^{2}\mathcal{N}\left(U\left(t,s\right)x_{0},Q\left(t,s\right)\right)\left(dy\right)\\
 & =\left\Vert U\left(t,s\right)\right\Vert _{\mathcal{L}\left(X\right)}^{2}P_{s,t}\left(\left\Vert \nabla_{\mathcal{L}_{2}}\Phi\left(\cdot\right)\right\Vert _{\mathcal{L}_{2}\left(X\right)}^{2}\right)\left(x_{0}\right).
\end{align*}
This concludes the proof.
\end{proof}
We note that Remark $\text{\ref{rem: Pst operatoriale su UCb}}$ implies,
using $\text{\eqref{eq: muinv}}$:
\begin{align*}
\int_{X}\overrightarrow{\overrightarrow{P_{s,t}}}F\left(x\right)\mu_{s}\left(dx\right) & =\int_{X}F\left(x\right)\mu_{t}\left(dx\right),
\end{align*}
even though the knowledge of such identity is not necessary to our
purposes.

The very relation $\text{\eqref{eq: muinv}}$ also implies:

\begin{align}
\left\Vert \overrightarrow{\overrightarrow{P_{s,t}}}F\right\Vert _{L^{p}\left(X,\mu_{s};\mathcal{L}_{2}\left(X\right)\right)} & \leq\left\Vert F\right\Vert _{L^{p}\left(X,\mu_{t};\mathcal{L}_{2}\left(X\right)\right)}\quad\forall F\in UC_{b}\left(X;\mathcal{L}_{2}\left(X\right)\right).\label{eq: Ps,t operatoriale contraz}
\end{align}
Even though by the definition in $\text{\eqref{eq: def P_st operatoriale}}$
the operator $\overrightarrow{\overrightarrow{P_{s,t}}}$ acts on
(operator-valued) bounded continuous functions, thanks to relation
$\text{\eqref{eq: Ps,t operatoriale contraz}}$ we can define $\overrightarrow{\overrightarrow{P_{s,t}}}\nabla_{\mathcal{L}_{2}}\Phi$
for a function $\Phi\in W^{1,p}\left(X,\mu_{t};X\right)$ in a standard
way.

Indeed, assume that $\left(\Phi_{n}\right)_{n}\subseteq\mathcal{F}C_{b}^{1}\left(X;X\right)$
is such that $\Phi_{n}\to\Phi$ in $L^{p}\left(X,\mu_{t};X\right)$
and $\nabla_{\mathcal{L}_{2}}\Phi_{n}\to\nabla_{\mathcal{L}_{2}}\Phi$
in $L^{p}\left(X,\mu_{t};\mathcal{L}_{2}\left(X\right)\right)$; by
Lemma $\text{\ref{lem: norma H-S conv unif}}$, for every $n\in\mathbb{N}$,
$\nabla_{\mathcal{L}_{2}}\Phi_{n}\in UC_{b}\left(X;\mathcal{L}_{2}\left(X\right)\right)$,
so that, as observed, $\overrightarrow{\overrightarrow{P_{s,t}}}\nabla_{\mathcal{L}_{2}}\Phi_{n}\in C_{b}\left(X;\mathcal{L}_{2}\left(X\right)\right)$.
By $\text{\eqref{eq: Ps,t operatoriale contraz}}$, the latter sequence
is Cauchy in $L^{p}\left(X,\mu_{s};\mathcal{L}_{2}\left(X\right)\right)$,
thus we may set:
\begin{equation}
\overrightarrow{\overrightarrow{P_{s,t}}}\nabla_{\mathcal{L}_{2}}\Phi:=\lim_{n\to\infty}\overrightarrow{\overrightarrow{P_{s,t}}}\nabla_{\mathcal{L}_{2}}\Phi_{n}\quad\text{in }L^{p}\left(X,\mu_{s};\mathcal{L}_{2}\left(X\right)\right),\label{eq: def Ps,t op nabla Phi Lp}
\end{equation}
regardless of the choice of $\left(\Phi_{n}\right)_{n}$.

\vspace{1mm}
\begin{rem}
\label{rem: su Pst operatoriale}i) The above construction allows
to extend relation $\text{\eqref{eq: formula grad vettoriale}}$ to
any $\Phi\in W^{1,p}\left(X,\mu_{t};X\right)$, for $\mu_{s}$-almost
every $x\in X$ and with $\nabla_{\mathcal{L}}=\nabla_{\mathcal{L}_{2}}$.
Indeed, taking $\left(\Phi_{n}\right)_{n}$ as above, we have:
\[
\left[\nabla_{\mathcal{L}_{2}}\left(\overrightarrow{P_{s,t}}\Phi_{n}\right)\right]\left(\cdot\right)=\left[\overrightarrow{\overrightarrow{P_{s,t}}}\left(\nabla_{\mathcal{L}_{2}}\Phi_{n}\right)\right]\left(\cdot\right)U\left(t,s\right)
\]
pointwise in $X$, and both quantities converge in $L^{p}\left(X,\mu_{s};\mathcal{L}_{2}\left(X\right)\right)$:
the left hand member converges to $\nabla_{\mathcal{L}_{2}}\left(\overrightarrow{P_{s,t}}\Phi\right)$
by definition of $\nabla_{\mathcal{L}_{2}}$ since $\left(\overrightarrow{P_{s,t}}\Phi_{n}\right)_{n}$
is a sequence of \emph{cylindrical} functions approximating $\overrightarrow{P_{s,t}}\Phi$
in $L^{p}\left(X,\mu_{s};X\right)$, while the right hand member converges
to $\overrightarrow{\overrightarrow{P_{s,t}}}\left(\nabla_{\mathcal{L}_{2}}\Phi\right)$
by the definition in $\text{\eqref{eq: def Ps,t op nabla Phi Lp}}$.

\vspace{1mm}

ii) By a similar argument, one sees that also relation $\text{\eqref{eq: Ps,t operatoriale contraz}}$
can be extended to $F=\nabla_{\mathcal{L}_{2}}\Phi$ for some $\Phi\in W^{1,p}\left(X,\mu_{t};X\right)$.

iii) Relation $\eqref{eq: formula grad vettoriale}$ applied to a
$\Phi\in\mathcal{F}C_{b}^{1}\left(X;X\right)$ implies that the function
$\overrightarrow{\overrightarrow{P_{s,t}}}\nabla_{\mathcal{L}_{2}}\Phi$
actually belongs to $UC_{b}\left(X;\mathcal{L}_{2}\left(X\right)\right)$
and not merely to $C_{b}\left(X;\mathcal{L}_{2}\left(X\right)\right)$,
thanks to Lemmas $\text{\ref{lem: norma H-S conv unif}}$ and $\text{\ref{lem: grad e Ps,t di cilindriche}}$.
This property could be also proved directly based on the particular
form of $\nabla_{\mathcal{L}_{2}}\Phi\in UC_{b}\left(X;\mathcal{L}_{2}\left(X\right)\right)$.
\end{rem}

\begin{prop}
\label{prop: reg vett} Let $t>s$ and $\overrightarrow{P_{s,t}}:L^{p}\left(X,\mu_{t};X\right)\to L^{p}\left(X,\mu_{s};X\right)$
the extended operator obtained in Proposition $\text{\ref{prop: estensione P_s,t}}$.
We have:
\[
\overrightarrow{P_{s,t}}\left(L^{p}\left(X,\mu_{t};X\right)\right)\subseteq W^{1,p}\left(X,\mu_{s};X\right),
\]
and there exists $\bar{C}_{p}>0$ such that, for every $\Phi\in L^{p}\left(X,\mu_{t};X\right)$:
\begin{equation}
\left\Vert \nabla_{\mathcal{L}_{2}}\overrightarrow{P_{s,t}}\Phi\right\Vert _{L^{p}\left(X,\mu_{s};\mathcal{L}_{2}\left(X\right)\right)}\leq\bar{C}_{p}\left\Vert \Lambda\left(t,s\right)\right\Vert _{\mathcal{L}\left(X\right)}\left\Vert \Phi\right\Vert _{L^{p}\left(X,\mu_{t};X\right)}.\label{stima nabla Ps,t vett}
\end{equation}
\end{prop}

\begin{proof}
First we prove $\text{\eqref{stima nabla Ps,t vett}}$ for a function
$\Phi\in\mathcal{F}C_{b}^{1}\left(X;X\right)$, say $\Phi=\sum_{i=1}^{M}v_{i}\left(\cdot\right)y_{i}$
for some $v_{i}\in\mathcal{F}C_{b}^{1}\left(X\right)$ and $y_{i}\in X$.
We may assume that $y_{1},\dots,y_{M}$ are unit pairwise orthogonal
vectors; indeed if they are not, denoting by $\left\{ \widetilde{y_{1}},\dots,\widetilde{y_{M}}\right\} $
the system obtained by applying the Gram-Schmidt orthonormalization
procedure to the $y_{i}$'s, we may write $\Phi=\sum_{j=1}^{M}\sum_{i=1}^{M}\alpha_{ij}v_{i}\left(\cdot\right)\widetilde{y_{j}}$
and $\sum_{i=1}^{M}\alpha_{ij}v_{i}\in\mathcal{F}C_{b}^{1}\left(X\right)$
for $j=1,\dots,M$. Indeed, for our purposes it is actually enough
that the $y_{1,}\dots,y_{M}$ are pairwise orthogonal, which is what
we assume.

We show that there exists a positive constant $\bar{C}_{p}^{p}$ such
that for every $x_{0}\in X$:
\begin{equation}
\left\Vert \nabla_{\mathcal{L}_{2}}\overrightarrow{P_{s,t}}\Phi\left(x_{0}\right)\right\Vert _{\mathcal{L}_{2}\left(X\right)}^{p}\leq\bar{C}_{p}^{p}\left\Vert \Lambda\left(t,s\right)\right\Vert _{\mathcal{L}\left(X\right)}^{p}P_{s,t}\left(\left|\Phi\left(\cdot\right)\right|_{X}^{p}\right)\left(x_{0}\right).\label{eq: norma H-s Nabla Ps,t alla p}
\end{equation}
Note that the left hand member of the above inequality makes sense
because of point i) in Lemma $\text{\ref{lem: grad e Ps,t di cilindriche}}$,
which ensures that $\overrightarrow{P_{s,t}}\Phi$ belongs to the
domain of the operator $\nabla_{\mathcal{L}_{2}}$. Fix $x_{0}\in X$;
using to the orthogonality assumption we obtain:
\begin{align}
\left\Vert \nabla_{\mathcal{L}_{2}}\overrightarrow{P_{s,t}}\Phi\left(x_{0}\right)\right\Vert _{\mathcal{L}_{2}\left(X\right)}^{2} & =\sum_{j=1}^{+\infty}\left|\nabla_{\mathcal{L}_{2}}\overrightarrow{P_{s,t}}\sum_{i=1}^{M}v_{i}\left(\cdot\right)y_{i}\left(x_{0}\right)e_{j}\right|_{X}^{2}\nonumber \\
 & =\sum_{j=1}^{+\infty}\left|\sum_{i=1}^{M}\nabla_{\mathcal{L}}\left[P_{s,t}v_{i}\left(\cdot\right)y_{i}\right]\left(x_{0}\right)e_{j}\right|_{X}^{2}\nonumber \\
 & =\sum_{j=1}^{+\infty}\left|\sum_{i=1}^{M}\left\langle \nabla\left(P_{s,t}v_{i}\right)\left(x_{0}\right),e_{j}\right\rangle _{X}y_{i}\right|_{X}^{2}\nonumber \\
 & =\sum_{j=1}^{+\infty}\sum_{i=1}^{M}\left\langle \nabla\left(P_{s,t}v_{i}\right)\left(x_{0}\right),e_{j}\right\rangle _{X}^{2}\left|y_{i}\right|_{X}^{2}\nonumber \\
 & =\sum_{i=1}^{M}\left|\nabla\left(P_{s,t}v_{i}\right)\left(x_{0}\right)\right|_{X}^{2}\left|y_{i}\right|_{X}^{2}.\label{eq: norma H-S NablaPs,t}
\end{align}
Further, we know from the proof of Proposition $\text{\ref{prop: operatore compatto}}$,
and in particular from relation $\text{\eqref{eq: stima norma X nabla Ps,t}}$
that, for every $q>1$ there exists $C_{q}>0$ such that:
\begin{equation}
\left|\nabla\left(P_{s,t}v_{i}\right)\left(x_{0}\right)\right|_{X}^{2}\leq C_{q}^{2}\left\Vert \Lambda\left(t,s\right)\right\Vert _{\mathcal{L}\left(X\right)}^{2}\left(P_{s,t}\left|v_{i}\right|^{q}\left(x_{0}\right)\right)^{\frac{2}{q}}.\label{eq: stima basica deriv direz Ps,t}
\end{equation}

We consider the above estimate with different values of $q$ depending
on whether $p\in\left(1,2\right)$ or $p\geq2$.

If $p\in\left(1,2\right)$, we take $q=p$ in $\text{\eqref{eq: stima basica deriv direz Ps,t}}$.
Setting $\alpha:=\frac{2}{p}>1$, $\alpha'=\alpha/\left(\alpha-1\right)$
and applying the Riesz representation theorem to the $\alpha$ norm
$\left|\cdot\right|_{M,\alpha}$ in $\mathbb{R}^{M}$, we obtain from
$\text{\eqref{eq: norma H-S NablaPs,t}}$ and $\text{\eqref{eq: stima basica deriv direz Ps,t}}$
:
\begin{align*}
\left\Vert \nabla_{\mathcal{L}_{2}}\overrightarrow{P_{s,t}}\Phi\left(x_{0}\right)\right\Vert _{\mathcal{L}_{2}\left(X\right)}^{p} & \leq C_{p}^{p}\left\Vert \Lambda\left(t,s\right)\right\Vert _{\mathcal{L}\left(X\right)}^{p}\left[\sum_{i=1}^{M}\left(P_{s,t}\left|v_{i}\left(\cdot\right)\right|^{p}\left(x_{0}\right)\left|y_{i}\right|_{X}^{p}\right)^{\alpha}\right]^{\frac{1}{\alpha}}\\
 & =C_{p}^{p}\left\Vert \Lambda\left(t,s\right)\right\Vert _{\mathcal{L}\left(X\right)}^{p}\sup_{\left|z\right|_{M,\alpha'}=1}P_{s,t}\left(\sum_{i=1}^{M}z_{i}\left|y_{i}\right|_{X}^{p}\left|v_{i}\left(\cdot\right)\right|^{p}\right)\left(x_{0}\right)\\
 & \leq C_{p}^{p}\left\Vert \Lambda\left(t,s\right)\right\Vert _{\mathcal{L}\left(X\right)}^{p}P_{s,t}\left(\left[\sum_{i=1}^{M}\left|y_{i}\right|_{X}^{p\alpha}\left|v_{i}\left(\cdot\right)\right|^{p\alpha}\right]^{\frac{1}{\alpha}}\right)\left(x_{0}\right)\\
 & =C_{p}^{p}\left\Vert \Lambda\left(t,s\right)\right\Vert _{\mathcal{L}\left(X\right)}^{p}P_{s,t}\left(\left|\Phi\left(\cdot\right)\right|_{X}^{p}\right)\left(x_{0}\right),
\end{align*}
where we used again the orthogonality of the $y_{i}$'s.

In the case $p\geq2$, it is sufficient to consider $\text{\eqref{eq: stima basica deriv direz Ps,t}}$
with any $q_{0}\in(1,2]$ - so that the function $x\mapsto x^{\frac{2}{q_{0}}}$
is convex:
\begin{align*}
\left\Vert \nabla_{\mathcal{L}_{2}}\overrightarrow{P_{s,t}}\Phi\left(x_{0}\right)\right\Vert _{\mathcal{L}_{2}\left(X\right)}^{p} & \leq C_{q_{0}}^{p}\left\Vert \Lambda\left(t,s\right)\right\Vert _{\mathcal{L}\left(X\right)}^{p}\left[\sum_{i=1}^{M}\left(P_{s,t}\left|v_{i}\right|^{q_{0}}\left(x_{0}\right)\right)^{\frac{2}{q_{0}}}\left|y_{i}\right|_{X}^{2}\right]^{\frac{p}{2}}\\
 & \leq C_{q_{0}}^{p}\left\Vert \Lambda\left(t,s\right)\right\Vert _{\mathcal{L}\left(X\right)}^{p}\left[\sum_{i=1}^{M}P_{s,t}v_{i}^{2}\left(x_{0}\right)\left|y_{i}\right|_{X}^{2}\right]^{\frac{p}{2}}\\
 & =C_{q_{0}}^{p}\left\Vert \Lambda\left(t,s\right)\right\Vert _{\mathcal{L}\left(X\right)}^{p}\left[P_{s,t}\left|\Phi\right|_{X}^{2}\left(x_{0}\right)\right]^{\frac{p}{2}},
\end{align*}
which leads to $\text{\eqref{eq: norma H-s Nabla Ps,t alla p}}$ since
$x\mapsto x^{\frac{p}{2}}$ is also convex.

Integrating $\text{\eqref{eq: norma H-s Nabla Ps,t alla p}}$ over
$x_{0}$ with respect to the measure $\mu_{s}$ and remembering $\text{\eqref{eq: muinv}}$
leads to the fact that $\text{\eqref{stima nabla Ps,t vett}}$ holds
for any $\Phi\in\mathcal{F}C_{b}^{1}\left(X;X\right)$.

The extension of the validity of $\text{\eqref{stima nabla Ps,t vett}}$
to $L^{p}\left(X,\mu_{t};X\right)$ functions is obtained \emph{mutatis
mutandis} exactly as in the proof of Proposition $\text{\ref{prop: operatore compatto}}$,
with the precaution of observing that, if $\left(\Phi_{n}\right)_{n}\subseteq\mathcal{F}C_{b}^{1}\left(X;X\right)$
approximates some $\Phi\in L^{p}\left(X,\mu_{t};X\right)$, then the
functions $P_{s,t}\Phi_{n}$ are also cylindrical, by Lemma $\text{\ref{lem: grad e Ps,t di cilindriche}}$.
\end{proof}
\section{Asymptotic decay in case of time-periodic coefficients}\label{sez: tasso conv}

In this second part of the paper, our aim is to focus on the case
where the data, namely the families of operators $\left\{ A\left(t\right)\right\} _{t\in\mathbb{R}}$
and $\left\{ B\left(t\right)\right\} _{t\in\mathbb{R}}$, are periodic,
and to prove the convergence
\begin{equation}
\left\Vert P_{s,t}f-\int_{X}f\left(y\right)\mu_{t}\left(dy\right)\right\Vert _{L^{p}\left(X,\mu_{s}\right)}\to0\quad\text{for }t-s\to+\infty.\label{eq: conv}
\end{equation}
Most important, we want to estimate the rate of this convergence.

To this purpose, it is useful to begin with an ad-hoc formalization.
Let $\Pi_{t}:L^{p}\left(X,\mu_{t}\right)\to L^{p}\left(X,\mu_{t}\right)$
the projection onto the space of zero-mean functions, namely 
\begin{equation}
\Pi_{t}f:=f-\int_{X}f\left(y\right)\mu_{t}\left(dy\right).\label{eq: proiezione media 0}
\end{equation}

\begin{rem}
\label{rem: proiezione}Clearly, $\Pi_{t}=\Pi_{t}^{2}$ is linear
and, by H\"older's inequality, satisfies
\[
\left\Vert \Pi_{t}\right\Vert _{\mathcal{L}\left(L^{p}\left(X,\mu_{t}\right)\right)}\leq2.
\]

Further, for every $s<t$ the following relations hold: 
\begin{align}
 & L^{p}\left(X,\mu_{t}\right)=Im\Pi_{t}\oplus Ker\Pi_{t}\label{eq: scomposizion Lp}\\
 & P_{s,t}\left(Im\Pi_{t}\right)\subseteq Im\Pi_{s}.\label{eq: Ps,t ristretto}
\end{align}
The first relation is a consequence of $\Pi_{t}$ being a projection.
Note that $Ker\Pi_{t}$ actually does not depend on $t$ since it
coincides with the set of the constant functions from $X$ to $\mathbb{R}$;
further, thanks to $\text{\eqref{eq: scomposizion Lp}}$, it is easy
to check that, for every $r>0$, $Im\Pi_{r}$ is the set of the functions
whose $\mu_{r}$-mean is zero. Thus, since by $\text{\eqref{eq: muinv}}$:
\[
\int_{X}P_{s,t}\Pi_{t}\varphi\left(x\right)\mu_{s}\left(dx\right)=\int_{X}\Pi_{t}\varphi\left(x\right)\mu_{t}\left(dx\right)=0\quad\forall\varphi\in L^{p}\left(X,\mu_{t}\right),
\]
thus $P_{s,t}\Pi_{t}\varphi\in Im\Pi_{s}$, and actually $\Pi_{s}P_{s,t}\Pi_{t}=P_{s,t}\Pi_{t}$.

Further, since $P_{s,t}$ is constant on constant functions, $P_{s,t}f-\int_{X}f\left(y\right)\mu_{t}\left(dy\right)=P_{s,t}\Pi_{t}f$
and we have for every $s<t$:

\begin{align*}
\left\Vert P_{s,t}f-\int_{X}f\left(y\right)\mu_{t}\left(dy\right)\right\Vert _{L^{p}\left(X,\mu_{s}\right)} & =\left\Vert P_{s,t}\Pi_{t}f\right\Vert _{L^{p}\left(X,\mu_{s}\right)}\\
 & =\left\Vert P_{s,t\mid Im\Pi_{t}}\Pi_{t}f\right\Vert _{L^{p}\left(X,\mu_{s}\right)}\\
 & \leq\left\Vert P_{s,t\mid Im\Pi_{t}}\right\Vert _{\mathcal{L}\left(Im\Pi_{t};Im\Pi_{s}\right)}\left\Vert \Pi_{t}f\right\Vert _{L^{p}\left(X,\mu_{t}\right)}\\
 & \leq2\left\Vert P_{s,t\mid Im\Pi_{t}}\right\Vert _{\mathcal{L}\left(Im\Pi_{t};Im\Pi_{s}\right)}\left\Vert f\right\Vert _{L^{p}\left(X,\mu_{t}\right)}.
\end{align*}

So, in order to estimate the behaviour of $P_{s,t}f$, we can focus
on the restricted operator
\begin{equation}
P_{s,t\mid Im\Pi_{t}}:Im\Pi_{t}\to Im\Pi_{s}.\label{eq: def P0}
\end{equation}
Note that both the spaces are Banach with the respective $L^{p}$
norm.
\end{rem}

It is clear that, if the system is, say, $T$-periodic, then the above
operator maps $Im\Pi_{t}$ in itself, for suitable values of $s$.
This opens to the possibility of using spectral theory to understand
the behaviour of the operator $P_{0,T}$ - which becomes, roughly
speaking, a ``root'' of the generic operator $P_{\sigma,\tau}$.

\begin{assumption}
\textcolor{red}{\label{assu: periodicita}} There exists $T>0$ such
that the system is $T$-periodic. Namely:
\[
A\left(t\right)=A\left(t+T\right),\quad B\left(t\right)=B\left(t+T\right),\quad\forall t\in\mathbb{R}.
\]
\end{assumption}

From now on, the period $T$ is fixed. Note that as a consequence
of the latter assumption we have:
\[
U\left(t,s\right)=U\left(t+T,s+T\right),\quad Q\left(t,s\right)=Q\left(t+T,s+T\right)\quad\forall s\leq t.
\]

The rate of convergence in $\text{\eqref{eq: conv}}$ will be estimated
in terms of the decay rate of the evolution operator $U\left(\cdot,\cdot\right)$,
i.e. the negative quantity $\omega_{0}$ appearing in Assumption $\text{\ref{assu: omega_0 neg}}$,
which is a datum of the problem.

Roughly speaking, our approach is the following: thanks to the periodicity
Assumption $\text{\ref{assu: periodicita}}$, the asymptotic behaviour
of $P_{s,t\mid Im\Pi_{t}}$ for $t-s\to+\infty$ is the same as that
of the powers $P_{0,T\mid Im\Pi_{0}}^{N}$ for $N\to\infty$, whose
norms are related to the spectral radius through a well known formula.
The compactness result proved in Proposition $\text{\ref{prop: operatore compatto}}$
will be a fundamental tool to estimate the spectral radius of $P_{0,T\mid Im\Pi_{0}}$;
in addition to it, a certain work will be needed in order to establish
a link between the spectral radius of $P_{0,T\mid Im\Pi_{0}}$ and
$\omega_{0}$. Precisely, the radius is proved to be less than or
equal to the quantity $e^{\omega_{0}T}$, which implies that the norm
$\left\Vert P_{0,T\mid Im\Pi_{0}}^{N}\right\Vert _{\mathcal{L}\left(L^{p}\left(X,\mu_{0}\right)\right)}$
tends to $0$ with a speed that is lower but arbitrarily near to $e^{\omega_{0}NT}$.
This result is Theorem $\ref{teo omega>omega0}$.

But our goal goes beyond this first estimate of the rate of convergence:
we want to give reasonable sufficient conditions (let us say, ``almost''
necessary conditions) under which $e^{\omega_{0}NT}$ is certainly
the highest possible speed of convergence, meaning that $\left\Vert P_{0,T\mid Im\Pi_{0}}^{N}\right\Vert _{\mathcal{L}\left(L^{p}\left(X,\mu_{0}\right)\right)}\leq Ce^{\omega_{0}NT}$
definitely in $N$, and the estimate cannot hold for any $\nu<\omega_{0}$
in place of $\omega_{0}$. These conditions will be identified in
the eigenvalues of maximum modulus of $U\left(T,0\right)^{*}$ being
semisimple, something that holds both in the finite dimensional case
and, most important, in the example given in Section $\text{\ref{sez: esempio}}$.
The proof of the sufficiency of this condition is quite delicate,
particularily in the part regarding the relation between the spectrum
of $U\left(T,0\right)^{*}$ and the spectrum of $P_{0,T\mid Im\Pi_{0}}$.

We split up all these results into two subsections, the first one
regarding the basic estimate, and the second one regarding the optimality
of such estimate.

\vspace{2mm}

Before going into the details we put forward a Lemma which is a basic
consequence of Assumptions $\text{\ref{assu: omega_0 neg}}$ and $\text{\ref{assu: periodicita}}$.
\begin{lem}
\label{lem: raggio spettr U}Let the system satisfy Assumption $\text{\ref{assu: periodicita}}.$
Then the spectral radius of $U\left(T,0\right)$ and $U\left(T,0\right)^{*}$
is $e^{\omega_{0}T}$.
\end{lem}

\begin{proof}
Denote the spectral radius of $U\left(T,0\right)$ by $\rho_{T}$.
By the Gelfand formula we have:
\begin{equation}
\rho_{T}=\lim_{N\to\infty}\left\Vert U\left(T,0\right)^{N}\right\Vert _{\mathcal{L}\left(X\right)}^{\frac{1}{N}}=\lim_{N\to\infty}\left\Vert U\left(NT,0\right)\right\Vert _{\mathcal{L}\left(X\right)}^{\frac{1}{N}},\label{eq: formula raggio spettrale}
\end{equation}
where the latter equality follows from Assumption $\text{\ref{assu: periodicita}}$.

First we show that $\rho_{T}\leq e^{\omega_{0}T}$. Let $\left(\omega_{k}\right)_{k}\subseteq\Omega_{U}$
be such that $\omega_{k}\to\omega_{0}$ for $k\to\infty$, where $\Omega_{U}$
is the set appearing in Assumption $\text{\ref{assu: omega_0 neg}}$.
Then there exists a sequence $\left(M_{k}\right)_{k}\subseteq\left(0,+\infty\right)$
such that
\[
\left\Vert U\left(NT,0\right)\right\Vert _{\mathcal{L}\left(X\right)}\leq M_{k}e^{N\omega_{k}T}\quad\forall k,N\in\mathbb{N}.
\]
Thus, elevating both members of the previous inequality to the power
$1/N$ and then passing to the limit for $N\to\infty$ we obtain $\rho_{T}\leq e^{\omega_{k}T}$
for every $k\in\mathbb{N}$, which implies the desired inequality.

Now assume by contradiction that $\rho_{T}<e^{\omega_{0}T}$; we show
that in this case there exists $\nu_{0}\in\left(-\infty,\omega_{0}\right)\cap\Omega_{U}$.
We deduce from formula $\text{\eqref{eq: formula raggio spettrale}}$
that there exist $\nu_{0}<\omega_{0}$ and $N_{0}\in\mathbb{N}$ such
that:
\begin{equation}
\left\Vert U\left(NT,0\right)\right\Vert _{\mathcal{L}\left(X\right)}\leq e^{\nu_{0}NT}\quad\forall N\geq N_{0}.\label{eq: ip ass rho U}
\end{equation}
We show that $\nu_{0}\in\Omega_{U}$. First note that by Assumption
$\text{\ref{assu: omega_0 neg}}$ there exists $C_{0}>0$ such that:
\begin{equation}
\left\Vert U\left(t,s\right)\right\Vert _{\mathcal{L}\left(X\right)}\leq C_{0}\quad\forall s\leq t.\label{eq: U limitato}
\end{equation}
Now define:
\begin{align*}
 & C_{1}:=\max\left\{ C_{0},e^{\nu_{0}\left(2-N_{0}\right)T}\right\} ,\\
 & C_{2}:=C_{1}^{2}e^{-2\nu_{0}T}.
\end{align*}
Fix $s\leq t$. If $t-s<N_{0}T$, then we note that, by our choice
of the constants:
\begin{align*}
\left\Vert U\left(t,s\right)\right\Vert _{\mathcal{L}\left(X\right)} & \leq C_{1}\\
 & \leq C_{1}^{2}e^{\nu_{0}\left(N_{0}-2\right)T}\\
 & =C_{2}e^{\nu_{0}N_{0}T}\\
 & \leq C_{2}e^{\nu_{0}\left(t-s\right)}.
\end{align*}
If $t-s\geq N_{0}T$ then there exist $K\in\mathbb{Z}$ and $N\in\mathbb{N}$
depending on $t$ and $s$ such that $N\geq N_{0}$ and:
\begin{align*}
 & \left(K-1\right)T\leq s\leq KT,\\
 & \left(K+N\right)T\leq t\leq\left(K+N+1\right)T.
\end{align*}
Thus:
\begin{align*}
\left\Vert U\left(t,s\right)\right\Vert _{\mathcal{L}\left(X\right)} & =\left\Vert U\left(t,\left(K+N\right)T\right)U\left(NT,0\right)U\left(KT,s\right)\right\Vert _{\mathcal{L}\left(X\right)}\\
 & \leq C_{0}^{2}\left\Vert U\left(NT,0\right)\right\Vert _{\mathcal{L}\left(X\right)}\\
 & \leq C_{0}^{2}e^{\nu_{0}NT}\\
 & \leq C_{1}^{2}e^{-2\nu_{0}T}e^{\nu_{0}\left(t-s\right)}\\
 & =C_{2}e^{\nu_{0}\left(t-s\right)},
\end{align*}
by $\text{\eqref{eq: ip ass rho U}}$ and since $t-s\leq\left(N+2\right)T$
by the choice of $N$. Hence we have proved that
\[
\left\Vert U\left(t,s\right)\right\Vert _{\mathcal{L}\left(X\right)}\leq C_{2}e^{\nu_{0}\left(t-s\right)}\quad\forall s\leq t,
\]
where the quantity $C_{2}>0$ depends only on $\nu_{0}$, $T$ and
the problem's data. This implies that $\nu_{0}\in\Omega_{U}$, which
contradicts the definition of $\omega_{0}$. Thus $\rho_{T}=e^{\omega_{0}T}$,
and the same holds for the spectral radius of $U\left(T,0\right)^{*}$.
\end{proof}
\subsection{Exponential convergence of the Ornstein-Uhlenbeck operators}
\leavevmode\label{subsez: caso omega>omega0}

\begin{thm}
\label{teo omega>omega0}Let the system satisfy Assumption $\text{\ref{assu: periodicita}}.$
Then $P_{0,T}\left(Im\Pi_{0}\right)\subseteq Im\Pi_{0}$, and the
spectral radius of $P_{0,T\mid Im\Pi_{0}}$ is less than or equal
to $e^{\omega_{0}T}$.

Consequently, for every $\omega\in\left(\omega_{0},0\right)$, there
exists $N_{\omega}\in\mathbb{N}$ such that:
\begin{equation}
\left\Vert P_{0,T\mid Im\Pi_{0}}^{N}\right\Vert _{\mathcal{L}\left(L^{p}\left(X,\mu_{0}\right)\right)}\leq e^{\omega NT}\quad\forall N\geq N_{\omega}.\label{eq: stima potenza P_0,T}
\end{equation}
Further, there exists also $C_{T,\omega}>0$ such that:
\begin{align}
\left\Vert P_{s,t}f-\int_{X}f\left(y\right)\mu_{t}\left(dy\right)\right\Vert _{L^{p}\left(X,\mu_{s}\right)}\leq C_{T,\omega}e^{\omega\left(t-s\right)}\left\Vert f\right\Vert _{L^{p}\left(X,\mu_{t}\right)}\label{eq: stima Ps,tf - mtf}\\
\forall s<t,\ f\in L^{p}\left(X,\mu_{t}\right),\nonumber 
\end{align}
where $\left\{ \mu_{t}\mid t\in\mathbb{R}\right\} $ is the family
of Gaussian measures defined in $\text{\eqref{def misura inv}}$.
\end{thm}

\begin{proof}
The periodicity of the system implies $\mu_{T}=\mu_{0}$, thus $P_{0,T}\left(Im\Pi_{0}\right)\subseteq Im\Pi_{0}$
by $\text{\eqref{eq: Ps,t ristretto}}$; thus, for every $N\in\mathbb{N}$,
also $Im\Pi_{0}$ is invariant for $P_{0,T}^{N}$, which coincides
with $P_{0,NT}$ by relation $\text{\eqref{eq: op evoluzione}}$ extended
to $L^{p}\left(X,\mu_{0}\right)$ functions.

By Proposition $\text{\ref{prop: operatore compatto}}$, $P_{0,T\mid Im\Pi_{0}}$
is a compact operator, since $Im\Pi_{0}$ is a closed and $P_{0,T}$-invariant
subspace of $L^{p}\left(X,\mu_{0}\right)$. This allows us to relate
the behaviour of the powers of such operator to its eigenvalues.

Denote by $\Sigma$ the spectrum of $P_{0,T\mid Im\Pi_{0}}$ and by
$R_{T}$ its spectral radius, namely $R_{T}:=\max_{\lambda\in\Sigma}\left|\lambda\right|$.
We are going to prove that:
\begin{equation}
R_{T}\leq e^{\omega_{0}T}.\label{eq: R_t leq exp(omega0T)}
\end{equation}
We have:
\begin{align*}
R_{T}=\lim_{N\to\infty}\left\Vert P_{0,T\mid Im\Pi_{0}}^{N}\right\Vert _{\mathcal{L}\left(Im\Pi_{0}\right)}^{\frac{1}{N}} & ,
\end{align*}
where the notation is lawful since $Im\Pi_{0}$ is $P_{0,T}$-invariant.
Thus relation $\text{\eqref{eq: R_t leq exp(omega0T)}}$ will imply
immediately that, for every $\omega\in\left(\omega_{0},0\right)$,
there exists $N_{0}\in\mathbb{N}$ depending on $\omega$ such that
relation $\text{\eqref{eq: stima potenza P_0,T}}$ holds.

The compactness of $P_{0,T\mid Im\Pi_{0}}$ implies that $0\in\Sigma$
because $Im\Pi_{0}$ has infinite dimension; we assume without loss
of generality that $R_{T}>0$ so that $\Sigma$ contains some non-zero
elements, which are eigenvalues of $P_{0,T\mid Im\Pi_{0}}$. Thus:
\[
R_{T}=\sup\left\{ \left|\lambda\right|\mid\lambda\text{ is an eigenvalue of }P_{0,T\mid Im\Pi_{0}},\ \lambda\neq0\right\} .
\]

Let $\lambda$ be an eigenvalue of $P_{0,T\mid Im\Pi_{0}}$ with $\lambda\neq0$,
and let $f\in Im\Pi_{0}$ the associated eigenfunction. Thus:
\begin{equation}
P_{0,NT}f=P_{0,T}^{N}f=\lambda^{N}f\quad\forall N\in\mathbb{N}.\label{eq: autofunzione}
\end{equation}
The idea to estimate the magnitude of $\left|\lambda\right|$ is to
pass to the gradients in the above relation. This will allow us to
use our knowledge about the operators $P_{0,NT}$ and $P_{0,T}$ which
derives from applying the Propositions $\text{\ref{prop: op evoluzione}}$,
$\text{\ref{prop: invarianza}}$ and $\text{\ref{prop: operatore compatto}}$
to them.

Fix $N\in\mathbb{N}$. First observe that relation $\text{\eqref{eq: autofunzione}}$
obviously implies that $f\in W^{1,p}\left(X,\mu_{0}\right)$, by Proposition
$\text{\ref{prop: operatore compatto}}$.

The gradients of the two sides of $\text{\eqref{eq: autofunzione}}$
do coincide, but applying the estimate in $\text{\eqref{stima nabla Ps,t}}$
directly to $\nabla P_{0,NT}$ would lead to the need for estimating
the norm of the operator $\Lambda\left(NT,0\right)=Q^{-\frac{1}{2}}\left(NT,0\right)U\left(NT,0\right)$.
To exploit Assumption $\text{\ref{assu: omega_0 neg}}$ more easily,
we have to use the gradient formula $\text{\eqref{eq: formula grad C1}}$
for $W^{1,p}$ functions, which is legitimate by Proposition $\text{\ref{prop: operatore compatto}}$.

Observe that the function $P_{\left(N-1\right)T,NT}f$ also belongs
to $W^{1,p}\left(X,\mu_{0}\right)$. Consider the operator $\overrightarrow{P_{0,\left(N-1\right)T}}$
as an operator from $L^{p}\left(X,\mu_{0};X\right)$ to itself, as
allowed by Proposition $\text{\ref{prop: estensione P_s,t}}$. Using
the evolution operator property in $\text{\eqref{eq: op evoluzione}}$
- which obviously holds for $\varphi\in W^{1,p}\left(X,\mu_{0}\right)$
- and the gradient formula $\text{\eqref{eq: formula grad C1}}$ for
$W^{1,p}$ functions, we obtain:
\begin{align*}
\nabla P_{0,NT}f & =\nabla\left[P_{0,\left(N-1\right)T}\left(P_{\left(N-1\right)T,NT}f\right)\right]\\
 & =U\left(\left(N-1\right)T,0\right)^{*}\overrightarrow{P_{0,\left(N-1\right)T}}\left(\nabla P_{\left(N-1\right)T,NT}f\right)\\
 & =U\left(\left(N-1\right)T,0\right)^{*}\overrightarrow{P_{0,\left(N-1\right)T}}\left(\nabla P_{0,T}f\right),
\end{align*}
Now let $\omega\in\left(\omega_{0},0\right)$; by Assumption $\text{\ref{assu: omega_0 neg}}$,
there exists $M_{\omega}>0$ such that
\[
\left\Vert U\left(\tau,\sigma\right)^{*}\right\Vert _{\mathcal{L}\left(X\right)}\leq M_{\omega}e^{\omega\left(\tau-\sigma\right)}\quad\forall\sigma<\tau.
\]

Thus, applying $\text{\eqref{stima nabla Ps,t}}$ and remembering
that $\overrightarrow{P_{0,\left(N-1\right)T}}$ is a contraction:
\begin{align}
\left\Vert \nabla P_{0,NT}f\right\Vert _{L^{p}\left(X,\mu_{0};X\right)} & \leq\left\Vert U\left(\left(N-1\right)T,0\right)^{*}\right\Vert _{\mathcal{L}\left(X\right)}\left\Vert \nabla P_{0,T}f\right\Vert _{L^{p}\left(X,\mu_{0};X\right)}\nonumber \\
 & \leq M_{\omega}e^{\omega\left(N-1\right)T}C_{p}\left\Vert \Lambda\left(T,0\right)\right\Vert _{\mathcal{L}\left(X\right)}\left\Vert f\right\Vert _{L^{p}\left(X,\mu_{0}\right)}\nonumber \\
 & =:C\left(p,T,\omega\right)e^{\omega NT}\left\Vert f\right\Vert _{L^{p}\left(X,\mu_{0}\right)}.\label{eq: stima nabla esponenziale}
\end{align}

Hence, passing to the gradients in $\text{\eqref{eq: autofunzione}}$
and using $\text{\eqref{eq: stima nabla esponenziale}}$ we obtain:
\begin{align*}
\left|\lambda\right|^{N}\left\Vert \nabla f\right\Vert _{L^{p}\left(X,\mu_{0};X\right)} & =\left\Vert \nabla P_{0,NT}f\right\Vert _{L^{p}\left(X,\mu_{0};X\right)}\\
 & \leq C\left(p,T,\omega\right)e^{\omega NT}\left\Vert f\right\Vert _{L^{p}\left(X,\mu_{0}\right)},\quad\forall N\in\mathbb{N}.
\end{align*}
This implies $\left|\lambda\right|\leq e^{\omega T}$. Indeed, in
the opposite case, we would infer from the above inequality that $\nabla f=0$,
by dividing both members by $\left|\lambda\right|^{N}$(which is not
null) and then letting $N\to\infty$. Thus, $f$ would be constant,
namely $f\in Ker\Pi_{0}$. Since by assumption $f\in Im\Pi_{0}$,
the decomposition in $\text{\eqref{eq: scomposizion Lp}}$ would imply
$f=0$, a contradiction.

The reason why we have focused our attention to the restricted operator
$P_{0,T\mid Im\Pi_{0}}$ is now apparent: such operator does not admit
any constant eigenfunction.

This argument shows that $R_{T}\leq e^{\omega T}$, and, therefore,
that $\text{\eqref{eq: R_t leq exp(omega0T)}}$ holds, since $\omega>\omega_{0}$
is generic.

Now we prove the final part of the statement. Assume without loss
of generality that $N_{\omega}\geq2$ and fix $s<t$.

If $t-s<N_{\omega}T$ then
\begin{equation}
\left\Vert P_{s,t\mid Im\Pi_{t}}\right\Vert _{\mathcal{L}\left(Im\Pi_{t};Im\Pi_{s}\right)}\leq1\leq e^{-\omega N_{\omega}T}e^{\omega\left(t-s\right)}.\label{eq: stima P_s,t t-s piccolo}
\end{equation}
In case that $t-s\geq N_{\omega}T$, let $\left(K-1\right)T\leq s\leq KT$
and $\left(K+N\right)T\leq t\leq\left(K+N+1\right)T$ for suitable
integers $K,N$ with $N\geq N_{\omega}$. We have:

\begin{align}
P_{s,t} & =P_{s,KT}P_{KT,\left(K+N\right)T}P_{\left(K+N\right)T,t}\nonumber \\
 & =P_{s,KT}P_{0,T}^{N}P_{\left(K+N\right)T,t}.\label{scomposizione}
\end{align}
Observe that, in general, $P_{q,r}P_{r,t\mid Im\Pi_{t}}=P_{q,r\mid Im\Pi_{r}}P_{r,t\mid Im\Pi_{t}}$,
by $\text{\eqref{eq: Ps,t ristretto}}$. Hence it follows from $\text{\eqref{scomposizione}}$
that:
\[
P_{s,t\mid Im\Pi_{t}}=P_{s,KT\mid Im\Pi_{0}}P_{0,T\mid Im\Pi_{0}}^{N}P_{\left(K+N\right)T,t\mid Im\Pi_{t}}.
\]
In particular, since every Ornstein-Uhlenbeck operator is a contraction,
we obtain from this and $\text{\eqref{eq: stima potenza P_0,T}}$:
\begin{align*}
\left\Vert P_{s,t\mid Im\Pi_{t}}\right\Vert _{\mathcal{L}\left(Im\Pi_{t};Im\Pi_{s}\right)} & \leq\left\Vert P_{0,T\mid Im\Pi_{0}}^{N}\right\Vert _{\mathcal{L}\left(Im\Pi_{0}\right)}\\
 & \leq e^{\omega NT}\\
 & \leq e^{-\omega N_{\omega}T}e^{\omega\left(t-s\right)},
\end{align*}
where the last inequality holds since $\left(N+2\right)T\geq t-s$,
by the choice of $N$ and $K$.

This estimate together with $\text{\eqref{eq: stima P_s,t t-s piccolo}}$
means:
\begin{equation}
\left\Vert P_{s,t\mid Im\Pi_{t}}\right\Vert _{\mathcal{L}\left(Im\Pi_{t};Im\Pi_{s}\right)}\leq e^{-\omega N_{\omega}T}e^{\omega\left(t-s\right)}\quad\forall s<t.\label{eq: stima P_s,t s e t qualunque}
\end{equation}

Combining $\text{\eqref{eq: stima P_s,t s e t qualunque}}$ with the
final estimate in Remark $\text{\ref{rem: proiezione}}$ we obtain,
for every $f\in L^{P}\left(X,\mu_{t}\right)$: 
\begin{align*}
\left\Vert P_{s,t}f-\int_{X}f\left(y\right)\mu_{t}\left(dy\right)\right\Vert _{L^{p}\left(X,\mu_{s}\right)} & \leq2\left\Vert P_{s,t\mid Im\Pi_{t}}\right\Vert _{\mathcal{L}\left(Im\Pi_{t};Im\Pi_{s}\right)}\left\Vert f\right\Vert _{L^{p}\left(X,\mu_{t}\right)}\\
 & \leq2e^{-\omega N_{\omega}T}e^{\omega\left(t-s\right)}\left\Vert f\right\Vert _{L^{p}\left(X,\mu_{t}\right)}.
\end{align*}
Relation $\text{\eqref{eq: stima Ps,tf - mtf}}$ is thus proved with
$C_{T,\omega}=2e^{-\omega N_{\omega}T}$, which makes sense since
$N_{\omega}$ depends on $T$, $\omega$, $\bar{\omega}$ and thus,
ultimately, on $T$, $\omega_{0}$ and $\omega$.

This concludes the proof.
\end{proof}
\subsection{Bounds on the convergence rates}
\leavevmode\label{subsez: caso omega leq omega0}

Here we want to explore the question of the optimality of the convergence
estimates in Theorem $\text{\ref{teo omega>omega0}}$. The new assumption,
satisfied by the example in Section $\text{\ref{sez: esempio}}$,
is the following.
\begin{assumption}
\label{ass: U* ha autov di mod max}The operator $U\left(T,0\right)^{*}$
has an eigenvalue whose modulus is equal to its spectral radius, namely
- as established by Lemma $\text{\ref{lem: raggio spettr U}}$ - to
$e^{\omega_{0}T}$.
\end{assumption}

The main tool in this subsection will be spectral analysis, and in
particular semisimple eigenvalues and spectral projections. We briefly
illustrate the some basic properties of these notions.

If $E$ be a Banach space, $T\in\mathcal{L}\left(E\right)$, we denote
by $\sigma\left(T\right)$ and $\rho\left(T\right):=\mathbb{C}\setminus\sigma\left(T\right)$
the spectrum and the resolvent set of $T$, respectively. For every
$z\in\rho\left(T\right)$, the operator $R\left(z,T\right):=\left(zI-T\right)^{-1}$
is the \emph{resolvent operator} of $T$ at the point $z$, and belongs
to $\mathcal{L}\left(E\right)$ by definition. The domain $\rho\left(T\right)$
of the operator valued function $R\left(\cdot,T\right)$ is open and
$R\left(\cdot,T\right)$ is holomorphic in $\rho\left(T\right)$,
and has singularities at the points of $\sigma\left(T\right)$, which
is compact.

An isolated element of $\sigma\left(T\right)$\footnote{Namely, a complex number $\lambda\in\sigma\left(T\right)$ that admits
a punctured neighbourhood not intersecting $\sigma\left(T\right)$.} that is a first order pole of $R\left(\cdot,T\right)$ is called
\emph{semisimpl}e \emph{for} $T$. This notion has many characterizations,
among which are the following.

\vspace{2mm}

\emph{If $\lambda$ is an isolated element of $\sigma\left(T\right)$,
then $\lambda$ is semisimple for $T$ if and only $Im\left(\lambda I-T\right)$
is closed and 
\begin{equation}
E=Ker\left(\lambda I-T\right)\oplus Im\left(\lambda I-T\right).\label{eq: scomposizione spazio semisempl}
\end{equation}
}

We recall the notion of spectral projection. Let $\Omega\subseteq\mathbb{C}$
be a bounded region with rectifiable boundary $\partial\Omega\subseteq\rho\left(T\right)$,
and let $\sigma_{0}\left(T\right)\subseteq\sigma\left(T\right)$ such
that $\sigma_{0}\left(T\right)$ is made of isolated elements of $\sigma\left(T\right)$.
Assume that $\sigma_{0}\left(T\right)=\sigma\left(T\right)\cap\Omega$. 

The \emph{spectral projection} of $T$ relative to the region $\Omega$
is the operator defined by:
\[
\mathcal{P}\left(\Omega,T\right)=\frac{1}{2\pi i}\int_{\partial\Omega}R\left(\xi,T\right)d\xi.
\]
Clearly $\mathcal{P}\left(\Omega,T\right)\in\mathcal{L}\left(E\right)$,
and it is easily proven that $\mathcal{P}\left(\Omega,T\right)=\mathcal{P}\left(\Omega,T\right)^{2}$
(the reason why it is called a projection). This implies that both
$Ker\mathcal{P}\left(\Omega,T\right)$ and $Im\mathcal{P}\left(\Omega,T\right)$
are closed and $T$-invariant subspaces of $E$, and that
\[
E=Im\mathcal{P}\left(\Omega,T\right)\oplus Ker\mathcal{P}\left(\Omega,T\right).
\]
The invariance of $Im\mathcal{P}\left(\Omega,T\right)$ and $Ker\mathcal{P}\left(\Omega,T\right)$
is a consequence of the fact that $T$ and $\mathcal{P}\left(\Omega,T\right)$
commute, because any bounded operator can enter and exit the Bochner
integral through which $\mathcal{P}$ is defined.

Further, it can be proved that, for every $\lambda\in\sigma_{0}\left(T\right)$:
\begin{align}
 & Ker\mathcal{P}\left(\Omega,T\right)\subseteq Im\left(\lambda I-T\right),\label{eq: Ker Proiez Spettr}\\
 & Ker\left(\lambda I-T\right)\subseteq Im\mathcal{P}\left(\Omega,T\right).\label{eq: Im  Proiez Spettr}
\end{align}

The main utility of this projection operator is in the following spectrum
separation property:
\begin{align*}
\sigma\left(T_{\mid Im\mathcal{P}\left(\Omega,T\right)}\right) & =\sigma_{0}\left(T\right),\\
\sigma\left(T_{\mid Ker\mathcal{P}\left(\Omega,T\right)}\right) & =\sigma\left(T\right)\setminus\sigma_{0}\left(T\right).
\end{align*}

Spectral projections can be used to further characterize semisimplicity.

If $\lambda$ is an isolated element of $\sigma\left(T\right)$, we
denote by $\mathcal{P}\left(\lambda,T\right)$ the spectral projection
of $T$ relative to any small region containing $\lambda$ and no
other element of $\sigma\left(T\right)$ (the integral in the definition
does not depend on the path, by Cauchy's integral theorem).

Then, an analysis of the $\mathcal{L}\left(X\right)$-valued coefficients
in the Laurent development of $R\left(\cdot,T\right)$ leads to the
following equivalence:

\vspace{2mm}

\emph{$\lambda$ is semisimple for $T$ is and only if the operator
$\left(\lambda I-T\right)\mathcal{P}\left(\lambda,T\right)$ is identically
zero}\footnote{For $n\in\mathbb{Z}$, the $n$-th coefficient in the Laurent development
of $R\left(\cdot,T\right)$ in $\lambda$ is $a_{n}=\frac{1}{2\pi i}\int_{\gamma\left(\lambda\right)}\left(z-\lambda\right)^{-\left(n+1\right)}R\left(z,T\right)dz$.
It is easily proven that, for every $n\leq-2$:
\[
a_{n}=\left(\lambda I-T\right)^{-\left(n+1\right)}\mathcal{P}\left(\lambda,T\right).
\]
Hence if $\left(\lambda I-T\right)\mathcal{P}\left(\lambda,T\right)=0$
then $a_{n}=0$ for every $n\leq-2$, which means exactly that $\lambda$
is a first order pole of $R\left(\cdot,T\right)$ - namely that $\lambda$
is semisimple for $T$. The converse is obvious. }\emph{.}

\vspace{2mm}

Note that, in virtue of the above assertion, an isolated element $\lambda$
of $\sigma\left(T\right)$ is semisimple for $T$ if and only if the
inclusions in $\eqref{eq: Ker Proiez Spettr}$ and $\text{\eqref{eq: Im  Proiez Spettr}}$
are actually identities.
\begin{rem}
\label{rem:  semisemplici compatto}If $T$ is compact and $\lambda\neq0$
is an element of $\sigma\left(T\right)$, then the condition 
\begin{equation}
Ker\left(\lambda I-T\right)^{2}\subseteq Ker\left(\lambda I-T\right)\label{eq: semisempl compatto}
\end{equation}
is equivalent to the semisimplicity of $\lambda$ with respect to
$T$.

Indeed, if $\lambda$ is semisimple for $T$ then relation $\text{\eqref{eq: scomposizione spazio semisempl}}$
implies $\eqref{eq: semisempl compatto}$.

Conversely, assume that $\eqref{eq: semisempl compatto}$ holds and
recall that by a well known property of compact operators, $\lambda$
must be an isolated eigenvalue of $T$ and a \emph{pole} of $R\left(\cdot,T\right)$
- meaning that the minimum $n_{0}$ of the natural numbers $n$ such
that $\left(\lambda I-T\right)^{n}\mathcal{P}\left(\lambda,T\right)^{n}=0$
must be finite. We show that $\eqref{eq: semisempl compatto}$ implies
that $n_{0}=1$.

Indeed, assume by contradiction that $n_{0}\geq2$. There exists $x_{0}\neq0$
such that
\[
y_{0}:=\left(\lambda I-T\right)^{n_{0}-1}\mathcal{P}\left(\lambda,T\right)^{n_{0}-1}x_{0}\neq0.
\]
Since $\left(\lambda I-T\right)^{n_{0}-2}\mathcal{P}\left(\lambda,T\right)^{n_{0}}x_{0}\in Ker\left(\lambda I-T\right)^{2}$,
we have $y_{0}\in Ker\mathcal{P}\left(\lambda,T\right)$ by $\eqref{eq: semisempl compatto}$.
Since by definition $y_{0}\in Im\mathcal{P}\left(\lambda,T\right)$
(because $n_{0}>1$), we have $y_{0}=0$, which is absurd.
\end{rem}

\begin{rem}
\label{rem: polinomiali}By substituting, in the definitional formula
$\text{\eqref{eq: def O-U formula}}$, the domain $C_{b}\left(X\right)$
with the larger set $C_{p}\left(X\right)$ of continuous functions
with polynomial growth of a certain order\footnote{Namely, the functions $\varphi\in C\left(X\right)$ such that there
exists $k\in\mathbb{N}$ and $C>0$ such that $\left|\varphi\left(x\right)\right|\leq C\left(1+\left|x\right|_{X}^{k}\right)$
for every $x\in X$.} one obtains a substantially equivalent theory. Precisely, all the
results in Section $\text{\ref{sez: O-U Cb}}$ and in Subsection $\text{\ref{subsez: reg Cb}}$
hold true with the same formulation, as well as Proposition $\text{\ref{prop: CM formula}}$,
while the ``invariance'' property in Proposition $\text{\ref{prop: invarianza}}$
holds with $\varphi\in C_{p}\left(X\right)$ (with the same proof).
Consequently, the extension in Proposition $\text{\ref{prop: estensione P_s,t}}$
can be performed by taking $C_{p}\left(X\right)$ as the initial domain,
which implies that the extended operator $P_{s,t}^{p}$, when applied
to a function in $C_{p}\left(X\right)$, can be considered to be represented
by formula $\text{\eqref{eq: def O-U formula}}$ .
\end{rem}

We establish a first relationship between the spectrum of $P_{0,T\mid Im\Pi_{0}}$
and the spectrum of the evolution operator $U\left(T,0\right)^{*}$.
\begin{lem}
\label{lem: autov. 1}i) Let the system satisfy Assumption $\text{\ref{assu: periodicita}}$.
Then every eigenvalue $\lambda$ of $U\left(T,0\right)^{*}$ is an
eigenvalue of $P_{0,T\mid Im\Pi_{0}}$, and, for every $w\in Ker\left(\lambda I-U\left(T,0\right)^{*}\right)\setminus\left\{ 0\right\} $,
the linear function $\left\langle w,\cdot\right\rangle _{X}$ is an
eigenfunction of $P_{0,T\mid Im\Pi_{0}}$ relative to $\lambda$.

ii) If also Assumption $\text{\ref{ass: U* ha autov di mod max}}$
is satisfied, the spectral radius of $P_{0,T\mid Im\Pi_{0}}$ is $e^{\omega_{0}T}$.
\end{lem}

\begin{proof}
i) Let $\lambda\in\mathbb{C}$ be an eigenvalue of $U^{*}\left(T,0\right)$
and $w\in Ker\left(\lambda I-U\left(T,0\right)^{*}\right)\setminus\left\{ 0\right\} $.
Define $\psi:=\left\langle w,\cdot\right\rangle _{X}$. Then
\[
\int_{x}\psi\left(y\right)\mu_{0}\left(dy\right)=\int_{x}\left\langle w,y\right\rangle _{X}\mu_{0}\left(dy\right)=0,
\]
since the mean of $\mu_{0}$ is zero, by the definition in $\text{\eqref{def misura inv}}$.
Thus $\psi\in Im\Pi_{0}$. Further, by Remark $\text{\ref{rem: polinomiali}}$,
we have, for every $x\in X$:
\begin{align*}
P_{0,T}\psi\left(x\right) & =\int_{X}\left\langle w,y+U\left(T,0\right)x\right\rangle _{X}\gamma_{0,T}\left(dy\right)\\
 & =\int_{X}\left\langle w,y\right\rangle _{X}\gamma_{0,T}\left(dy\right)+\left\langle U\left(T,0\right)^{*}w,x\right\rangle _{X}\\
 & =\lambda\left\langle w,x\right\rangle _{X}\\
 & =\lambda\psi\left(x\right).
\end{align*}

So $\lambda$ is also an eigenvalue of $P_{0,T\mid Im\Pi_{0}}$, relative
to the linear eigenfunction $\psi$.

ii) The latter argument implies that the spectral radius $R_{0}$
of $P_{0,T\mid Im\Pi_{0}}$ is greater than or equal to the maximum
modulus of an eigenvalue of $U^{*}\left(T,0\right)$ which is - by
Assumption $\text{\ref{ass: U* ha autov di mod max}}$, the spectral
radius of $U^{*}\left(T,0\right)$. The latter is $e^{\omega_{0}T}$,
by Lemma $\text{\ref{lem: raggio spettr U}}$; thus, $R_{0}\geq e^{\omega_{0}T}$.
The equality follows from Theorem $\text{\ref{teo omega>omega0}}$.
\end{proof}
Thanks to the previous lemma, we can be sure that $e^{\omega_{0}\left(t-s\right)}$
is the highest possible speed of the convergence in $\text{\eqref{eq: conv}}$.

\begin{thm}
Assume that $\omega<\omega_{0}$. Then there is no $N_{\omega}\in\mathbb{N}$
such that $\text{\eqref{eq: stima potenza P_0,T}}$ holds, and there
is no $C_{T,\omega}>0$ such that $\text{\eqref{eq: stima Ps,tf - mtf}}$
holds.
\end{thm}

\begin{proof}
Let $\lambda_{0}$ be an eigenvalue of $U^{*}\left(T,0\right)$ such
that $\left|\lambda_{0}\right|=e^{\omega_{0}T}$; such $\lambda_{0}$
exists by Assumption $\text{\ref{ass: U* ha autov di mod max}}$.
Take $\psi_{0}=\left\langle w_{0},\cdot\right\rangle _{X}$, where
$w_{0}\in Ker\left[\lambda_{0}I-U\left(T,0\right)^{*}\right]$ . Thus,
by Lemma $\text{\ref{lem: autov. 1}}$, we have for every $N\in\mathbb{N}$:
\begin{align*}
 & \left\Vert P_{0,T\mid Im\Pi_{0}}^{N}\psi_{0}\right\Vert _{L^{p}\left(X,\mu_{0}\right)}=e^{\omega_{0}NT}\left\Vert \psi_{0}\right\Vert _{L^{p}\left(X,\mu_{0}\right)}\\
\implies & \left\Vert P_{0,T\mid Im\Pi_{0}}^{N}\right\Vert _{\mathcal{L}\left(Im\Pi_{0}\right)}\geq e^{\omega_{0}NT}.
\end{align*}
This implies that, if relation $\text{\eqref{eq: stima potenza P_0,T}}$
holds for some $\omega\in\mathbb{R}$, then $\omega_{0}\leq\omega$.
Thus, there are no $\omega<\omega_{0}$ and $N_{0}\in\mathbb{N}$
such that relation $\text{\eqref{eq: stima potenza P_0,T}}$ holds.

Further, taking $f=\psi_{0}$, $s=0$ and $t=NT$ in relation $\text{\eqref{eq: stima Ps,tf - mtf}}$
leads to
\[
e^{\left(\omega_{0}-\omega\right)NT}\leq C_{T,\omega}\quad\forall N\in\mathbb{N};
\]
therefore, there is no $\omega<\omega_{0}$ such that $\text{\eqref{eq: stima Ps,tf - mtf}}$
holds.
\end{proof}
Now we pass to the last and hardest part of the section, which is
the proof of the following result stating a sufficient condition on
the evolution operator $\left\{ U\left(t,s\right)\right\} _{t<s}$
in order that the speed of convergence $e^{\omega_{0}\left(t-s\right)}$
can actually be reached in $\text{\eqref{eq: conv}}$.
\begin{thm}
\label{teo: omega uguale omega0}Let the system satisfy Assumptions
$\text{\ref{assu: periodicita}}$ and $\text{\ref{ass: U* ha autov di mod max}}$,
and assume that all the eigenvalues of $U\left(T,0\right)^{*}$ of
maximum modulus $e^{\omega_{0}T}$ are semisimple for $U\left(T,0\right)^{*}$.
Then, there exists a constant $C>0$, depending only on the problem's
data, and a positive integer $N_{0}$, such that:
\[
\left\Vert P_{0,T\mid Im\Pi_{0}}^{N}\right\Vert _{\mathcal{L}\left(L^{p}\left(X,\mu_{0}\right)\right)}\leq Ce^{\omega_{0}NT}\quad\forall N\geq N_{0}.
\]

Consequently, there exists $C_{T}>0$ such that:
\begin{align*}
\left\Vert P_{s,t}f-\int_{X}f\left(y\right)\mu_{t}\left(dy\right)\right\Vert _{L^{p}\left(X,\mu_{s}\right)}\leq C_{T}e^{\omega_{0}\left(t-s\right)}\left\Vert f\right\Vert _{L^{p}\left(X,\mu_{t}\right)}\\
\forall s<t,\ f\in L^{p}\left(X,\mu_{t}\right).
\end{align*}
\end{thm}

To face the proof of this theorem, we gather some useful notations,
definitions, theoretical results and lemmas.

First, we define the second order Sobolev space with respect to $\mu_{0}$
as:
\begin{equation}
W^{2,p}\left(X,\mu_{0}\right):=\left\{ f\in W^{1,p}\left(X,\mu_{0}\right)\mid\nabla f\in W^{1,p}\left(X,\mu_{0};X\right)\right\} ,\label{def W2p}
\end{equation}
where $W^{1,p}\left(X,\mu_{0};X\right)$ is the space defined in Subsection
$\text{\ref{subsez: Sob vett}}$.

For every $f\in W^{2,p}\left(X,\mu_{0}\right)$, we set
\[
\nabla^{2}f:=\nabla_{\mathcal{L}_{2}}\nabla f.
\]

\begin{lem}
\label{cor: nabla2 uguale 0}Assume that $f\in W^{2,p}\left(X,\mu_{0}\right)$
and $\nabla^{2}f=0$ $\mu_{0}$-almost everywhere in $X$. Thus $f$
coincides with an affine functional, $\mu_{0}$-almost everywhere
in $X$.

If, further, $f\in Im\Pi_{0}$, then $f$ is linear, $\mu_{0}$-almost
everywhere in $X$.
\end{lem}

\begin{proof}
Set for simplicity of notation $Q_{0}=Q\left(0,-\infty\right)$. It
is a direct consequence of \cite[Theorem 2.7]{ADDONA-Poinc} applied
to our problem setting that there exists $K_{p}>0$ such that the
following infinite dimensional, Gaussian Poincar\'e-type inequality
for gradients holds:
\[
\left\Vert Q_{0}^{\frac{1}{2}}\nabla\varphi-\int_{X}Q_{0}^{\frac{1}{2}}\nabla\varphi\left(x\right)\mu_{0}\left(dx\right)\right\Vert _{L^{p}\left(X,\mu_{0};X\right)}\leq K_{p}\left\Vert Q_{0}^{\frac{1}{2}}\nabla^{2}\varphi Q_{0}^{\frac{1}{2}}\right\Vert _{L^{p}\left(X,\mu_{0};\mathcal{L}_{2}\left(X\right)\right)},
\]
for every $\varphi\in W^{2,p}\left(X,\mu_{0}\right)$. Applying this
inequality to a function $f\in W^{2,p}\left(X,\mu_{0}\right)$ such
that $\nabla^{2}f=0$, we obtain that there exists $v\in X$ such
that $\nabla f=v$, $\mu_{0}$- almost everywhere in $X$, remembering
that $Q_{0}^{\frac{1}{2}}$ is injective because of Assumption $\text{\ref{assu: mu_t non degeneri}}$.

Set $c:=\int_{X}\left\{ f\left(x\right)-\left\langle v,x\right\rangle \right\} \mu_{0}\left(dx\right)$.
Thus, for some $k_{p}>0$:
\begin{align*}
\left\Vert f-\left\langle v,\cdot\right\rangle _{X}-c\right\Vert _{L^{p}\left(X,\mu_{0}\right)} & \leq k_{p}\left\Vert \nabla\left(f-\left\langle v,\cdot\right\rangle _{X}\right)\right\Vert _{L^{p}\left(X,\mu_{0};X\right)}\\
 & =k_{p}\left\Vert \nabla f-v\right\Vert _{L^{p}\left(X,\mu_{0};X\right)}\\
 & =0,
\end{align*}
by the scalar, Gaussian, infinite dimensional Poincar\'e inequality,
whose proof can also be derived from \cite[Theorem 2.7]{ADDONA-Poinc}
(see also \cite[Theorem 5.5.11]{BOG;GM}). Thus
\[
f\left(x\right)=\left\langle v,x\right\rangle _{X}+c\quad\text{for }\mu_{0}\text{-almost every }x\in X.
\]
Note that $\int_{X}fd\mu_{0}=c$ since $\int_{X}\left\langle v,\cdot\right\rangle _{X}d\mu_{0}=0$,
so also the last statement is proved because $f\in Im\Pi_{0}$ means
$\int_{X}fd\mu_{0}=0$.
\end{proof}
To simplify the notation, we denote:
\begin{itemize}
\item $P_{0}:=P_{0,T\mid Im\Pi_{0}}$;
\item by $R_{0}$ the spectral radius of $P_{0}$;
\item by $\sigma_{p}\left(\cdot\right)$ the point spectrum of an operator;
\item by $\sigma_{s}\left(\cdot\right)$ the set of semisimple eigenvalues
of an operator;
\item by $S_{\mathbb{C}}\left(0,r\right)$ the complex circumference of
radius $r>0$, centred at $0$.
\end{itemize}
The following results about $P_{0}$ obviously require Assumption
$\text{\ref{assu: periodicita}}$, but, formally, they do not require
Assumption $\text{\ref{ass: U* ha autov di mod max}}$.
\begin{lem}
\label{lem: autospazio di P0}Every eigenfunction of $P_{0}$ relative
to a non-zero eigenvalue is twice differentiable. Further, for every
$\omega\in\left(\omega_{0},0\right)$, there exists $C_{\omega}>0$
such that, for every $N\geq2$, $\lambda\in\sigma\left(P_{0}\right)\setminus\left\{ 0\right\} $
and $g\in Ker\left(\lambda I-P_{0}\right)\setminus\left\{ 0\right\} $:
\begin{equation}
\left|\lambda\right|^{N}\left\Vert \nabla^{2}g\right\Vert _{L^{p}\left(X;\mu_{0};\mathcal{L}\left(X\right)\right)}\leq C_{\omega}e^{2\omega NT}\left\Vert g\right\Vert _{L^{p}\left(X,\mu_{0}\right)}.\label{eq: rel nabla quadro}
\end{equation}

Consequently, if $\lambda_{0}\in\sigma\left(P_{0}\right)$ with $\left|\lambda_{0}\right|=e^{\omega_{0}T}$
and $g$ is any eigenfunction of $P_{0}$ relative to $\lambda_{0}$,
then
\[
\left\Vert \nabla^{2}g\right\Vert _{L^{p}\left(X;\mu_{0};\mathcal{L}\left(X\right)\right)}=0.
\]
Thus, by Lemma $\text{\ref{cor: nabla2 uguale 0}}$, $g$ can be considered
a linear functional, up to a $\mu_{0}$-negligible subset of $X$.
\end{lem}

\begin{proof}
Fix $\omega$, $N$, $\lambda$ and $g$ like in the hypothesis. We
first observe that $\lambda g=P_{0}g$ implies $g\in W^{1,p}\left(X,\mu_{0}\right)$,
so $\lambda\nabla g=\nabla P_{0}g=U\left(T,0\right)^{*}\overrightarrow{P_{0,T}}\nabla g$,
by Proposition $\text{\ref{prop: operatore compatto}}$ and by formula
$\text{\eqref{eq: formula grad C1}}$ applied to functions in $W^{1,p}\left(X,\mu_{0}\right)$.
By Proposition $\text{\ref{prop: reg vett}}$, $\overrightarrow{P_{0,T}}\nabla g\in W^{1,p}\left(X,\mu_{0};X\right)$,
hence also $\nabla g\in W^{1,p}\left(X,\mu_{0};X\right)$.

From $\lambda^{N}g=P_{0}^{N}g$ and formula $\text{\eqref{eq: formula grad vettoriale}}$
applied to $W^{1,p}\left(X,\mu_{0};X\right)$-functions (which is
a lawful tool as explained at point i) in Remark $\text{\ref{rem: su Pst operatoriale}}$),
we infer that:
\begin{align*}
\lambda^{N}\nabla^{2}g & =\nabla^{2}P_{0}^{N}g\\
 & =\nabla^{2}P_{0,\left(N-1\right)T}P_{0}g\\
 & =\nabla_{\mathcal{L}_{2}}U\left(\left(N-1\right)T,0\right)^{*}\overrightarrow{P_{0,\left(N-1\right)T}}\nabla P_{0}g\\
 & =U\left(\left(N-1\right)T,0\right)^{*}\nabla_{\mathcal{L}_{2}}\overrightarrow{P_{0,\left(N-2\right)T}}\overrightarrow{P_{0,T}}\nabla P_{0}g\\
 & =U\left(\left(N-1\right)T,0\right)^{*}\overrightarrow{\overrightarrow{P_{0,\left(N-2\right)T}}}\left(\nabla_{\mathcal{L}_{2}}\overrightarrow{P_{0,T}}\nabla P_{0}g\right)U\left(\left(N-2\right)T,0\right).
\end{align*}
The central term needs some clarification. Remember that $P_{0}:Im\Pi_{0}\subseteq L^{p}\left(X,\mu_{0}\right)\to W^{1,p}\left(X,\mu_{0}\right)$
and $\overrightarrow{P_{0,T}}:L^{p}\left(X,\mu_{0};X\right)\to W^{,1p}\left(X,\mu_{0};X\right)$;
thus by the definition in $\text{\eqref{eq: def Ps,t op nabla Phi Lp}}$,
$\overrightarrow{\overrightarrow{P_{0,\left(N-2\right)T}}}\nabla_{\mathcal{L}_{2}}\overrightarrow{P_{0,T}}\nabla P_{0}g\in L^{p}\left(X,\mu_{0};\mathcal{L}_{2}\left(X\right)\right)$.
By point ii) in Remark $\text{\ref{rem: su Pst operatoriale}}$, remembering
relations $\text{\eqref{stima nabla Ps,t vett}}$ and $\text{\eqref{stima nabla Ps,t}}$,
we obtain:
\begin{align*}
\left\Vert \overrightarrow{\overrightarrow{P_{0,\left(N-2\right)T}}}\nabla_{\mathcal{L}_{2}}\overrightarrow{P_{0,T}}\nabla P_{0}g\right\Vert _{L^{p}\left(X,\mu_{0};\mathcal{L}_{2}\left(X\right)\right)} & \leq\left\Vert \nabla_{\mathcal{L}_{2}}\overrightarrow{P_{0,T}}\nabla P_{0}g\right\Vert _{L^{p}\left(X,\mu_{0};\mathcal{L}_{2}\left(X\right)\right)}\\
 & \leq\widetilde{C}_{p}\left\Vert g\right\Vert _{L^{p}\left(X,\mu_{0}\right)}.
\end{align*}

It follows from Assumption $\text{\ref{assu: omega_0 neg}}$ that
there exists a constant $C_{\omega}>0$ such that $\text{\eqref{eq: rel nabla quadro}}$
holds.

Assume now that $\lambda_{0}\in\sigma\left(P_{0}\right)$ with $\left|\lambda_{0}\right|=e^{\omega_{0}T}$
and $g\in Ker$$\left(\lambda_{0}I-P_{0}\right)\setminus\left\{ 0\right\} $;
then by $\text{\eqref{eq: rel nabla quadro}}$:
\begin{align*}
\left\Vert \nabla^{2}g\right\Vert _{L^{p}\left(X,\mu_{0};\mathcal{L}\left(X\right)\right)} & \leq C_{\omega}e^{\left(2\omega-\omega_{0}\right)NT}\left\Vert g\right\Vert _{L^{p}\left(X,\mu_{0}\right)}.
\end{align*}
Fixing $\omega\in\left(\omega_{0},\omega_{0}/2\right)$ and letting
$N\to\infty$ ($\omega$ does not depend on $N$), we obtain $\nabla^{2}g=0$.
\end{proof}
\begin{lem}
\label{lem: autospazio generalizzato P0} If $\left|\lambda_{0}\right|=e^{\omega_{0}T}$
and $f\in Ker\left(\lambda_{0}I-P_{0}\right)^{2}\setminus Ker\left(\lambda_{0}I-P_{0}\right)$,
then $f\in W^{2,p}\left(X,\mu_{0}\right)$ and $\text{\eqref{eq: rel nabla quadro}}$
holds with $\lambda=\lambda_{0}$ and $g=f$, for every $\omega>\omega_{0}$.

Therefore $\nabla^{2}f=0$, $\mu_{0}$-almost everywhere in $X$ and
$f$ can be considered a linear functional.
\end{lem}

\begin{proof}
Fix $\lambda_{0}\in\mathbb{C}$ such that $\left|\lambda_{0}\right|=e^{\omega_{0}T}$
and $f$ as in the hypothesis. First note that this implies $f\in W^{2,p}\left(X,\mu_{0}\right)$;
indeed $\left(\lambda_{0}I-P_{0}\right)^{2}f=0$ means
\[
\lambda_{0}^{2}f=2\lambda_{0}P_{0}f-P_{0}^{2}f,
\]
hence $f\in W^{1,p}\left(X,\mu_{0}\right)$ by Proposition $\text{\ref{prop: operatore compatto}}$.
Further, setting for simplicity $g:=\left(\lambda_{0}I-P_{0}\right)f$,
we have by formula $\text{\eqref{eq: formula grad C1}}$ applied to
$W^{1,p}\left(X,\mu_{0}\right)$:
\begin{equation}
\lambda_{0}\nabla f=\nabla g+\nabla P_{0}f=\nabla g+U\left(T,0\right)^{*}\overrightarrow{P_{0,T}}\left(\nabla f\right);\label{eq: rel previncente}
\end{equation}
since $g\in Ker\left(\lambda_{0}I-P_{0}\right)\setminus\left\{ 0\right\} $,
it follows from Lemma $\text{\ref{lem: autospazio di P0}}$ that $g\in W^{2,p}\left(X,\mu_{0}\right)$,
so $\nabla f\in W^{1,p}\left(X,\mu_{0};X\right)$, taking into account
by Proposition $\text{\ref{prop: reg vett}}$.

We now show that:
\begin{equation}
\lambda_{0}^{N}\nabla^{2}f=\nabla^{2}P_{0}^{N}f.\label{eq: rel vincente}
\end{equation}

Since $\nabla^{2}g=0$ $\mu_{0}$-almost everywhere in $X$ (again
by Lemma $\text{\ref{lem: autospazio di P0}}$), relation $\text{\eqref{eq: rel previncente}}$
implies: 
\begin{align}
\lambda_{0}\nabla^{2}f & =\nabla^{2}P_{0}f\nonumber \\
 & =U\left(T,0\right)^{*}\nabla_{\mathcal{L}}\overrightarrow{P_{0,T}}\nabla f\nonumber \\
 & =U\left(T,0\right)^{*}\overrightarrow{\overrightarrow{P_{0,T}}}\left(\nabla^{2}f\right)U\left(T,0\right).\label{eq: rel decisiva}
\end{align}
The expression $U\left(T,0\right)^{*}\overrightarrow{\overrightarrow{P_{0,T}}}\left(\cdot\right)U\left(T,0\right)$
can be evaluated also at $\nabla^{2}P_{0}f$ thanks to the definition
in $\text{\eqref{eq: def Ps,t op nabla Phi Lp}}$, then $\text{\eqref{eq: formula grad vettoriale}}$
can be used thanks to Remark $\text{\ref{rem: su Pst operatoriale}}$,
together with formula $\text{\eqref{eq: formula grad C1}}$. Doing
this we obtain:
\begin{align*}
U\left(T,0\right)^{*}\overrightarrow{\overrightarrow{P_{0,T}}}\left(\nabla^{2}P_{0}f\right)U\left(T,0\right) & =U\left(T,0\right)^{*}\nabla_{\mathcal{L}_{2}}\overrightarrow{P_{0,T}}\left(\nabla P_{0}f\right)\\
 & =\nabla_{\mathcal{L}_{2}}U\left(T,0\right)^{*}\overrightarrow{P_{0,T}}\left(\nabla P_{0}f\right)\\
 & =\nabla^{2}P_{0}^{2}f
\end{align*}
Thus, by the identities in $\text{\eqref{eq: rel decisiva}}$:
\begin{align*}
\nabla^{2}P_{0}^{2}f & =\lambda_{0}U\left(T,0\right)^{*}\overrightarrow{\overrightarrow{P_{0,T}}}\left(\nabla^{2}f\right)U\left(T,0\right)\\
 & =\lambda_{0}^{2}\nabla^{2}f.
\end{align*}
Iterating the procedure we get $\text{\eqref{eq: rel vincente}}$.
Arguing like in the proof of Lemma $\text{\ref{lem: autospazio di P0}}$
we obtain that $\text{\eqref{eq: rel nabla quadro}}$ holds with $\lambda=\lambda_{0}$
and $g=f$, for every $\omega>\omega_{0}$.
\end{proof}
The latter two lemmas imply this
\begin{prop}
\label{prop: autovalori di max mod}The following two facts hold true.

i) $\sigma_{p}\left(U\left(T,0\right)^{*}\right)\cap S_{\mathbb{C}}\left(0,e^{\omega_{0}T}\right)=\sigma_{p}\left(P_{0}\right)\cap S_{\mathbb{C}}\left(0,e^{\omega_{0}T}\right)$;

ii) $\sigma_{s}\left(U\left(T,0\right)^{*}\right)\cap S_{\mathbb{C}}\left(0,e^{\omega_{0}T}\right)\subseteq\sigma_{s}\left(P_{0}\right)$.
\end{prop}

\begin{proof}
i) We already know from point i) in Lemma $\text{\ref{lem: autov. 1}}$
that the inclusion $\sigma_{p}\left(U\left(T,0\right)^{*}\right)\subseteq\sigma_{p}\left(P_{0}\right)$
holds; thus we need to prove that $\sigma_{p}\left(P_{0}\right)\cap S_{\mathbb{C}}\left(0,e^{\omega_{0}T}\right)\subseteq\sigma_{p}\left(U\left(T,0\right)^{*}\right)$.

Let $\lambda_{0}\in\sigma_{p}\left(P_{0}\right)$ with $\left|\lambda_{0}\right|=e^{\omega_{0}T}$.
Then every function $g\in Ker\left(\lambda_{0}I-P_{0}\right)\setminus\left\{ 0\right\} $
has the form $g=\left\langle v,\cdot\right\rangle $, $\mu_{0}$-almost
everywhere in $X$, by Lemma $\text{\ref{lem: autospazio di P0}}$,
and it is easy to verify that $v\in Ker\left(\lambda_{0}I-U\left(T,0\right)^{*}\right)\setminus\left\{ 0\right\} $.
Indeed:
\[
\left\langle \lambda_{0}v,\cdot\right\rangle _{X}=P_{0}\left\langle v,\cdot\right\rangle _{X}=\left\langle v,U\left(T,0\right)\left(\cdot\right)\right\rangle _{X}=\left\langle U\left(T,0\right)^{*}v,\cdot\right\rangle _{X}
\]
(note that we have used again Remark $\text{\ref{rem: polinomiali}}$).

ii) Assume that $\lambda_{0}\notin\sigma_{s}\left(P_{0}\right)$ while
$\lambda_{0}\in\sigma_{p}\left(U\left(T,0\right)^{*}\right)$ and
$\left|\lambda_{0}\right|=e^{\omega_{0}T}$. We show that in this
case $\lambda_{0}$ cannot be semisimple for $U\left(T,0\right)^{*}$.

We have $\lambda_{0}\in\sigma_{p}\left(P_{0}\right)$ by Lemma $\text{\ref{lem: autov. 1}}$,
i); since $P_{0}$ is compact, Remark $\text{\ref{rem:  semisemplici compatto}}$
implies that:
\[
Ker\left(\lambda_{0}I-P_{0}\right)^{2}\subsetneq Ker\left(\lambda_{0}I-P_{0}\right).
\]
Choose $f\in Ker\left(\lambda_{0}I-P_{0}\right)^{2}\setminus Ker\left(\lambda_{0}I-P_{0}\right)$.
Then by Lemmas $\text{\ref{lem: autospazio di P0}}$ and $\text{\ref{lem: autospazio generalizzato P0}}$,
there exist $v,w\in X\setminus\left\{ 0\right\} $ such that $f=\left\langle w,\cdot\right\rangle _{X}$
and $\left(\lambda_{0}I-P_{0}\right)f=\left\langle v,\cdot\right\rangle _{X}$;
further $\left(\lambda_{0}I-P_{0}\right)^{2}f=0$. These relations
between functions transform into the following relations between vectors,
remembering that $P_{0}f=\left\langle U\left(T,0\right)^{*}w,\cdot\right\rangle _{X}$:
\begin{align*}
\left[\lambda_{0}I-U\left(T,0\right)^{*}\right]w & =v,\\
\left[\lambda_{0}I-U\left(T,0\right)^{*}\right]^{2}w & =0.
\end{align*}

This means that $w\in Ker\left(\lambda_{0}I-U\left(T,0\right)^{*}\right)^{2}\setminus Ker\left(\lambda_{0}I-U\left(T,0\right)^{*}\right)$,
so
\[
Im\left(\lambda_{0}I-U\left(T,0\right)^{*}\right)\cap Ker\left(\lambda_{0}I-U\left(T,0\right)^{*}\right)\neq\left\{ 0\right\} .
\]
Therefore, condition $\text{\eqref{eq: scomposizione spazio semisempl}}$
is not met, and $\lambda_{0}$ is an eigenvalue of $U\left(T,0\right)^{*}$
that is not semisimple.
\end{proof}
\begin{prop}
\label{prop: omega uguale omega0 p 1}Assume that $R_{0}>0$ and that
all the elements of $\sigma\left(P_{0}\right)$ of modulus $R_{0}$,
if any, are semisimple for $P_{0}$. Then, there exists a constant
$C>0$, depending only on the problem's data, and a positive integer
$N_{0}$, such that
\[
\left\Vert P_{0,T\mid Im\Pi_{0}}^{N}\right\Vert _{\mathcal{L}\left(L^{p}\left(X,\mu_{0}\right)\right)}\leq CR_{0}^{N}\quad\forall N\geq N_{0}.
\]
\end{prop}

\begin{proof}
For simplicity of notation, denote $P_{0}=P_{0,T\mid Im\Pi_{0}}$.
Since $P_{0}$ is compact, only a finite number of elements of its
spectrum are contained in, say, $\mathbb{C}\setminus B\left(0,\frac{1}{2}R_{0}\right)$,
and all of them are eigenvalues of $P_{0}$.

Denote by $\lambda_{1},\dots,\lambda_{K}$ the eigenvalues of maximum
modulus $R_{0}$; thus we can identify a compact region $\Omega$
of the complex plane such that the only elements of the spectrum of
$P_{0}$ contained in $\Omega$ are $\lambda_{1},\dots,\lambda_{K}$.
We may assume that $\Omega$ is a disjoint union of the form $\cup_{j=1}^{K}B\left(\lambda_{j},\epsilon\right)$,
and denote by $\gamma_{j}$ a simple, counterclockwise parametrization
of $\partial B\left(\lambda_{j},\epsilon\right)$.

We prove the thesis by a recursive construction. Assume that $\mu_{1},\dots,\mu_{K}$
are semisimple eigenvalues for $P_{0,T\mid Im\Pi_{0}}$. 

\emph{Step 0}. Define:
\begin{align*}
 & E_{0}:=Im\Pi_{0},\\
 & P_{0}:=P_{0,T\mid E_{0}},\\
 & \mathcal{P}_{0}:=\mathcal{P}\left(\Omega,P_{0}\right).
\end{align*}
We have $E_{0}=Im\mathcal{P}_{0}\oplus Ker\mathcal{P}_{0}$. We have,
by the previous considerations:

\begin{equation}
\sigma\left(P_{0\mid Ker\mathcal{P}_{0}}\right)=\sigma\left(P_{0}\right)\setminus\left\{ \mu_{1},\dots,\mu_{K}\right\} ,\label{eq: spettro modulo basso}
\end{equation}
remembering that $Ker\mathcal{P}_{0}$ is a closed and $P_{0}$- invariant
subspace of $E_{0}$. Thus $Ker\mathcal{P}_{0}$ is a Banach spaces
and $P_{0\mid Ker\mathcal{P}_{0}}$ is a compact operator acting on
it.

\emph{Step 1}: Define:
\begin{align*}
 & E_{1}:=Im\mathcal{P}_{0},\\
 & P_{1}:=P_{0\mid E_{1}},\\
 & \mathcal{P}_{1}:=\mathcal{P}\left(\mu_{1},P_{1}\right).
\end{align*}
Due to the previous considerations, $E_{1}$ is a closed, $P_{0}$-invariant
subspace of $E_{0}$. Thus $E_{1}$ is a Banach space and $P_{1}\in\mathcal{L}\left(E_{1}\right)$
is a compact operator, and $\mathcal{P}_{1}\in\mathcal{L}\left(E_{1}\right)$.
Further, 
\[
\sigma\left(P_{1}\right)=\left\{ \mu_{1},\dots,\mu_{K}\right\} ,
\]
and all these eigenvalues can be proved to be semisimple for $P_{1}$.
Indeed, for every fixed $i=1,\dots,K$:
\begin{align*}
\left(\mu_{i}I-P_{1}\right)\mathcal{P}\left(\mu_{i},P_{1}\right) & =\left(\mu_{i}I-P_{1}\right)\int_{\gamma\left(\mu_{i}\right)}R\left(\xi,P_{1}\right)d\xi\\
 & =\left(\mu_{i}I-P_{1}\right)\int_{\gamma\left(\mu_{i}\right)}R\left(\xi,P_{0}\right)_{\mid E_{1}}d\xi\\
 & =\left(\mu_{i}I-P_{0}\right)_{\mid E_{1}}\mathcal{P}\left(\mu_{i},P_{0}\right)_{\mid E_{1}}\\
 & =\left[\left(\mu_{i}I-P_{0}\right)\mathcal{P}\left(\mu_{i},P_{0}\right)\right]_{\mid E_{1}}.
\end{align*}
The second key equality holds because $E_{1}$ is a closed and $P_{0}$-invariant
subspace of $E_{0}$ which is complemented in $E_{0}$ by an analogous
subspace. Using the characterization of semisimplicity by the spectral
projection, we have that $\left(\mu_{i}I-P_{0}\right)\mathcal{P}\left(\mu_{i},P_{0}\right)=0$
because $\mu_{i}$ is semisimple for $P_{0}$, and thus the same holds
for $P_{1}$ in place of $P_{0}$, which means that $\mu_{i}$ is
semisimple for $P_{1}$.

In particular, using the second characterization of semisimplicity,
applied to $\mu_{1}$ and $P_{1}$, we have:
\[
E_{1}=Ker\left(\mu_{1}I-P_{1}\right)\oplus Ker\mathcal{P}_{1}.
\]

\emph{Recursive-inductive step}: For $1\leq j\leq k-1$, assume that
$E_{j}$ is a Banach space, that $P_{j}\in\mathcal{L}\left(E_{j}\right)$
is compact with $\sigma\left(P_{j}\right)=\left\{ \mu_{j},\dots,\mu_{K}\right\} $
, and that $\mu_{j},\dots,\mu_{K}$ are semisimple eigenvalues of
$P_{j}$. Denoting $\mathcal{P}_{j}:=\mathcal{P}\left(\mu_{j},P_{j}\right)$,
we have $\mathcal{P}_{j}\in\mathcal{L}\left(E_{j}\right)$ and 
\begin{equation}
E_{j}=Ker\left(\mu_{j}I-P_{j}\right)\oplus Ker\mathcal{P}_{j}.\label{eq: ip ind E_j}
\end{equation}

Define:
\begin{align*}
 & E_{j+1}:=Ker\mathcal{P}_{j},\\
 & P_{j+1}:=P_{j\mid E_{j+1}},\\
 & \mathcal{P}_{j+1}:=\mathcal{P}\left(\mu_{j+1},P_{j+1}\right).
\end{align*}

The space $E_{j+1}$ is a closed and $P_{j}$-invariant subspace of
$E_{j}$, so $P_{j+1}\in\mathcal{L}\left(E_{j+1}\right)$ is compact
and
\begin{align*}
\sigma\left(P_{j+1}\right) & =\sigma\left(P_{j\mid Ker\mathcal{P}\left(\mu_{j},P_{j}\right)}\right)\\
 & =\sigma\left(P_{j}\right)\setminus\left\{ \mu_{j}\right\} \\
 & =\left\{ \mu_{j+1},\dots,\mu_{K}\right\} .
\end{align*}

With the same argument as in Step 1, we see that $\mu_{j+1},\dots,\mu_{K}$
are semisimple eigenvalues of $P_{j+1}$; thus, in particular:
\[
E_{j+1}=Ker\left(\mu_{j+1}I-P_{j+1}\right)\oplus Ker\mathcal{P}_{j+1}.
\]

With this procedure we have proved by induction that, for every integer
$j\in\left[1,\dots,K\right]$, there exist a closed subspace $E_{j}$
of $E_{0}$, a compact operator $P_{j}$ on $E_{j}$ whose spectrum
$\sigma\left(P_{j}\right)=\left\{ \mu_{j},\dots,\mu_{K}\right\} $
is made of semisimple eigenvalues, such that $\text{\eqref{eq: ip ind E_j}}$
holds with $\mathcal{P}_{j}=\mathcal{P}\left(\mu_{j},P_{j}\right)$.
Further, $E_{j+1}=Ker\mathcal{P}_{j}\subseteq E_{j}$ and $P_{j+1}=P_{j\mid E_{j+1}}$,
for every $j\in\left[1,\dots,K-1\right]$.

\emph{Final step}: Applying the latter statement to $j=K$, we have
the following situation: $P_{K}\in\mathcal{L}\left(E_{K}\right)$
is a compact operator of the Banach space $E_{K}$ with $\sigma\left(P_{K}\right)=\left\{ \mu_{K}\right\} $
, $\mu_{K}$ is semisimple for $P_{K}$ and, denoting $\mathcal{P}_{K}:=\mathcal{P}\left(\mu_{K},P_{K}\right)$:
\[
E_{K}=Ker\left(\mu_{K}I-P_{K}\right)\oplus Ker\mathcal{P}_{K}.
\]
By the spectral separation property of the spectral projection, we
have that $\sigma\left(P_{K\mid Ker\mathcal{P}_{K}}\right)=\emptyset$;
since $P_{K}$ is bounded this implies that $E_{K+1}:=Ker\mathcal{P}_{K}=\left\{ 0\right\} $.

Applying relation $\text{\eqref{eq: ip ind E_j}}$ forward from $j=1$
to $j=k$ we obtain:
\begin{align}
E_{1} & =Ker\left(\mu_{1}I-P_{1}\right)\oplus E_{2}\nonumber \\
 & =Ker\left(\mu_{1}I-P_{1}\right)\oplus Ker\left(\mu_{2}I-P_{2}\right)\oplus E_{3}\nonumber \\
 & \vdots\nonumber \\
 & =\bigoplus_{j=1}^{K}Ker\left(\mu_{j}I-P_{j}\right)\bigoplus E_{K+1}\nonumber \\
 & =\bigoplus_{j=1}^{K}Ker\left(\mu_{j}I-P_{j}\right).\label{eq: E_1 =00003D somma autospazi}
\end{align}
Now fix $g\in E_{1}$. Then there exist unique $g_{1},\dots,g_{K}$,
with $g_{j}\in Ker\left(\mu_{j}I-P_{j}\right)$ for every $j=1\dots,K$,
such that:
\[
g=g_{1}+\dots+g_{K}.
\]

Fix a $m\in\left\{ 1,\dots,K-1\right\} $ . Then
\[
Ker\left(\mu_{m}I-P_{m}\right)=Im\mathcal{P}_{m}
\]

because $\mu_{m}$ is semisimple for $P_{m}$. Thus, since $g_{m}+\dots+g_{K}\in E_{m}$,
$g_{m}\in Im\mathcal{P}_{m}$ and $g_{m+1}+\dots+g_{K}\in Ker\mathcal{P}_{m}=Im\left(I-\mathcal{P}_{m}\right)$
, we have:
\begin{align*}
g_{m} & =\mathcal{P}_{m}\left(g_{m}+\dots+g_{K}\right),\\
g_{m+1}+\dots+g_{K} & =\left(I-\mathcal{P}_{m}\right)\left(g_{m}+\dots+g_{K}\right).
\end{align*}
Thus, inductively, for every $j=1,\dots,.K$:
\begin{align*}
g_{j} & =\mathcal{P}_{j}\left(g_{j}+\dots+g_{K}\right)\\
 & =\mathcal{P}_{j}\left(I-\mathcal{P}_{j-1}\right)\left(g_{j-1}+\dots+g_{K}\right)\\
 & =\mathcal{P}_{j}\left(I-\mathcal{P}_{j-1}\right)\left(I-\mathcal{P}_{j-2}\right)\left(g_{j-2}+\dots+g_{K}\right)\\
 & \vdots\\
 & =\mathcal{P}_{j}\bigcirc_{m=1}^{j-1}\left(I-\mathcal{P}_{j-m}\right)g.
\end{align*}

Further, since $P_{j}$ is a restriction of $P_{1}$, $P_{1}g_{j}=P_{j}g_{j}=\mu_{j}g_{j}$.
Now fix $N\in\mathbb{N}$; thanks to the latter decomposition, we
have
\begin{align}
P_{1}^{N}g & =P_{1}^{N}\sum_{j=1}^{K}g_{j}=\sum_{j=1}^{K}\mu_{j}^{N}g_{j}\nonumber \\
 & =\sum_{j=1}^{K}\mu_{j}^{N}\mathcal{P}_{j}\bigcirc_{m=1}^{j-1}\left(I-\mathcal{P}_{j-m}\right)g.\label{eq: pot P_1 su Im pro}
\end{align}
Before concluding, observe that, by $\text{\eqref{eq: spettro modulo basso}}$,
the spectral radius of $P_{0\mid Ker\mathcal{P}_{0}}$ is strictly
less than $R_{0}$; thus, the formula for the spectral radius implies
that there exists $N_{0}\in\mathbb{N}$ such that
\begin{equation}
\left\Vert P_{0\mid Ker\mathcal{P}_{0}}^{N}\right\Vert \leq R_{0}^{N}\quad\forall N\geq N_{0}.\label{eq: tasso conv P_0 su Ker pro}
\end{equation}

Finally, let $f\in E_{0}$ and $N\geq N_{0}$. Since $\mathcal{P}_{0}f\in E_{1}$,
we have by $\text{\eqref{eq: pot P_1 su Im pro}}$:
\begin{align*}
P_{0}^{N}f= & P_{0}^{N}\left\{ \mathcal{P}_{0}f+\left(I-\mathcal{P}_{0}\right)f\right\} \\
= & P_{1}^{N}\mathcal{P}_{0}f+P_{0}^{N}\left(I-\mathcal{P}_{0}\right)f\\
= & \sum_{j=1}^{K}\mu_{j}^{N}\mathcal{P}_{j}\bigcirc_{m=1}^{j-1}\left(I-\mathcal{P}_{j-m}\right)\mathcal{P}_{0}f+P_{0}^{N}\left(I-\mathcal{P}_{0}\right)f.
\end{align*}
Observe that every $\mathcal{P}_{j}\bigcirc_{m=1}^{j-1}\left(I-\mathcal{P}_{j-m}\right)\mathcal{P}_{0}$
belongs to $\mathcal{L}\left(E_{0}\right)$. Take $C>0$ bigger than
the quantities $\sum_{j=1}^{K}\left\Vert \mathcal{P}_{j}\bigcirc_{m=1}^{j-1}\left(I-\mathcal{P}_{j-m}\right)\mathcal{P}_{0}\right\Vert _{\mathcal{L}\left(E_{0}\right)}$
and $\left\Vert \left(I-\mathcal{P}_{0}\right)\right\Vert _{\mathcal{L}\left(E_{0}\right)}$,
where the norm in $E_{0}$ is the one inherited by $L^{p}\left(X,\mu_{0}\right)$.
So $C$ depends only on the $\mathcal{P}_{j}$'s and, therefore, on
the problem's data. In particular, $C$ does not depend neither on
$N_{0}$ or on $N$.

Passing to the norms in the previous identity and remembering that
$\left(I-\mathcal{P}_{0}\right)f\in Ker\mathcal{P}_{0}$ we obtain
from $\text{\eqref{eq: tasso conv P_0 su Ker pro}}$:
\begin{align*}
\left\Vert P_{0}^{N}f\right\Vert _{L^{p}\left(X,\mu_{0}\right)}= & \left\Vert P_{0}^{N}f\right\Vert _{E_{0}}\\
= & \left\Vert \sum_{j=1}^{K}\mu_{j}^{N}\mathcal{P}_{j}\bigcirc_{m=1}^{j-1}\left(I-\mathcal{P}_{j-m}\right)\mathcal{P}_{0}f+P_{0}^{N}\left(I-\mathcal{P}_{0}\right)f\right\Vert _{E_{0}}\\
\leq & R_{0}^{N}\Biggl\{\sum_{j=1}^{K}\left\Vert \mathcal{P}_{j}\bigcirc_{m=1}^{j-1}\left(I-\mathcal{P}_{j-m}\right)\mathcal{P}_{0}\right\Vert _{\mathcal{L}\left(E_{0}\right)}\left\Vert f\right\Vert _{E_{0}}\\
 & +\left\Vert \left(I-\mathcal{P}_{0}\right)\right\Vert _{\mathcal{L}\left(E_{0}\right)}\left\Vert f\right\Vert _{E_{0}}\Biggr\}\\
 & \leq CR_{0}^{N}\left\Vert f\right\Vert _{L^{p}\left(X,\mu_{0}\right)}.
\end{align*}
This concludes the proof.
\end{proof}
Now we can formalize the proof of the main result of the subsection.
\begin{proof}[Proof of Theorem $\text{\ref{teo: omega uguale omega0}}$]
Let the system satisfy Assumptions $\text{\ref{assu: periodicita}}$
and $\text{\ref{ass: U* ha autov di mod max}}$, and suppose that
all the eigenvalues of $U\left(T,0\right)^{*}$ of modulus $e^{\omega_{0}T}$
are semisimple.

We show that the hypotheses of Proposition $\text{\ref{prop: omega uguale omega0 p 1}}$
are satisfied. First, since Assumption $\text{\ref{ass: U* ha autov di mod max}}$
holds, it follows from Lemma $\text{\ref{lem: autov. 1}}$, point
ii) that $R_{0}=e^{\omega_{0}T}>0$. Further, let $\lambda$ be an
eigenvalue of $P_{0}$ such that $\left|\lambda\right|=R_{0}$; by
point i) in Proposition $\text{\ref{prop: autovalori di max mod}}$,
$\lambda\in\sigma_{p}\left(U\left(T,0\right)^{*}\right)$. Thus, since
$\lambda$ is semisimple for $U\left(T,0\right)^{*}$ by assumption,
$\lambda$ is semisimple for $P_{0}$ due to point ii) in the same
Proposition $\text{\ref{prop: autovalori di max mod}}$. So the hypotheses
of Proposition $\text{\ref{prop: omega uguale omega0 p 1}}$ are satisfied.
Hence the desired estimates hold, the second one being obtained from
the first one exactly the same way $\text{\eqref{eq: stima Ps,tf - mtf}}$
was obtained by $\text{\eqref{eq: stima potenza P_0,T}}$ in the proof
of Theorem $\text{\ref{teo omega>omega0}}$.
\end{proof}
\section{Example}\label{sez: esempio}

%`

The primary data of this paper are a Hilbert space $X$ and two families
$\left\{ A\left(t\right)\mid t\in\mathbb{R}\right\} $, $\left\{ B\left(t\right)\mid t\in\mathbb{R}\right\} $
of linear operators on $X$, which must satisfy the assumptions listed
at the end of Subsection $\text{\ref{subsec: setting}}$.

In this section we provide an example where $X=L^{2}\left(\Omega\right)$
(being $\Omega$ a bounded domain of $\mathbb{R}^{N}$), $A\left(t\right)$
is a differential operator on $H^{1}\left(\Omega\right)$ and $B\left(t\right)$
is a proper fractional power of $A\left(t\right)^{-1}$, and we verify
that these objects satisfy the assumptions of the paper, except for
the periodicity of the system which is obtained just by imposing that
the coefficients in the operator $A\left(t\right)$ are periodic.
Moreover, the example meets the hypotheses of Theorem $\text{\ref{teo: omega uguale omega0}}$.
The latter feature is important because it ensures that the example
is optimal with regard to the rate of convergence of the Ornstein-Uhlenbeck
operator $P_{s,t}$. In other terms, in our example the speed of converge
of $P_{s,t}$ to $\int_{X}fd\mu_{t}$ for $t-s\to\infty$ is the highest
possible.

The operators $A\left(t\right):D\left(A\left(t\right)\right)\subseteq X\to X$
will be defined as elliptic differential operators agreeing with the
hypotheses of Acquistapace-Terreni; this guarantees the existence
of an evolution operator with the expected properties. As far as the
operators $B\left(t\right)$ are concerned, the choice of a fractional
power of $A\left(t\right)^{-1}$ is similar to that made in \cite{BigDef23},
and exploits a technique intrduced in \cite{Cer21}. This approach
also makes use of some basic properties of $\left\{ U\left(t,s\right)\mid s<t\right\} $
guaranteed in the Acquistapace-Terreni framework.

In the first part of the section, we will prove that the assumptions
regarding the sole $U\left(t,s\right)$ - thus, the sole $A\left(t\right)$
- are satisfied. We first check the hypotheses of Acquistapace-Terreni,
which is not trivial and seems to be new for an operator defined in
a subspace of $H^{1}\left(\Omega\right)$. This technique does not
exploit any maximum regularity result and may be used for more general
purposes. Then we complete the check of Assumption $\text{\ref{assu: omega_0 neg}}$
by proving the exponential decay of the evolution operator. Finally
we show that also Assumption $\text{\ref{ass: U* ha autov di mod max}}$,
as well as the other crucial hypothesis of Theorem $\text{\ref{teo: omega uguale omega0}}$,
hold true for $A\left(t\right)$ as we defined it.

In the second part of the section, we prove that our choice for $B\left(t\right)$
satisfies the basic Assumptions $\text{\ref{assu: B(t) limitati}}$
and $\text{\ref{assu: mu_t non degeneri}}$, and that it is coherent
with the ``controllability condition'' prescribed by Assumption $\text{\ref{ass: controllabilita}}$.
We prove the validity of Assumption $\text{\ref{assu: tracce finite}}$
in a similar way as in \cite{Cer21}.

Summing up, the scheme of this section is the following.

\begin{eqnarray*}
\text{Subsection }\text{\ref{subsec: A(t)}}: & \quad & \begin{aligned}\triangleright\  & \text{Assumption }\text{\ref{assu: omega_0 neg}}\\
\triangleright\  & \text{Assumption }\text{\ref{ass: U* ha autov di mod max}}\\
\triangleright\  & \text{Hypotheses of Theorem }\text{\ref{teo: omega uguale omega0}}
\end{aligned}
\\
\\
\\
\text{Subsection }\text{\ref{subsec: B(t)}} & \quad & \begin{aligned}\triangleright\  & \text{Assumption }\text{\ref{assu: B(t) limitati}}\\
\triangleright\  & \text{Assumption }\text{\ref{assu: tracce finite}}\\
\triangleright\  & \text{Assumption }\text{\ref{ass: controllabilita}}\\
\triangleright\  & \text{Assumption }\text{\ref{assu: mu_t non degeneri}}
\end{aligned}
\end{eqnarray*}

\subsection{The operators \textit{A}($\cdot$)}\label{subsec: A(t)} 

Let $N\in\mathbb{N}$ and $\Omega\subseteq\mathbb{R}^{N}$ be a bounded
domain with $\mathcal{C}^{2}$ boundary and let $\nu$ be the outward
unit vector normal to $\partial\Omega$. We set $X:=L^{2}\left(\Omega\right)$.
For $i,j=1,\dots,N$ we choose $a_{ij}:\mathbb{R}\times\overline{\Omega}\to\mathbb{R}$
to be real functions that are uniformly bounded in $i,j$ and satisfy
the uniform H\"older condition:
\begin{align}
 & \left|a_{ij}\left(t,x\right)-a_{ij}\left(s,x\right)\right|\leq C\left|t-s\right|^{\alpha}\label{eq: aij holder}\\
 & \forall t,s\in\mathbb{R},\ x\in\Omega,\ i,j=1,\dots,N\nonumber 
\end{align}
for some fixed $\alpha\in\left(1/2,1\right)$. We assume that, for
every $t\in\mathbb{R}$ and $x\in\overline{\Omega}$, the matrix $\left[a_{ij}\left(t,x\right)\right]_{i,j=1}^{N}$
is symmetric and satisfies a common strong ellipticity assumption;
namely, for some $\eta>0$:
\begin{equation}
\Re\sum_{i,j=1}^{N}y_{i}a_{ij}\left(t,x\right)y_{j}\geq\eta\left|y\right|_{\mathbb{C}^{N}}^{2}\quad\forall y\in\mathbb{C}^{N}.\label{eq: matrice coerciva}
\end{equation}

For every fixed $t\in\mathbb{R}$ and $g\in X$, we consider the following
second order elliptic equation in divergence form:
\begin{equation}
\begin{cases}
{\displaystyle \sum_{i,j=1}^{N}}D_{j}\left(a_{ij}\left(t,\cdot\right)D_{i}\mathrm{u}\right)\left(x\right)-\mathrm{u}\left(x\right)=g\left(x\right) & \quad x\in\Omega\\
{\displaystyle \sum_{i,j=1}^{N}}a_{ij}\left(t,x\right)D_{i}\mathrm{u}\left(x\right)\nu_{j}\left(x\right)=0 & \quad x\in\partial\Omega
\end{cases}\label{eq: problema ellittico}
\end{equation}
in the unknown $\mathrm{u}$. Every such problem has a unique weak
solution in $H^{1}\left(\Omega\right)$; for an extensive treatment
of this kind of problems see for instance \cite{GIL-TRUD}.

Define, for every $t\in\mathbb{R}$:
\[
a\left(t\right)\left(u,v\right):=-\int_{\Omega}\left\{ \sum_{i,j=1}^{N}D_{i}u\left(x\right)a_{ij}\left(t,x\right)D_{j}v\left(x\right)+u\left(x\right)v\left(x\right)\right\} dx\quad\forall u,v\in H^{1}\left(\Omega\right).
\]
Clearly $a\left(t\right)\left(\cdot,\cdot\right)$ is a symmetric,
positive-definite bilinear form in $H^{1}\left(\Omega\right)$. Set
\[
D\left(A\left(t\right)\right):=\left\{ u\in H^{1}\left(\Omega\right)\mid\exists f\in X\text{ such that }a\left(t\right)\left(u,v\right)=\left\langle f,v\right\rangle _{X}\quad\forall v\in H^{1}\left(\Omega\right)\right\} ,
\]
and, for every $u\in D\left(A\left(t\right)\right)$, $v\in H^{1}\left(\Omega\right)$:
\[
\left\langle A\left(t\right)u,v\right\rangle _{X}:=a\left(t\right)\left(u,v\right).
\]
In other terms, $A\left(t\right)$ is the weak realization in $X$
of the differential operator:
\begin{equation}
\mathcal{A}\left(t\right)u:=\sum_{i,j=1}^{N}D_{j}\left(a_{ij}\left(t,\cdot\right)D_{i}u\right)-u\label{eq: op differenziale}
\end{equation}
with conormal derivative boundary conditions as in $\text{\eqref{eq: problema ellittico}}$.
\begin{rem}
The abstract, weak formulation of $\text{\eqref{eq: problema ellittico}}$
is given by:
\[
\begin{cases}
\left\langle A\left(t\right)\mathrm{u},v\right\rangle _{X}=\left\langle g,v\right\rangle _{X} & \quad\forall v\in X\\
\mathrm{u}\in D\left(A\left(t\right)\right).
\end{cases}
\]
This abstract problem has memory of the conormal derivative boundary
condition because in the definition of $D\left(A\left(t\right)\right)$
the test function $v$ varies in $H^{1}\left(\Omega\right)$. In order
to represent the weak problem with Dirichlet boundary conditions,
one just needs to replace the expression ``$\forall v\in H^{1}\left(\Omega\right)$''
with the expression ``$\forall v\in H_{0}^{1}\left(\Omega\right)$''
in the definition of $D\left(A\left(t\right)\right)$. For this reason
all the following results also hold with exactly the same arguments
if $D\left(A\left(t\right)\right)$ is such as to represent the Dirichlet
problem.
\end{rem}

We prove three basic estimates relative to the elements of $D\left(A\left(t\right)\right)$.
First, it is an immediate consequence of the uniform ellipticity property
in $\text{\eqref{eq: matrice coerciva}}$ that, for every $t\in\mathbb{R}$,
the operator $A\left(t\right)$ is dissipative:
\begin{equation}
\left\langle A\left(t\right)u,u\right\rangle _{X}\leq-\left|u\right|_{X}^{2}\quad\forall u\in D\left(A\left(t\right)\right).\label{eq: forma dissipativa}
\end{equation}

Note that the above relation implies $\left|u\right|_{X}\leq\left|A\left(t\right)u\right|_{X}$
for every $u\in D\left(A\left(t\right)\right)\setminus\left\{ 0\right\} $
so the operator $A\left(t\right):D\left(A\left(t\right)\right)\subseteq X\to X$
is injective; further, it is surjective because of the existence of
weak solutions of $\text{\eqref{eq: problema ellittico}}$ for every
$g\in X$. We conclude that $0$ belongs to the resolvent set of $A\left(t\right)$
and
\begin{equation}
\left\Vert A\left(t\right)^{-1}\right\Vert _{\mathcal{L}\left(X\right)}\leq1.\label{eq: stima A(t)^-1}
\end{equation}

Relation $\text{\eqref{eq: forma dissipativa}}$ can be improved for
two special subclasses of $X$, namely that of the constants and that
of the functions whose integral mean in $\Omega$ - with respect to
the Lebesgue measure $\lambda$ - is zero. Indeed, for every $t\in\mathbb{R}$,
a function that is almost everywhere constant in $\Omega$ belongs
to $D\left(A\left(t\right)\right)$, and:
\begin{equation}
\left\langle A\left(t\right)u,v\right\rangle _{X}=-\left\langle u,v\right\rangle _{X}\quad\forall u\text{ a.e. constant in }\Omega,\ v\in H^{1}\left(\Omega\right).\label{eq: forma su costanti}
\end{equation}

For every $u\in L^{1}\left(\Omega\right)$, denote
\[
u_{\Omega}:=\frac{1}{\lambda\left(\Omega\right)}\int_{\Omega}u\left(x\right)dx.
\]
By Poincar\'e's inequality, there exists a positive constant $C_{\Omega}$
such that, for every $t\in\mathbb{R}$:
\begin{align}
\left\langle A\left(t\right)u,u\right\rangle _{X} & \leq-\int_{\Omega}\left\{ \eta\left|\nabla u\left(x\right)\right|_{\mathbb{C}^{N}}^{2}+\left|u\left(x\right)\right|^{2}\right\} dx\nonumber \\
 & \leq-\left(1+\eta C_{\Omega}\right)\left|u\right|_{X}^{2}\quad\forall u\in D\left(A\left(t\right)\right)\text{ with }u_{\Omega}=0.\label{eq: forma su media nulla}
\end{align}

Finally, we observe that it follows from $\text{\eqref{eq: stima A(t)^-1}}$
and the first inequality in $\text{\eqref{eq: forma su media nulla}}$
that every $f\in X$:
\begin{align}
\eta\int_{\Omega}\left|\nabla A\left(t\right)^{-1}f\left(x\right)\right|_{\mathbb{C}^{N}}^{2}dx & \leq\left|f\right|_{X}^{2}.\label{eq: norma X grad inversa}
\end{align}

\vspace{2mm}
\begin{fact}
The family $\left\{ A\left(t\right)\mid t\in\mathbb{R}\right\} $
satisfies Assumption $\text{\ref{assu: omega_0 neg}}$, i). Precisely,
it generates an evolution operator $\left\{ U\left(t,s\right)\mid s\leq t\right\} $
such that, for every $t>s$, $ImU\left(t,s\right)\subseteq D\left(A\left(t\right)\right)$,
$A\left(t\right)U\left(t,s\right)\in\mathcal{L}\left(X\right)$ and
\begin{equation}
\left\Vert A\left(t\right)U\left(t,s\right)\right\Vert _{\mathcal{L}\left(X\right)}\leq\frac{c}{t-s}.\label{eq: A(t)U(t,s)}
\end{equation}
\end{fact}

We prove that $\left\{ A\left(t\right)\right\} _{t\in\mathbb{R}}$
satisfies, in the whole real line, the two hypotheses of Acquistapace-Terreni
stated in \cite{Acq - Ev op}, \S 1 for the case $t\in\left[0,T\right]$.
By Theorem 2.3 ($v)$ in the same paper, $\left\{ A\left(t\right)\right\} _{t\in\mathbb{R}}$
generates a unique evolution operator $\left\{ U\left(t,s\right)\right\} _{0\leq s\leq t\leq T}$
satisfying the above requirements. The extension of $U\left(\cdot,\cdot\right)$
to the half plane $\left\{ \left(s,t\right)\mid s\leq t\right\} $
is straightforward due to the uniqueness of the solution of $\text{\eqref{eq: evoluzione}}$.

\vspace{2mm}

\textsc{Hypothesis I}. We have to show that there exists $c>0$ and
a sector of the complex plane of the form $S_{\vartheta_{0}}=\left\{ z\in\mathbb{C}\mid\left|\arg z\right|\leq\vartheta_{0}\right\} $,
with $\vartheta_{0}\in\left(\frac{\pi}{2},\pi\right]$, contained
in the resolvent set of every $A\left(t\right)$, such that
\begin{equation}
\left\Vert R\left(\lambda,A\left(t\right)\right)\right\Vert _{\mathcal{L}\left(X\right)}\leq\frac{c}{1+\left|\lambda\right|}\quad\forall\lambda\in\overline{S_{\vartheta_{0}}},\ t\in\mathbb{R}.\label{eq: AT1}
\end{equation}

For every $\lambda\in\mathbb{C}$, $t\in\mathbb{R}$ and $u\in D\left(A\left(t\right)\right)\setminus\left\{ 0\right\} $
we have
\begin{align*}
\left\langle \left[\lambda I-A\left(t\right)\right]u,u\right\rangle _{X} & =\lambda\left|u\right|_{X}^{2}-a\left(t\right)\left(u,u\right)\\
 & =\left(\lambda+1\right)\left|u\right|_{X}^{2}+\int_{\Omega}\left\{ \sum_{i,j=1}^{N}D_{i}u\left(x\right)a_{ij}\left(t,x\right)D_{j}u\left(x\right)\right\} dx.
\end{align*}
Passing to the imaginary part $\Im$ and to the real part $\Re$ separately,
and considering that the last addend is a non-negative real number
by $\text{\eqref{eq: matrice coerciva}}$, we see that the quantities
$\left|\Im\lambda\right|\left|u\right|_{X}^{2}$ and $\left(\Re\lambda+1\right)\left|u\right|_{X}^{2}$
are bounded above by $\left|\left[\lambda I-A\left(t\right)\right]u\right|_{X}\left|u\right|_{X}$.
Summing up we obtain:
\[
\left(\left|\Im\lambda\right|+\Re\lambda+1\right)\left|u\right|_{X}\leq2\left|\left[\lambda I-A\left(t\right)\right]u\right|_{X}.
\]
Hence, provided that $\left|\Im\lambda\right|+\Re\lambda+1>0$, $\lambda$
belongs to the resolvent set of $A\left(t\right)$ and:
\begin{equation}
\left\Vert R\left(\lambda,A\left(t\right)\right)\right\Vert _{\mathcal{L}\left(X\right)}\leq\frac{2}{\left|\Im\lambda\right|+\Re\lambda+1}.\label{eq: stima ris settore grande}
\end{equation}

Note that the inequality $\left|\Im\lambda\right|+\Re\lambda+1>0$
configures a complex sector, precisely the sector $S_{\frac{3}{4}\pi}-1$,
which contains $\overline{S_{\frac{3}{4}\pi}}$; but we need the estimate
in $\eqref{eq: AT1}$ - which is sharper than $\text{\eqref{eq: stima ris settore grande}}$
if $\Re\lambda<0$ - at the cost of taking a $\vartheta_{0}<\frac{3}{4}\pi$
(provided that $\vartheta_{0}$ stays grater than $\frac{\pi}{2}$).

To this end, start by observing that for every $\vartheta_{0}\in\left(\frac{\pi}{2},\pi\right]$
the sector $\overline{S_{\vartheta_{0}}}$ can be written as
\begin{align}
\overline{S_{\vartheta_{0}}} & =\left\{ z\in\mathbb{C}\mid\Re z-\frac{1}{\tan\vartheta_{0}}\left|\Im z\right|\geq0\right\} ,\label{eq: caratt sett}
\end{align}
since $\tan\vartheta_{0}<0$. 
\begin{lem}
\label{rem: lemmino settore}For every $c_{1}>0$ and $c_{2}>1$ such
that $c_{1}c_{2}>1$, there exists $\vartheta_{0}\in\left(\frac{\pi}{2},\pi\right]$
such that
\begin{equation}
\overline{S_{\vartheta_{0}}}\subseteq\left\{ z\in\mathbb{C}\mid\frac{1+\left|z\right|}{c_{2}}\leq1+\Re z+c_{1}\left|\Im z\right|\right\} .\label{eq: settore incluso in stima buona}
\end{equation}
\end{lem}

\begin{proof}
The sketched proof of the statement is as follows. Clearly there exists
$\epsilon>0$ such that $\left(c_{1}-\epsilon\right)c_{2}-\sqrt{1+\epsilon^{2}}>0$,
and for this choice of $\epsilon$ we have
\[
\left[\left(c_{1}-\epsilon\right)c_{2}-\sqrt{1+\epsilon^{2}}\right]\left|y\right|>1-c_{2}\quad\forall y\in\mathbb{R}.
\]
A simple computation shows that in this case
\[
-\epsilon\left|y\right|\leq x<0\implies1+\sqrt{x^{2}+y^{2}}\leq c_{2}\left(1+x+c_{1}\left|y\right|\right),
\]
while obviously the latter inequality holds for $x\geq0$ as well.
Thus, for every $x$, $y$ $\in\mathbb{R}$:
\[
x+\epsilon\left|y\right|\geq0\implies1+\sqrt{x^{2}+y^{2}}\leq c_{2}\left(1+x+c_{1}\left|y\right|\right).
\]
In other words this means that $\text{\eqref{eq: settore incluso in stima buona}}$
holds, having set $\vartheta_{0}:=\arctan\left(-\frac{1}{\epsilon}\right)+\pi$
and remembering $\text{\eqref{eq: caratt sett}}$. Note that, as $\epsilon\to0^{+}$,
$\arctan\left(-\frac{1}{\epsilon}\right)\to-\frac{\pi}{2}^{+}$ so
we can assume that $\vartheta_{0}\in\left(\frac{\pi}{2},\pi\right]$.
\end{proof}
To fix the ideas, we could choose for instance $c_{1}=1$, $c_{2}=4$
and $\epsilon=-\frac{1}{\tan\left(\frac{2}{3}\pi\right)}\thickapprox0,58$,
thus $\vartheta_{0}=\frac{2}{3}\pi$ and we would obtain:
\[
\overline{S_{\frac{2}{3}\pi}}\subseteq\left\{ z\in\mathbb{C}\mid1+\left|z\right|\leq4\left(1+\Re z+\left|\Im z\right|\right)\right\} .
\]

Taking $c_{1}=1$ and $c_{2}$, $\vartheta_{0}$ as in Lemma $\text{\ref{rem: lemmino settore}}$,
we obtain from $\text{\eqref{eq: stima ris settore grande}}$ that
$\text{\eqref{eq: AT1}}$ holds with $c=2c_{2}$.

\vspace{2mm}

\textsc{Hypothesis II}. As far as the second hypothesis of Acquistapace-Terreni
is concerned, we prove that there exists a different angle $\vartheta_{0}\in\left(\frac{\pi}{2},\pi\right]$
such that:
\begin{equation}
\left\Vert A\left(t\right)R\left(\lambda,A\left(t\right)\right)\left[A\left(t\right)^{-1}-A\left(s\right)^{-1}\right]\right\Vert _{\mathcal{L}\left(X\right)}\leq C\frac{\left|t-s\right|^{\alpha}}{\sqrt{1+\left|\lambda\right|}}\quad\forall\lambda\in\overline{S_{\vartheta_{0}}}\setminus\left\{ 0\right\} ,\ s<t,\label{eq: AT2}
\end{equation}
which is slightly stronger than Hypothesis II stated in \cite{Acq - Ev op}.

The proof is more complicated than the previous one. Fix $s<t$ and
set for simplicity 
\[
T:=A\left(t\right)R\left(\lambda,A\left(t\right)\right)\left[A\left(t\right)^{-1}-A\left(s\right)^{-1}\right].
\]
The heuristic idea is to write down $\left[\lambda I-A\left(t\right)\right]T$,
which is formally not admissible because the image of $T$ is not
necessarily contained in $D\left(A\left(t\right)\right)$.

We begin by decomposing $T$ into the sum of two operators which have
range contained in $D\left(A\left(t\right)\right)$ and $D\left(A\left(s\right)\right)$
respectively. Indeed, noting that $A\left(t\right)R\left(\lambda,A\left(t\right)\right)=\lambda R\left(\lambda,A\left(t\right)\right)-I$
and that, by the resolvent identity, $\lambda R\left(\lambda,A\left(t\right)\right)A\left(t\right)^{-1}=R\left(\lambda,A\left(t\right)\right)+A\left(t\right)^{-1}$,
we have:
\begin{align}
T & =\left[\lambda R\left(\lambda,A\left(t\right)\right)-I\right]\left[A\left(t\right)^{-1}-A\left(s\right)^{-1}\right]\nonumber \\
 & =R\left(\lambda,A\left(t\right)\right)\left[I-\lambda A\left(s\right)^{-1}\right]+A\left(s\right)^{-1},\label{eq: scomposizione T}
\end{align}
which implies $ImT\subseteq D\left(A\left(t\right)\right)+D\left(A\left(s\right)\right)\subseteq H^{1}\left(\Omega\right)$.
Thus, if $\widetilde{A}\left(t\right):H^{1}\left(\Omega\right)\to X$
is any operator extending $A\left(t\right)$:
\begin{align*}
\left[\lambda I-\widetilde{A}\left(t\right)\right]T & =I-\widetilde{A}\left(t\right)A\left(s\right)^{-1}.
\end{align*}
Define $\left\langle \widetilde{A}\left(t\right)u,v\right\rangle _{X}:=a\left(t\right)\left(u,v\right)$
for every $u,v\in H^{1}\left(\Omega\right)$; applying the above relation
to some $f\in X$ and multiplying by $v\in H^{1}\left(\Omega\right)$
in $X$ leads to:
\[
\lambda\left\langle Tf,v\right\rangle _{X}-a\left(t\right)\left(Tf,v\right)=\left[a\left(s\right)-a\left(t\right)\right]\left(A\left(s\right)^{-1}f,v\right),
\]
since $\left\langle f,v\right\rangle _{X}=\left\langle A\left(s\right)A\left(s\right)^{-1}f,v\right\rangle _{X}=a\left(s\right)\left(A\left(s\right)^{-1}f,v\right)$
by definition of $A\left(s\right)$.

By definition of $a\left(t\right)$ and taking $v=Tf$ we have:
\begin{align}
\left(\lambda+1\right)\left|Tf\right|_{X}^{2}= & -\int_{\Omega}\sum_{i,j=1}^{N}D_{i}Tf\left(x\right)a_{ij}\left(t,x\right)D_{j}Tf\left(x\right)dx\nonumber \\
 & +\int_{\Omega}\sum_{i,j=1}^{N}D_{i}A\left(s\right)^{-1}f\left(x\right)\left[a_{ij}\left(t,x\right)-a_{ij}\left(s,x\right)\right]D_{j}Tf\left(x\right)dx\label{eq: 1+lambda fond}\\
=: & -I_{1}+I_{2}.\nonumber 
\end{align}
Now we only have to carefully estimate $I_{1}$ and $I_{2}$. First,
setting
\[
M:=\max_{\substack{i,j=1,\dots,N\\
\left(r,x\right)\in\mathbb{R}\times\overline{\Omega}
}
}\left|a_{ij}\left(r,x\right)\right|,
\]
we have by $\text{\eqref{eq: matrice coerciva}}$:
\begin{equation}
\eta\int_{\Omega}\left|\nabla Tf\left(x\right)\right|_{\mathbb{C}^{N}}^{2}dx\leq\Re I_{1}\leq\left|I_{1}\right|\leq MN^{2}\int_{\Omega}\left|\nabla Tf\left(x\right)\right|_{\mathbb{C}^{N}}^{2}dx\label{eq: I1 primo}
\end{equation}
which implies
\begin{equation}
\frac{\eta}{2MN^{2}}\left|I_{1}\right|-\frac{1}{2}\Re I_{1}\leq0.\label{eq: I1 secondo}
\end{equation}
As far as $I_{2}$ is concerned, we have:
\begin{align}
\left|I_{2}\right|\leq & \frac{N}{2K}\int_{\Omega}\left|\nabla Tf\left(x\right)\right|_{\mathbb{C}^{N}}^{2}dx+\frac{NK}{2}\left|t-s\right|^{2\alpha}\int_{\Omega}\left|\nabla A\left(s\right)^{-1}f\left(x\right)\right|_{\mathbb{C}^{N}}^{2},\label{eq: I2}
\end{align}
for every $K>0$, since the euclidean product in $\mathbb{C}^{N}$
obviously satisfies $\left\langle w_{1},w_{2}\right\rangle _{\mathbb{C}^{N}}\leq\frac{1}{2K}\left|w_{1}\right|_{\mathbb{C}^{N}}^{2}+\frac{K}{2}\left|w_{2}\right|_{\mathbb{C}^{N}}^{2}$,
and remembering the assumption in $\text{\eqref{eq: aij holder}}$.
Now we infer from $\text{\eqref{eq: 1+lambda fond}}$:
\begin{align*}
\left(1+\Re\lambda\right)\left|Tf\right|_{X}^{2} & \leq-\Re I_{1}+\left|I_{2}\right|\\
\frac{\eta}{2MN^{2}}\left|\Im\lambda\right|\left|Tf\right|_{X}^{2} & \leq\frac{\eta}{2MN^{2}}\left(\left|I_{1}\right|+\left|I_{2}\right|\right).
\end{align*}

Summing up, we obtain from $\eqref{eq: I1 secondo}$ and $\text{\eqref{eq: I2}}$:
\begin{align*}
\left(1+\Re\lambda+\frac{\eta}{2MN^{2}}\left|\Im\lambda\right|\right)\left|Tf\right|_{X}^{2}\leq & -\Re I_{1}+\frac{\eta}{2MN^{2}}\left|I_{1}\right|+\left(\frac{\eta}{2MN^{2}}+1\right)\left|I_{2}\right|\\
\leq & -\frac{1}{2}\Re I_{1}+\frac{1}{K}\frac{\eta+2MN^{2}}{4MN}\int_{\Omega}\left|\nabla Tf\left(x\right)\right|_{\mathbb{C}^{N}}^{2}dx\\
 & +K\frac{\eta+2MN^{2}}{4MN}\left|t-s\right|^{2\alpha}\int_{\Omega}\left|\nabla A\left(s\right)^{-1}f\left(x\right)\right|_{\mathbb{C}^{N}}^{2}.
\end{align*}
By $\text{\eqref{eq: I1 primo}}$, the sum of the first two terms
in the right-hand member is non-positive if $K$ is such that $\frac{1}{K}\frac{\eta+2MN^{2}}{4MN}\leq\frac{\eta}{2}$.
We deduce that there exists a constant $C>0$ depending only on the
problem's data $M$, $N$, $\eta$, such that:
\[
\left(1+\Re\lambda+\frac{\eta}{2MN^{2}}\left|\Im\lambda\right|\right)\left|Tf\right|_{X}^{2}\leq C\left|t-s\right|^{2\alpha}\left|f\right|_{X}^{2},
\]
having used $\text{\eqref{eq: norma X grad inversa}}$. By invoking
Lemma $\text{\ref{rem: lemmino settore}}$ with $c_{1}=\frac{\eta}{2MN^{2}}$
and any $c_{2}>1$ we obtain the desired estimate $\text{\eqref{eq: AT2}}$
with the related $\theta_{0}$.

\vspace{2mm}

Thus, Assumption $\text{\ref{assu: omega_0 neg}}$, i) is proven.
We now show, in sequence, that: a) $U\left(t,s\right)$ satisfies
the second part of Assumption $\text{\ref{assu: omega_0 neg}}$; b)
$U\left(t,s\right)$ satisfies Assumption $\text{\ref{ass: U* ha autov di mod max}}$;
c) $U\left(t,s\right)^{*}$ admits a unique eigenvalue of maximum
modulus, and that this eigenvalue is semisimple.
\begin{fact}
$\left\{ U\left(t,s\right)\mid s\leq t\right\} $ satisfies Assumption
$\text{\ref{assu: omega_0 neg}}$, ii).
\end{fact}

\begin{proof}
Fix $s\in\mathbb{R}$ and $f\in X$. Multiplying in $X$ both sides
of the evolution equation in $\text{\eqref{eq: evoluzione}}$ by $U\left(t,s\right)f$
we obtain by $\text{\eqref{eq: forma dissipativa}}$:
\begin{align*}
\frac{\partial}{\partial t}\left|U\left(t,s\right)f\right|_{X}^{2} & =2\left\langle \frac{\partial U\left(t,s\right)}{\partial t}f,U\left(t,s\right)f\right\rangle _{X}\\
 & =2\left\langle A\left(t\right)U\left(t,s\right)f,U\left(t,s\right)f\right\rangle _{X}\\
 & \leq-2\left|U\left(t,s\right)f\right|_{X}^{2}\quad\forall t>s.
\end{align*}
Joint with the condition $\left|U\left(s,s\right)f\right|_{X}^{2}=\left|f\right|_{X}^{2}$,
this implies
\begin{equation}
\left|U\left(t,s\right)f\right|_{X}\leq e^{-\left(t-s\right)}\left|f\right|_{X}\quad\forall t\geq s.\label{eq: stima norma U(t,s) esempio}
\end{equation}
Thus $\left\Vert U\left(t,s\right)^{*}\right\Vert _{\mathcal{L}\left(X\right)}=\left\Vert U\left(t,s\right)\right\Vert _{\mathcal{L}\left(X\right)}\leq e^{-\left(t-s\right)}$,
and Assumption $\text{\ref{assu: omega_0 neg}}$ is satisfied by $\left\{ U\left(t,s\right)\mid s\leq t\right\} $.
\end{proof}
Observe that, if $g\in X$ is such that $\left[U\left(t,s\right)g\right]_{\Omega}=0$,
then, arguing in the same way, we obtain from $\text{\eqref{eq: forma su media nulla}}$
that $\frac{\partial}{\partial t}\left|U\left(t,s\right)g\right|_{X}^{2}\leq-2\left(1+\eta C_{\Omega}\right)\left|U\left(t,s\right)g\right|_{X}^{2}$.

Therefore:
\begin{equation}
\left|U\left(t,s\right)g\right|_{X}\leq e^{-\left(1+\eta C_{\Omega}\right)\left(t-s\right)}\left|g\right|_{X}\quad\forall t\geq s.\label{eq: decad U(t,s) media nulla}
\end{equation}

\vspace{2mm}

We denote by $\mathbb{R}_{\Omega}$ the set of the real functions
that equal a constant almost everywhere in $\Omega$.
\begin{fact}
\label{fact: autospazio U(t,s)}For every $s<t$, $Ker\left(e^{-\left(t-s\right)}I-U\left(t,s\right)\right)=\mathbb{R}_{\Omega}$.
\end{fact}

\begin{proof}
$\left(\supseteq\right)$ Let $f\in\mathbb{R}_{\Omega}$, $v\in H^{1}\left(\Omega\right)$
and $t\in\mathbb{R}$. By $\text{\eqref{eq: forma su costanti}}$
we have for every $s<t$:
\begin{align*}
\left\langle \frac{\partial U\left(t,s\right)}{\partial s}f,v\right\rangle _{X} & =-\left\langle U\left(t,s\right)A\left(s\right)f,v\right\rangle _{X}\\
 & =-\left\langle A\left(s\right)f,U\left(t,s\right)^{*}v\right\rangle _{X}\\
 & =\left\langle U\left(t,s\right)f,v\right\rangle _{X}.
\end{align*}
Since $U\left(t,t\right)=I$, the above equality implies $U\left(t,s\right)f=e^{-\left(t-s\right)}f$
a.e. in $X$, hence $f\in Ker\left(e^{-\left(t-s\right)}I-U\left(t,s\right)\right)$.

$\left(\subseteq\right)$ Now assume that $f$ is an eigenvector of
$U\left(t,s\right)$ relative to the eigenvalue $e^{-\left(t-s\right)}$.
First observe that the functions $U\left(t,s\right)\left(f-f_{\Omega}\right)$
and $U\left(t,s\right)f_{\Omega}$ are orthogonal in $X$, since
\[
\left\langle U\left(t,s\right)f,U\left(t,s\right)f_{\Omega}\right\rangle _{X}=e^{-2\left(t-s\right)}f_{\Omega}\int_{\Omega}f\left(x\right)dx=\lambda\left(\Omega\right)e^{-2\left(t-s\right)}f_{\Omega}^{2}=\left|U\left(t,s\right)f_{\Omega}\right|_{X}^{2}.
\]
In the same way we see that $\left[U\left(t,s\right)\left(f-f_{\Omega}\right)\right]_{\Omega}=0$.
Therefore, using $\text{\eqref{eq: decad U(t,s) media nulla}}$ with
$g=f-f_{\Omega}$ we obtain:
\begin{align*}
e^{-2\left(t-s\right)}\left|f-f_{\Omega}\right|_{X}^{2}+e^{-2\left(t-s\right)}\lambda\left(\Omega\right)f_{\Omega}^{2} & =\left|e^{-\left(t-s\right)}f\right|_{X}^{2}\\
 & =\left|U\left(t,s\right)\left(f-f_{\Omega}\right)+U\left(t,s\right)f_{\Omega}\right|_{X}^{2}\\
 & =\left|U\left(t,s\right)\left(f-f_{\Omega}\right)\right|_{X}^{2}+\left|U\left(t,s\right)f_{\Omega}\right|_{X}^{2}\\
 & \leq e^{-2\left(1+\eta C_{\Omega}\right)\left(t-s\right)}\left|f-f_{\Omega}\right|_{X}^{2}+e^{-2\left(t-s\right)}\lambda\left(\Omega\right)f_{\Omega}^{2},
\end{align*}
which is a contradiction unless $f=f_{\Omega}$ almost everywhere
in $\Omega$.
\end{proof}
Taking into account the estimate in $\eqref{eq: stima norma U(t,s) esempio}$
we have as consequence:
\[
\left\Vert U\left(t,s\right)^{*}\right\Vert _{\mathcal{L}\left(X\right)}=\left\Vert U\left(t,s\right)\right\Vert _{\mathcal{L}\left(X\right)}=e^{-\left(t-s\right)}.
\]

\begin{fact}
\label{fact: autospazio U(t,s)*}For every $s<t$, $Ker\left(e^{-\left(t-s\right)}I-U\left(t,s\right)^{*}\right)\supseteq\mathbb{R}_{\Omega}$.
\end{fact}

\begin{proof}
The proof is the same as before, considering the evolution $\frac{\partial}{\partial t}U\left(t,s\right)^{*}=U\left(t,s\right)^{*}A\left(t\right)$
for $t>s$. Take $f$ almost everywhere constant. We obtain from $\text{\eqref{eq: forma su costanti}}$
that for every $v\in X$:
\[
\left\langle \frac{\partial U\left(t,s\right)}{\partial t}^{*}f,v\right\rangle _{X}=\left\langle A\left(t\right)f,U\left(t,s\right)v\right\rangle _{X}=-\left\langle U\left(t,s\right)^{*}f,v\right\rangle _{X}\quad\forall t>s.
\]
This implies $U\left(t,s\right)^{*}f=e^{-\left(t-s\right)}f$, almost
everywhere in $X$, for every $t\geq s$.
\end{proof}
\begin{fact}
$\left\{ U\left(t,s\right)\mid s\leq t\right\} $ satisfies Assumption
$\text{\ref{ass: U* ha autov di mod max}}$.
\end{fact}

\begin{proof}
It follows in particular from Fact $\text{\ref{fact: autospazio U(t,s)*}}$
that, for every $s<t$, $e^{-\left(t-s\right)}$ is an eigenvalue
of $U\left(t,s\right)^{*}$ that coincides with $\left\Vert U\left(t,s\right)^{*}\right\Vert _{\mathcal{L}\left(X\right)}$
and thus with the spectral radius of $U\left(t,s\right)^{*}$.
\end{proof}
Now, with a sort of bootstrap argument, we can improve our initial
conclusions. Denote
\[
M_{\Omega}^{0}:=\left\{ f\in X\mid f_{\Omega}=0\right\} .
\]
Fix $s<t$ ; from the general identity $\overline{ImT}=\left(KerT^{*}\right)^{\perp}$
- which holds for every densely defined linear operator $T$ in an
Hilbert space - combined with Facts $\text{\ref{fact: autospazio U(t,s)}}$
and $\text{\ref{fact: autospazio U(t,s)*}}$ we deduce:
\begin{align}
 & \overline{Im\left(e^{-\left(t-s\right)}I-U\left(t,s\right)\right)}=Ker\left(e^{-\left(t-s\right)}I-U\left(t,s\right)^{*}\right)^{\perp}\subseteq M_{\Omega}^{0},\label{eq: Im autospazio U(t,s)}\\
 & \overline{Im\left(e^{-\left(t-s\right)}I-U\left(t,s\right)^{*}\right)}=Ker\left(e^{-\left(t-s\right)}I-U\left(t,s\right)\right)^{\perp}=M_{\Omega}^{0}.\label{eq: Im autospazio U(t,s)*}
\end{align}

In particular, $Im\left(e^{-\left(t-s\right)}I-U\left(t,s\right)\right)$
and $Im\left(e^{-\left(t-s\right)}I-U\left(t,s\right)^{*}\right)$
are contained in $M_{\Omega}^{0}$, so:
\[
\left[U\left(t,s\right)f\right]_{\Omega}=\left[U\left(t,s\right)^{*}f\right]_{\Omega}=e^{-\left(t-s\right)}f_{\Omega}\quad\forall f\in X.
\]
Hence $U\left(t,s\right)M_{\Omega}^{0}\subseteq M_{\Omega}^{0}$ and
$U\left(t,s\right)^{*}M_{\Omega}^{0}\subseteq M_{\Omega}^{0}$. Since
$M_{\Omega}^{0}$ is closed, both the restrictions $U\left(t,s\right)_{\mid M_{\Omega}^{0}}$
and $U\left(t,s\right)_{\mid M_{\Omega}^{0}}^{*}$ are bounded linear
operators of the Hilbert space $M_{\Omega}^{0}$.

Having established relation $\text{\eqref{eq: decad U(t,s) media nulla}}$
for $g\in\left(U\left(t,s\right)\right)^{-1}\left(M_{\Omega}^{0}\right)$,
we now see that such relation actually holds for $g\in M_{\Omega}^{0}$.
Namely:
\begin{equation}
\left|U\left(t,s\right)g\right|_{X}\leq e^{-\left(1+\eta C_{\Omega}\right)\left(t-s\right)}\left|g\right|_{X}\quad\forall g\in M_{\Omega}^{0},\label{eq: decad media nulla strong}
\end{equation}
which means 
\[
\left\Vert U\left(t,s\right)_{\mid M_{\Omega}^{0}}^{*}\right\Vert _{\mathcal{L}\left(M_{\Omega}^{0}\right)}=\left\Vert U\left(t,s\right)_{\mid M_{\Omega}^{0}}\right\Vert _{\mathcal{L}\left(M_{\Omega}^{0}\right)}\leq e^{-\left(1+\eta C_{\Omega}\right)\left(t-s\right)}.
\]
Thus $\text{\eqref{eq: decad media nulla strong}}$ holds also with
$U\left(t,s\right)^{*}$ in place of $U\left(t,s\right)$. Besides,
in particular, the spectral radia of $U\left(t,s\right)_{\mid M_{\Omega}^{0}}$
and $U\left(t,s\right)_{\mid M_{\Omega}^{0}}^{*}$ are less than or
equal to $e^{-\left(1+\eta C_{\Omega}\right)\left(t-s\right)}$, namely
the region $\mathbb{C}\setminus\overline{B\left(0,e^{-\left(1+\eta C_{\Omega}\right)\left(t-s\right)}\right)}$
is contained in the resolvent set of both operators.
\begin{fact}
For every $s<t$, the number $e^{-\left(t-s\right)}$ is the unique
eigenvalue of $U\left(t,s\right)^{*}$ with maximum modulus, it is
an isolated element of the spectrum of $U\left(t,s\right)^{*}$, and
\begin{align*}
Im\left(e^{-\left(t-s\right)}I-U\left(t,s\right)^{*}\right) & =M_{\Omega}^{0},\\
Ker\left(e^{-\left(t-s\right)}I-U\left(t,s\right)^{*}\right) & =\mathbb{R}_{\Omega}^{0}.
\end{align*}
Consequently:
\[
X=Ker\left(e^{-\left(t-s\right)}I-U\left(t,s\right)^{*}\right)\oplus Im\left(e^{-\left(t-s\right)}I-U\left(t,s\right)^{*}\right).
\]

In particular, $e^{-\left(t-s\right)}$ is semisimple as an eigenvalue
of $U\left(t,s\right)^{*}$.

These assertions also hold with $U\left(t,s\right)$ in place of $U\left(t,s\right)^{*}$.
\end{fact}

\begin{proof}
Fix $t>s$ and denote by $V\left(t,s\right)$ either $U\left(t,s\right)$
or $U\left(t,s\right)^{*}$. It follows from the last observation
that $e^{-\left(t-s\right)}$ belongs to the resolvent set of $V\left(t,s\right)_{\mid M_{\Omega}^{0}}$,
which implies that $\left[e^{-\left(t-s\right)}I-V\left(t,s\right)\right]_{\mid M_{\Omega}^{0}}$
is a bijection from $M_{\Omega}^{0}$ to itself. Therefore $M_{\Omega}^{0}\subseteq Im\left(e^{-\left(t-s\right)}I-V\left(t,s\right)\right)$,
and, as a product, relation $\text{\eqref{eq: Im autospazio U(t,s)}}$
transforms into a chain of identities, which implies further that
$Ker\left(e^{-\left(t-s\right)}I-U\left(t,s\right)^{*}\right)=\mathbb{R}_{\Omega}^{0}$.

Considering also relation $\text{\eqref{eq: Im autospazio U(t,s)*}}$
and remembering Fact $\text{\ref{fact: autospazio U(t,s)}}$ we can
write:
\[
X=Ker\left(e^{-\left(t-s\right)}I-V\left(t,s\right)\right)\oplus Im\left(e^{-\left(t-s\right)}I-V\left(t,s\right)\right).
\]
It remains to be proven that $e^{-\left(t-s\right)}$ is the unique
eigenvalue of $V\left(t,s\right)$ with modulus $e^{-\left(t-s\right)}$,
and that it is an isolated element of the spectrum of $V\left(t,s\right)$.

Assume $\left|\zeta\right|=e^{-\left(t-s\right)}$ and $f\in Ker\left(\zeta I-V\left(t,s\right)\right)\setminus\left\{ 0\right\} $.
Since both $\mathbb{R}_{\Omega}$ and $M_{\Omega}$ are $V\left(t,s\right)$-invariant,
$V\left(t,s\right)f_{\Omega}$ and $V\left(t,s\right)\left(f-f_{\Omega}\right)$
are orthogonal. Thus we can argue as in the proof of the second inclusion
in Fact $\text{\ref{fact: autospazio U(t,s)}}$ - where orthogonality
followed from the fact that $f$ was an eigenfunction of $U\left(t,s\right)$
relative to $e^{-\left(t-s\right)}$ - using $\text{\eqref{eq: decad media nulla strong}}$
with $V\left(t,s\right)$ in place of $U\left(t,s\right)$ and $g=f-f_{\Omega}$.
This way we obtain that $f=f_{\Omega}$ a.e. in $X$, and this implies
$\zeta=e^{-\left(t-s\right)}$.

\vspace{2mm}

Eventually, we show that the complex region $\mathbb{C}\setminus\overline{B\left(0,e^{-\left(1+\eta C_{\Omega}\right)\left(t-s\right)}\right)}\setminus\left\{ e^{-\left(t-s\right)}\right\} $
is contained in the resolvent set of $V\left(t,s\right)$.

Assume that $\left|\zeta\right|>e^{-\left(1+\eta C_{\Omega}\right)\left(t-s\right)}$
and $\zeta\neq e^{-\left(t-s\right)}$; we prove that the operator
$\zeta I-V\left(t,s\right):X\to X$ is bijective.

If $f\in Ker\left(\zeta I-V\left(t,s\right)\right)$, then:
\begin{align*}
 & \zeta\left(f-f_{\Omega}\right)+\zeta f_{\Omega}=V\left(t,s\right)\left(f-f_{\Omega}\right)+e^{-\left(t-s\right)}f_{\Omega}\\
\implies & \left[\zeta I-V\left(t,s\right)\right]\left(f-f_{\Omega}\right)=\left(e^{-\left(t-s\right)}-\zeta\right)f_{\Omega}.
\end{align*}
The left hand member of the above identity belongs to $M_{\Omega}^{0}$,
and the right hand member belongs to $\mathbb{R}_{\Omega}$; we deduce
that both quantities are zero. We have already observed that all the
complex numbers whose modulus exceeds $e^{-\left(1+\eta C_{\Omega}\right)\left(t-s\right)}$
are in the resolvent set of $V\left(t,s\right)_{\mid M_{\Omega}^{0}}$;
it follows that $\left[\zeta I-V\left(t,s\right)\right]_{\mid M_{\Omega}^{0}}$
is a bijection from $M_{\Omega}^{0}$ to itself - in particular, it
is injective. This implies that $f=f_{\Omega}$ almost everywhere
in $X$. If it was $f\neq0$ then $f$ would be an almost everywhere
constant eigenfunction of $V\left(t,s\right)$ relative to the eigenvalue
$\zeta$, thus we would have $\zeta=e^{-\left(t-s\right)}$, a contradiction.
This argument shows that $\zeta I-V\left(t,s\right)$ is injective.

The operator $\zeta I-V\left(t,s\right)$ is also surjective, because
for every $g\in X$ we have:
\[
g=\left[\zeta I-V\left(t,s\right)\right]\left(g_{0}+\frac{g_{\Omega}}{\zeta-e^{-\left(t-s\right)}}\right),
\]
where $g_{0}=\left[\zeta I-V\left(t,s\right)\right]^{-1}\left(g-g_{\Omega}\right)$.

Hence we have proved that $\zeta I-V\left(t,s\right)$ is bijective;
since it is bounded, $\zeta$ is in the resolvent set of $V\left(t,s\right)$.
\end{proof}
\subsection{The operators \textit{B}($\cdot$)}\label{subsec: B(t)} 

We begin by implementing a technique introduced by Cerrai and Lunardi
in \cite{Cer21} which gives sufficient conditions on a family $\left\{ B\left(t\right)\mid t\in\mathbb{R}\right\} $
in order that an operator $Q\left(t,s\right)$ defined as in $\text{\eqref{eq: def cov Q}}$
has finite trace. This result is substantially independent of $A\left(t\right)$,
provided that the solutions of the corresponding homogeneous parabolic
equation can be reasonably represented.

The method excludes that $B\left(t\right)$ can be chosen as the identity
operator, but opens to the possibility of defining it as a fractional
power of $A\left(t\right)^{-1}$, with strictly positive exponent.

We begin by a result characterizing the trace of $Q\left(t,s\right)$
at the most general level, given our choice for $A\left(t\right)$.
\begin{lem}
\label{lem: uguaglianza trQ(t,s)}Let $\left\{ B\left(t\right)\mid t\in\mathbb{R}\right\} \subseteq\mathcal{L}\left(X\right)$
be a family of bounded operators such that $B\left(\cdot\right)f$
is measurable for every $f\in X$. Then
\begin{equation}
trQ\left(t,s\right)=\int_{s}^{t}e^{-2\left(t-r\right)}\int_{\Omega}\left|B\left(r\right)^{*}k_{A}\left(x,\cdot,t,r\right)\right|_{X}^{2}dxdr,\label{eq: formula traccia}
\end{equation}
where $k_{A}:\Omega^{2}\times\left\{ \left(t,s\right)\mid s\leq t\right\} \to\mathbb{R}$
is a function depending only on $\left\{ U\left(t,s\right)\right\} _{s\leq t}$
such that $k_{A}\left(\cdot,\cdot,t,s\right)\in L^{\infty}\left(\Omega^{2}\right)$
for every $s\leq t$ and
\begin{align}
\left|k_{A}\left(x,y,t,s\right)\right| & \leq C\left(t-s\right)^{-\frac{N}{2}}e^{-\frac{\left|x-y\right|^{2}}{t-s}}\quad\forall t>s,\ x,y\in\Omega.\label{eq: nucleo}
\end{align}
\end{lem}

\begin{proof}
Fix $s\in\mathbb{R}$ and let $\mathcal{A}\left(\cdot\right)$ the
differential operator in $\text{\eqref{eq: op differenziale}}$. Consider
the following parabolic equation:
\begin{equation}
\begin{cases}
\frac{\partial}{\partial t}\mathrm{y}\left(t,x\right)=\left[\mathcal{A}\left(t\right)\mathrm{y}\left(t,\cdot\right)\right]\left(x\right)+\mathrm{y}\left(t,x\right) & \left(t,x\right)\in\left(s,+\infty\right)\times\Omega\\
\sum_{i,j=1}^{N}a_{i,j}\left(t,x\right)\frac{\partial}{\partial x_{i}}\mathrm{y}\left(t,x\right)\nu_{j}\left(x\right)=0 & \left(t,x\right)\in\left(s,+\infty\right)\times\partial\Omega\\
\mathrm{y}\left(s,x\right)=f\left(x\right) & x\in\Omega.
\end{cases}\label{eq: parabolica}
\end{equation}
For every $f\in L^{2}\left(\Omega\right)$, $\text{\eqref{eq: parabolica}}$
admits a unique weak solution $y_{s,f}\in L^{2}\left(s,T;H^{1}\left(\Omega\right)\right)\cap H^{1}\left(s,T;H^{-1}\left(\Omega\right)\right)$
which can be expressed as
\begin{equation}
y_{s,f}\left(t,x\right)=\int_{\Omega}k_{A}\left(x,y,t,s\right)f\left(y\right)dy\quad\forall\left(t,x\right)\in\left(s,+\infty\right)\times\Omega,\label{eq: nucleo sol A(t)+I}
\end{equation}
where $k_{A}$ is a kernel satisfying $\text{\eqref{eq: nucleo}}$.
See for instance \cite{DAN} for this representation formula. Obviously
$y_{s,f}\left(t,\cdot\right)\in D\left(A\left(t\right)\right)$ for
every $t>s$, since the equation and the boundary condition in $\text{\eqref{eq: parabolica}}$
applied to $y_{s,f}$ imply $a\left(t\right)\left(y_{s,f}\left(t,\cdot\right),v\right)=\left\langle \frac{\partial}{\partial t}y_{s,f}\left(t,\cdot\right)-y_{s,f}\left(t,\cdot\right),v\right\rangle _{X}$
for every $v\in H^{1}\left(\Omega\right)$. Note further that
\[
\frac{\partial}{\partial t}e^{-\left(t-s\right)}y_{s,f}\left(t,\cdot\right)=A\left(t\right)e^{-\left(t-s\right)}y_{s,f}\left(t,\cdot\right)\ \text{in }\Omega,\quad\forall t>s
\]
which implies $e^{-\left(t-s\right)}y_{s,f}\left(t,\cdot\right)=U\left(t,s\right)f$
for every $t>s$, by the uniqueness of the solution of the abstract
Cauchy problem $\mathrm{w}'=A\left(\cdot\right)\mathrm{w}$ in $\left]s,+\infty\right)$,
$\mathrm{w}\left(s\right)=f$.

By $\text{\eqref{eq: nucleo sol A(t)+I}}$ we have, for every $f,h\in X$:
\begin{align*}
\left\langle U\left(t,r\right)^{*}h,f\right\rangle _{X} & =\left\langle h,U\left(t,r\right)f\right\rangle _{X}\\
 & =\left\langle h,e^{-\left(t-r\right)}\int_{\Omega}k_{A}\left(\cdot,y,t,r\right)f\left(y\right)dy\right\rangle _{X}\\
 & =e^{-\left(t-r\right)}\int_{\Omega}\int_{\Omega}k_{A}\left(x,y,t,r\right)h\left(x\right)f\left(y\right)dxdy\\
 & =\left\langle e^{-\left(t-r\right)}\int_{\Omega}h\left(x\right)k_{A}\left(x,\cdot,t,r\right)dx,f\right\rangle _{X}.
\end{align*}

Fix $t>s$ and $h\in X$. Inserting the equality $U\left(t,r\right)^{*}h=e^{-\left(t-r\right)}\int_{\Omega}h\left(x\right)k_{A}\left(x,\cdot,t,r\right)dx$
in the definition $\eqref{eq: def cov Q}$ leads to:
\begin{align*}
\left\langle Q\left(t,s\right)h,h\right\rangle _{X} & =\int_{s}^{t}\left|B\left(r\right)^{*}U\left(t,r\right)^{*}h\right|_{X}^{2}dr\\
 & =\int_{s}^{t}e^{-2\left(t-r\right)}\left|B\left(r\right)^{*}\int_{\Omega}h\left(x\right)k_{A}\left(x,\cdot,t,r\right)dx\right|_{X}^{2}dr\\
 & =\int_{s}^{t}e^{-2\left(t-r\right)}\int_{\Omega}\left|\left[B\left(r\right)^{*}\int_{\Omega}h\left(x\right)k_{A}\left(x,\cdot,t,r\right)dx\right]\left(y\right)\right|^{2}dydr\\
 & =\int_{s}^{t}e^{-2\left(t-r\right)}\int_{\Omega}\left|\int_{\Omega}h\left(x\right)\left[B\left(r\right)^{*}k_{A}\left(x,\cdot,t,r\right)\right]\left(y\right)dx\right|^{2}dydr\\
 & =\int_{s}^{t}e^{-2\left(t-r\right)}\int_{\Omega}\left\langle h,x\mapsto\left[B\left(r\right)^{*}k_{A}\left(x,\cdot,t,r\right)\right]\left(y\right)\right\rangle _{X}^{2}dydr.
\end{align*}
Therefore:
\begin{align*}
trQ\left(t,s\right) & =\sum_{j=1}^{+\infty}\left\langle Q\left(t,s\right)e_{j},e_{j}\right\rangle _{X}\\
 & =\int_{s}^{t}e^{-2\left(t-r\right)}\int_{\Omega}\sum_{j=1}^{+\infty}\left\langle e_{j},x\mapsto\left[B\left(r\right)^{*}k_{A}\left(x,\cdot,t,r\right)\right]\left(y\right)\right\rangle _{X}^{2}dydr\\
 & =\int_{s}^{t}e^{-2\left(t-r\right)}\int_{\Omega}\left|x\mapsto\left[B\left(r\right)^{*}k_{A}\left(x,\cdot,t,r\right)\right]\left(y\right)\right|_{X}^{2}dydr\\
 & =\int_{s}^{t}e^{-2\left(t-r\right)}\int_{\Omega}\int_{\Omega}\left|\left[B\left(r\right)^{*}k_{A}\left(x,\cdot,t,r\right)\right]\left(y\right)\right|^{2}dydxdr\\
 & =\int_{s}^{t}e^{-2\left(t-r\right)}\int_{\Omega}\left|B\left(r\right)^{*}k_{A}\left(x,\cdot,t,r\right)\right|_{X}^{2}dxdr,
\end{align*}
and $\text{\eqref{eq: formula traccia}}$ is proven.
\end{proof}
\begin{fact}
\label{fact: per trQ(t,s) finita}If $\left\{ B\left(t\right)^{*}\mid t\in\mathbb{R}\right\} \subseteq\mathcal{L}\left(L^{q}\left(\Omega\right);X\right)$
for some $q\in\left(1,\frac{N}{N-1}\right)$\footnote{Note that this choice of $q$ implies $q<2$ and therefore $L^{q}\left(\Omega\right)\supseteq X$
whenever $N\geq2$.}, $B\left(\cdot\right)f$ is measurable for every $f\in X$ and the
above family of operators is bounded with respect to the $\mathcal{L}\left(L^{q}\left(\Omega\right);X\right)$
norm, then Assumption $\text{\ref{assu: tracce finite}}$ is satisfied.
\end{fact}

\begin{proof}
Fix $s<r<t$ and $k_{A}$ as in Lemma $\text{\ref{lem: uguaglianza trQ(t,s)}}$.
Since $k_{A}\left(\cdot,\cdot;t,r\right)$ belongs to $L^{p}\left(\Omega^{2}\right)$
for every $p>1$, using $\text{\eqref{eq: nucleo}}$ we can estimate:

\begin{align*}
\int_{\Omega}\left|B\left(r\right)^{*}k_{A}\left(x,\cdot,t,r\right)\right|_{X}^{2}dx & \leq\int_{\Omega}\left\Vert B\left(r\right)^{*}\right\Vert _{\mathcal{L}\left(L^{q}\left(\Omega\right);X\right)}^{2}\left|k_{A}\left(x,\cdot,t,r\right)\right|_{L^{q}\left(\Omega\right)}^{2}dx\\
 & \leq C\left(t-r\right)^{-N}\int_{\Omega}\left(\int_{\Omega}e^{-\frac{\left|x-y\right|^{2}}{t-r}q}dy\right)^{\frac{2}{q}}dx\\
 & \leq C_{q,N}\left(t-r\right)^{-N+\frac{N}{q}}\int_{\Omega}\left(\int_{\mathbb{R}^{N}}e^{-\left|z\right|^{2}}dz\right)^{\frac{2}{q}}dx\\
 & =C_{q,N,\Omega}\left(t-r\right)^{-N+\frac{N}{q}}.
\end{align*}
By $\text{\eqref{eq: formula traccia}}$ we have:
\begin{align*}
trQ\left(t,s\right) & \leq C_{q,N,\Omega}\int_{s}^{t}e^{-2\left(t-r\right)}\left(t-r\right)^{-N+\frac{N}{q}}dr\\
 & \leq\widetilde{C}_{q,N,\Omega}\Gamma\left(-N+\frac{N}{q}+1\right),
\end{align*}
where $\Gamma$ denotes the gamma function; observe that the assumed
condition $q<\frac{N}{N-1}$ is equivalent to $-N+\frac{N}{q}+1>0$,
so that the last quantity is well defined.
\end{proof}
Observe that any operator $A\left(t\right)$ is closed and, obviously,
densely defined since $D\left(A\left(t\right)\right)\supseteq C_{0}^{1}\left(\Omega\right)$;
thus, by $\text{\eqref{eq: AT1}}$, $A\left(t\right)$ is the generator
of an analytic semigroup. This allows to define the fractional power
operator $\left[-A\left(t\right)\right]^{-\theta}$, in particular,
for every fixed $\theta\geq0$. Following the approach in \cite{Pazy}, \S2.6,
we remark that the following representation formula holds for every
$\theta\in\left(0,1\right)$:
\begin{align*}
\left[-A\left(t\right)\right]^{-\theta} & =\frac{\sin\pi\theta}{\pi}\int_{0}^{+\infty}\sigma^{-\theta}R\left(\sigma,A\left(t\right)\right)d\sigma.
\end{align*}

The formula is meaningful and defines a bounded operator; indeed,
by $\text{\eqref{eq: AT1}}$:
\begin{align}
\left\Vert \left[-A\left(t\right)\right]^{-\theta}\right\Vert _{\mathcal{L}\left(X\right)} & \leq\frac{\sin\pi\theta}{\pi}\int_{0}^{+\infty}\sigma^{-\theta}\left\Vert R\left(\sigma,A\left(t\right)\right)\right\Vert _{\mathcal{L}\left(X\right)}d\sigma\nonumber \\
 & \leq c\frac{\sin\pi\theta}{\pi}\left\{ \int_{0}^{1}\sigma^{-\theta}d\sigma+\int_{1}^{+\infty}\sigma^{-\theta-1}d\sigma\right\} .\nonumber \\
 & =c\left\{ \frac{\sin\left[\pi\left(1-\theta\right)\right]}{\pi\left(1-\theta\right)}+\frac{\sin\pi\theta}{\pi\theta}\right\} \nonumber \\
 & \leq2c.\label{eq: stima norma B(t)}
\end{align}

We now show that there are fractional powers of $A\left(t\right)^{-1}$
(the identity operator being indeed among them) which, taken as $B\left(t\right)$,
guarantee that the assumptions different from Assumption $\text{\ref{assu: tracce finite}}$
are satisfied.

\begin{fact}
\label{fact: per controllabilita}If $B\left(t\right)=\left[-A\left(t\right)\right]^{-\theta}$
for some $\theta\in\left[0,1/2\right)$ then Assumptions $\text{\ref{assu: B(t) limitati}}$,
$\text{\ref{ass: controllabilita}}$ and $\text{\ref{assu: mu_t non degeneri}}$
are satisfied.
\end{fact}

\begin{proof}
Clearly, every function $B\left(\cdot\right)x$ is measurable, and
the validity of Assumption $\text{\ref{assu: B(t) limitati}}$ comes
directly from $\text{\eqref{eq: stima norma B(t)}}$. Also, the proof
that the above choice makes true Assumption $\text{\ref{assu: mu_t non degeneri}}$
is an immediate consequence of Remark $\text{\ref{rem: per Q(t,s) iniettivo}}$.

Now fix $s<t$ and $\theta\in\left[0,1/2\right)$. By Theorem 6.8,
Chapter 2 in \cite{Pazy}, the algebraic inverse $\left[-A\left(t\right)\right]^{\theta}$
of $B\left(t\right)$ is a closed linear operator densely defined
in $D\left(\left[-A\left(t\right)\right]^{\theta}\right):=Im\left(B\left(t\right)\right)$,
and $D\left(A\left(t\right)\right)\subseteq D\left(\left[-A\left(t\right)\right]^{\theta}\right).$
Further, by Theorem 6.10 in the same chapter, there exist $C,C_{\theta}>0$
such that, for every $f\in X$ and $r>s$:
\begin{align*}
\left|\left[-A\left(r\right)\right]^{\theta}U\left(r,s\right)f\right|_{X} & \leq C\left|U\left(r,s\right)f\right|_{X}^{1-\theta}\left|A\left(r\right)U\left(r,s\right)f\right|_{X}^{\theta}\\
 & \leq C_{\theta}\left|f\right|_{X}^{1-\theta}\left|f\right|_{X}^{\theta}\left\Vert A\left(r\right)U\left(r,s\right)\right\Vert _{\mathcal{L}\left(X\right)}^{\theta}\\
 & \leq C_{\theta}\left(r-s\right)^{-\theta}\left|f\right|_{X},
\end{align*}
the last relation following from $\text{\eqref{eq: A(t)U(t,s)}}$.
Namely, $\left[-A\left(r\right)\right]^{\theta}U\left(r,s\right)\in\mathcal{L}\left(X\right)$
with:
\begin{equation}
\left\Vert \left[-A\left(r\right)\right]^{\theta}U\left(r,s\right)\right\Vert _{\mathcal{L}\left(X\right)}\leq C_{\theta}\left(r-s\right)^{-\theta}.\label{eq: stima L2}
\end{equation}
Since we have chosen $\theta<1/2$, relation $\text{\eqref{eq: stima L2}}$
implies that, for every $f\in X$, the function $r\to\left[-A\left(r\right)\right]^{\theta}U\left(r,s\right)f$
belongs to $L^{2}\left(s,t;X\right)$; thus, the vector
\begin{align*}
U\left(t,s\right)f & =\fint_{s}^{t}U\left(t,r\right)U\left(r,s\right)fdr\\
 & =\int_{s}^{t}U\left(t,r\right)B\left(r\right)\left[-A\left(r\right)\right]^{\theta}U\left(r,s\right)\frac{f}{t-s}dr
\end{align*}
belongs to the image of the operator
\begin{eqnarray*}
L_{s,t}:L^{2}\left(s,t;X\right) & \to & X\\
\varphi & \mapsto & \int_{s}^{t}U\left(t,r\right)B\left(r\right)\varphi\left(r\right)dr.
\end{eqnarray*}

By $\text{\eqref{eq: immagini uguali Lst}}$, we deduce that the inclusion
$ImU\left(t,s\right)\subseteq ImQ\left(t,s\right)^{\frac{1}{2}}$
holds, as prescribed by Assumption $\text{\ref{ass: controllabilita}}$.
\end{proof}
Now we have to prove that there exist suitable fractional powers of
$A\left(t\right)^{-1}$ complying with the conditions required by
Fact $\text{\ref{fact: per trQ(t,s) finita}}$ . This match will introduce
a constraint on the dimension.
\begin{fact}
If $N\leq3$, then there exists $\theta\in\left(0,1/2\right)$ such
that defining $B\left(t\right)=\left[-A\left(t\right)\right]^{-\theta}$
makes true Assumptions $\text{\ref{assu: B(t) limitati}}$, $\text{\ref{assu: tracce finite}}$,
$\text{\ref{ass: controllabilita}}$ and $\text{\ref{assu: mu_t non degeneri}}$.
\end{fact}

Having in mind Facts $\text{\ref{fact: per trQ(t,s) finita}}$ and
$\text{\ref{fact: per controllabilita}}$, we look for some $q\in\left(1,\frac{N}{N-1}\right)$
and $\theta\in\left[0,1/2\right)$ such that $\left[\left[-A\left(t\right)\right]^{-\theta}\right]^{*}$can
be continuously extended to $L^{q}\left(\Omega\right)$, and that
the norm of this operator as an element of $\mathcal{L}\left(L^{q}\left(\Omega\right);X\right)$
is bounded above by an absolute constant.

Clearly, $B\left(t\right)=\left[-A\left(t\right)\right]^{-\theta}$
is bounded as an operator from $X$ to $D\left(\left[-A\left(t\right)\right]^{\theta}\right)$.
Therefore, if the latter is continuously embedded in $L^{q'}\left(\Omega\right)$
for some $q'>1$, then $B\left(t\right)\in\mathcal{L}\left(X;L^{q'}\left(\Omega\right)\right)$,
which implies $B\left(t\right)^{*}\in\mathcal{L}\left(L^{q}\left(\Omega\right);X\right)$,
where $q=\frac{q'}{q'-1}$.

In order to prove a useful embedding, we explain how interpolation
theory can be used in order to prove that:
\begin{equation}
D\left(\left[-A\left(t\right)\right]^{\theta}\right)=H^{2\theta}\left(\Omega\right)\quad\forall\theta\in\left(0,\frac{3}{4}\right),\label{eq: D(A(t)^theta) uguale H^2theta}
\end{equation}
an identity between Hilbert spaces. To this end, we refer to the book
\cite{Triebel}.

We need a suitable characterization of $D\left(A\left(t\right)\right)$
before proceeding. Since every $u\in D\left(A\left(t\right)\right)$
is, for some $g\in X$, a weak solution of $\eqref{eq: problema ellittico}$,
the maximal regularity theory for elliptic problems in divergence
form implies $D\left(A\left(t\right)\right)\subseteq H^{2}\left(\Omega\right)$.
Thus, if we fix $u\in D\left(A\left(t\right)\right)$ then we have
for every $v\in H^{1}\left(\Omega\right)$:
\begin{align*}
\left\langle A\left(t\right)u,v\right\rangle _{X}= & \ a\left(t\right)\left(u,v\right)\\
= & \left\langle \sum_{i,j=1}^{N}D_{j}\left[a_{ij}\left(t,\cdot\right)D_{i}u\right]-u,v\right\rangle _{X}\\
 & -\int_{\partial\Omega}\sum_{i,j=1}^{N}D_{i}u\left(x\right)a_{ij}\left(t,x\right)\nu_{j}\left(x\right)v\left(x\right)d_{\sigma}x.
\end{align*}

This holds in particular for every $v\in C_{0}^{1}\left(\Omega\right)$
whose density implies $A\left(t\right)u=\sum_{i,j=1}^{N}D_{j}\left[a_{ij}\left(t,\cdot\right)D_{i}u\right]-u.$
Thus
\[
\int_{\partial\Omega}\sum_{i,j=1}^{N}D_{i}u\left(x\right)a_{ij}\left(t,x\right)\nu_{j}\left(x\right)v\left(x\right)d_{\sigma}x=0\quad\forall v\in H^{1}\left(\Omega\right),
\]
which means that the $L^{2}\left(\partial\Omega\right)$ function
$\sum_{i,j=1}^{N}\left(D_{i}u\right)a_{ij}\left(t,\cdot\right)_{\mid\partial\Omega}\nu_{j}$
belongs to the orthogonal complement of the dense subset $H^{\frac{1}{2}}\left(\partial\Omega\right)$
- and therefore is zero. Vice versa, any $u\in H^{2}\left(\Omega\right)$
such that $\sum_{i,j=1}^{N}\left(D_{i}u\right)a_{ij}\left(t,\cdot\right)_{\mid\partial\Omega}\nu_{j}=0$
is obviously in $D\left(A\left(t\right)\right)$. Hence
\[
D\left(A\left(t\right)\right)=\left\{ u\in H^{2}\left(\Omega\right)\mid\sum_{i,j=1}^{N}\left(D_{i}u\right)a_{ij}\left(t,\cdot\right)_{\mid\partial\Omega}\nu_{j}=0\right\} .
\]
Therefore, $D\left(A\left(t\right)\right)$ is a closed subspace of
$H^{2}\left(\Omega\right)$ defined by a boundary condition. The differential
operator $u\mapsto\sum_{i,j=1}^{N}\left(D_{i}u\right)a_{ij}\left(t,\cdot\right)_{\mid\partial\Omega}\nu_{j}$
forms a normal system of differential boundary operators in the sense
of Definition 4.3.3.1 in \cite{Triebel}, since $\sum_{i,j=1}^{N}\nu_{i}\left(x\right)a_{ij}\left(t,x\right)\nu_{j}\left(x\right)\geq\eta\neq0$
for every $x\in\partial\Omega$, by the strong ellipticity condition
in $\text{\eqref{eq: matrice coerciva}}$. This, together with the
above characterization of $D\left(A\left(t\right)\right)$, allows
to apply Theorem 4.3.3 in \cite{Triebel} infering that
\[
\left[X,D\left(A\left(t\right)\right)\right]_{\theta}=H^{2\theta}\left(\Omega\right)\quad\forall\theta\in\left(0,3/4\right),
\]

where $\left[X,D\left(A\left(t\right)\right)\right]_{\theta}$ denotes
the complex interpolation space between $X$ and $D\left(A\left(t\right)\right)$
and the identity goes along with the equivalence of the respective
norms. Further, since $A\left(t\right)$ is symmetric and positive-definite,
Theorem 1.18.10 in the same book tells us that:
\[
D\left(\left[-A\left(t\right)\right]^{\theta}\right)=\left[X,D\left(A\left(t\right)\right)\right]_{\theta}\quad\forall\theta\in\left(0,1\right).
\]
Again, the latter is intended as an identity between Hilbert spaces.
Thus $\text{\eqref{eq: D(A(t)^theta) uguale H^2theta}}$ is proved.
We deduce, for every $\theta\in\left(0,\frac{3}{4}\right)$ and $N>4\theta$,
that:
\[
D\left(\left[-A\left(t\right)\right]^{\theta}\right)=H^{2\theta}\left(\Omega\right)\overset{\mathcal{C}}{\hookrightarrow}L^{\frac{2N}{N-4\theta}}\left(\Omega\right).
\]
By the considerations above, this implies that $B\left(t\right)^{*}\in\mathcal{L}\left(L^{\frac{2N}{N+4\theta}}\left(\Omega\right);X\right)$
so that we can choose $q=\frac{2N}{N+4\theta}$, and the condition
$q<\frac{N}{N-1}$ becomes $N<2+4\theta$ .

Due to the constraint $\theta<1/2$, we can choose the dimension up
to $N=3$.

Eventually, we note that the family $\left\{ B\left(t\right)\mid t\in\mathbb{R}\right\} $
is bounded in $\mathcal{L}\left(X;L^{\frac{2N}{N-4\theta}}\left(\Omega\right)\right)$.
Indeed, for every $f\in X$:
\begin{align*}
\left|B\left(t\right)f\right|_{L^{\frac{2N}{N-4\theta}}\left(\Omega\right)} & \leq C_{0}\left|B\left(t\right)f\right|_{H^{2\theta}\left(\Omega\right)}\\
 & \leq C_{1}\left|B\left(t\right)f\right|_{D\left(\left[-A\left(t\right)\right]^{\theta}\right)}\\
 & =C_{1}\left(\left|B\left(t\right)f\right|_{X}+\left|f\right|_{X}\right)\\
 & \leq C_{2}\left|f\right|_{X},
\end{align*}
where we have used $\text{\eqref{eq: stima norma B(t)}}$ in the last
inequality.

\appendix

\section{Proofs of the basic results}\label{appendice dim}
\begin{prop}
\label{prop: immagini op Hilbert}Let $X$, $X_{1}$, $X_{2}$ be
Hilbert spaces, and $T_{1}\in\mathcal{L}\left(X_{1};X\right)$, $T_{2}\in\mathcal{L}\left(X_{2};X\right)$.
Then $ImT_{1}\subseteq ImT_{2}$ if and only if there exists $C>0$
such that
\begin{equation}
\left|T_{1}^{*}x\right|_{X_{1}}\leq C\left|T_{2}^{*}x\right|_{X_{2}}\quad\forall x\in X.\label{eq: rel T1 T2}
\end{equation}
\end{prop}

\begin{proof}
First we prove $\left(\Longleftarrow\right)$. Assume that $\text{\eqref{eq: rel T1 T2}}$
holds and let $x_{1}\in X_{1}\setminus\left\{ 0\right\} $; starting
from the decomposition $X=KerT_{2}^{*}\oplus\overline{ImT_{2}}$ we
show that the orthogonal projection of $T_{1}x_{1}$ onto $\overline{ImT_{2}}$
actually belongs to $ImT_{2}$ and we deduce, \emph{a posteriori},
that $T_{1}x_{1}\in ImT_{2}$. Denote by $\left(T_{2}^{*}\right)^{-1}:ImT_{2}^{*}\subseteq X_{2}\to X$
the pseudoinverse of $T_{2}^{*}$. For every $x_{2}\in ImT_{2}^{*}$
we have:
\begin{align*}
\left\langle T_{1}x_{1},\left(T_{2}^{*}\right)^{-1}x_{2}\right\rangle _{X} & =\left\langle x_{1},T_{1}^{*}\left(T_{2}^{*}\right)^{-1}x_{2}\right\rangle _{X_{1}}\\
 & \leq\left|x_{1}\right|_{X_{1}}\left|T_{1}^{*}\left(T_{2}^{*}\right)^{-1}x_{2}\right|_{X_{1}}\\
 & \leq C\left|x_{1}\right|_{X_{1}}\left|x_{2}\right|_{X_{2}}.
\end{align*}
By the Hahn-Banach theorem, the linear functional $\left\langle T_{1}x_{1},\left(T_{2}^{*}\right)^{-1}\left(\cdot\right)\right\rangle _{X}$
defined on the subspace $ImT_{2}^{*}$ of $X_{2}$ is extended by
a linear functional $f:X_{2}\to\mathbb{R}$ satisfying:
\[
\left|fx_{2}\right|\leq C\left|x_{1}\right|_{X_{1}}\left|x_{2}\right|_{X_{2}}\quad\forall x_{2}\in X_{2}.
\]
This means that $f\in X_{2}^{*}$, thus by the Riesz representation
theorem there exists $\xi\left(T_{1},T_{2},x_{1}\right)\in X_{2}$
representing $f$.

For every $z\in\overline{ImT_{2}}=\left(KerT_{2}^{*}\right)^{\perp}$
we have $z=\left(T_{2}^{*}\right)^{-1}T_{2}^{*}z$. We infer:
\begin{align*}
\left\langle T_{1}x_{1},z\right\rangle _{X} & =\left\langle T_{1}x_{1},\left(T_{2}^{*}\right)^{-1}T_{2}^{*}z\right\rangle _{X}\\
 & =fT_{2}^{*}z\\
 & =\left\langle \xi\left(T_{1},T_{2},x_{1}\right),T_{2}^{*}z\right\rangle _{X_{2}}\\
 & =\left\langle T_{2}\xi\left(T_{1},T_{2},x_{1}\right),z\right\rangle _{X}.
\end{align*}
Hence
\[
T_{1}x_{1}=P_{KerT_{2}^{*}}T_{1}x_{1}+T_{2}\xi\left(T_{1},T_{2},x_{1}\right),
\]
where $P_{KerT_{2}^{*}}$ denotes the orthogonal projection onto $KerT_{2}^{*}$.
Note that, as a direct consequence of the assumption in $\text{\eqref{eq: rel T1 T2}}$
we have $KerT_{2}^{*}\subseteq KerT_{1}^{*}$, which implies $ImT_{1}\subseteq\left(KerT_{2}^{*}\right)^{\perp}$.
Therefore $P_{KerT_{2}^{*}}T_{1}x_{1}=0$,  and $T_{1}x_{1}\in ImT_{2}$.

\vspace{2mm}

Now we prove the implication $\left(\implies\right)$. Since $ImT_{1}\subseteq ImT_{2}$,
the operator $T_{2}^{-1}T_{1}:X_{1}\to X_{2}$ is well defined (on
the whole $X_{1}$); it is easily shown to be closed and, therefore,
bounded. For every $x\in X$, we thus have:
\[
\left|T_{1}^{*}x\right|_{X_{1}}^{2}=\left\langle T_{1}^{*}x,\left(T_{2}^{-1}T_{1}\right)^{*}T_{2}^{*}x\right\rangle _{X_{1}}\leq\left\Vert T_{2}^{-1}T_{1}\right\Vert _{\mathcal{L}\left(X_{1};X_{2}\right)}\left|T_{1}^{*}x\right|_{X_{1}}\left|T_{2}^{*}x\right|_{X_{2}}.
\]
Hence, $\text{\eqref{eq: rel T1 T2}}$ holds.
\end{proof}
\begin{proof}[Proof of Proposition $\text{\ref{prop: op evoluzione}}$]

$ $

i) Fix $s<t$ and $f\in C_{b}\left(X\right).$ Let $x\in X$ and $x_{n}\to x$.
Then:
\begin{align*}
\left[P_{s,t}f\right]\left(x_{n}\right) & =\int_{X}f\left(U\left(t,s\right)x_{n}+y\right)\mathcal{N}\left(0,Q\left(t,s\right)\right)\left(dy\right)\\
 & \overset{n\to\infty}{\to}\int_{X}f\left(U\left(t,s\right)x+y\right)\mathcal{N}\left(0,Q\left(t,s\right)\right)\left(dy\right)\\
 & =\left[P_{s,t}f\right]\left(x\right),
\end{align*}
by dominated convergence, and $P_{s,t}f\in C_{b}\left(X\right)$.
The same argument applies to the Bochner integral that defines $\overrightarrow{P_{s,t}}$.
The contraction property of both $P_{s,t}$ and $\overrightarrow{P_{s,t}}$
is an immediate consequence of the definition.

Now fix $\tau\in\left(s,t\right)$. First we show that:
\begin{align}
\left[P_{s,\tau}\left(P_{\tau,t}f\right)\right]\left(x\right) & =\int_{X}f\left(z\right)\left(\mu_{1}\star\mu_{2}\right)\left(dz\right)\label{eq: O-U composto uguale convoluzione}
\end{align}
where
\begin{align*}
\mu_{1}:= & \mathcal{N}\left(0,Q\left(t,\tau\right)\right),\\
\mu_{2}:= & U\left(t,\tau\right)_{\sharp}\mathcal{N}\left(U\left(\tau,s\right)x,Q\left(\tau,s\right)\right)
\end{align*}
and $\sharp$ denotes the push-forward operator. Secondly we will
prove that
\begin{equation}
\mu_{1}\star\mu_{2}=\mathcal{N}\left(U\left(t,s\right)x,Q\left(t,s\right)\right).\label{eq: conv uguale Gaussiana buona}
\end{equation}
Relations $\text{\eqref{eq: O-U composto uguale convoluzione}}$ and
$\text{ \eqref{eq: conv uguale Gaussiana buona}}$ lead to $\text{\eqref{eq: op evoluzione}}$.

To prove $\text{\eqref{eq: O-U composto uguale convoluzione}}$, observe
that the change of variable $U\left(t,\tau\right)y=v$ lets us write:
\begin{align*}
\left[P_{s,\tau}\left(P_{\tau,t}f\right)\right]\left(x\right) & =\int_{X}\left[P_{\tau,t}f\right]\left(y\right)\mathcal{N}\left(U\left(\tau,s\right)x,Q\left(\tau,s\right)\right)\left(dy\right)\\
 & =\int_{X}\int_{X}f\left(u\right)\mathcal{N}\left(U\left(t,\tau\right)y,Q\left(t,\tau\right)\right)\left(du\right)\mathcal{N}\left(U\left(\tau,s\right)x,Q\left(\tau,s\right)\right)\left(dy\right)\\
 & =\int_{X}\int_{X}f\left(u+v\right)\mu_{1}\left(du\right)\mu_{2}\left(dv\right)\\
 & =\int_{X}f\left(z\right)\left(\mu_{1}\star\mu_{2}\right)\left(dz\right).
\end{align*}
For $\text{\eqref{eq: conv uguale Gaussiana buona}}$ we use the multiplication
property of the Fourier transform with respect to the convolution.
For every $\xi\in X$ we have:
\begin{align*}
\widehat{\mu_{2}}\left(\xi\right) & =\int_{X}e^{i\left\langle \xi,U\left(t,\tau\right)y\right\rangle _{X}}\mathcal{N}\left(U\left(\tau,s\right)x,Q\left(\tau,s\right)\right)\left(dy\right)\\
 & =\widehat{\mathcal{N}\left(U\left(\tau,s\right)x,Q\left(\tau,s\right)\right)}\left(U\left(t,\tau\right)^{*}\xi\right)\\
 & =\exp\left(i\left\langle U\left(t,\tau\right)^{*}\xi,U\left(\tau,s\right)x\right\rangle _{X}-\frac{1}{2}\left\langle U\left(t,\tau\right)^{*}\xi,Q\left(\tau,s\right)U\left(t,\tau\right)^{*}\xi\right\rangle _{X}\right)
\end{align*}

On the other hand, $\widehat{\mu_{1}}\left(\xi\right)=\exp\left(-\frac{1}{2}\left\langle \xi,Q\left(t,\tau\right)\xi\right\rangle _{X}\right)$;
therefore:
\begin{align*}
 & \widehat{\mu_{1}\star\mu_{2}}\left(\xi\right)=\widehat{\mu_{1}}\left(\xi\right)\widehat{\mu_{2}}\left(\xi\right)\\
= & \exp\Biggl(i\left\langle U\left(t,\tau\right)^{*}\xi,U\left(\tau,s\right)x\right\rangle _{X}-\frac{1}{2}\left\langle \xi,\left[Q\left(t,\tau\right)+U\left(t,\tau\right)Q\left(\tau,s\right)U\left(t,\tau\right)^{*}\right]\xi\right\rangle _{X}\Biggr)\\
= & \exp\left(i\left\langle \xi,U\left(t,s\right)x\right\rangle _{X}-\frac{1}{2}\left\langle \xi,Q\left(t,s\right)\xi\right\rangle _{X}\right)\\
= & \widehat{\mathcal{N}\left(U\left(t,s\right)x,Q\left(t,s\right)\right)}\left(\xi\right),
\end{align*}
where we have used the relation $U\left(t,\tau\right)Q\left(\tau,s\right)U\left(t,\tau\right)^{*}=Q\left(t,s\right)-Q\left(t,\tau\right)$,
which follows from $\text{\eqref{eq: def cov Q}}$. This leads to
$\text{\eqref{eq: conv uguale Gaussiana buona}}$, by the characterization
property of the Fourier transform mentioned in Remark $\text{\ref{rem: mis gauss e trasf F}}$,
since $\xi$ is generic.

\vspace{2mm}

ii) Let $\varphi\in C_{b}^{1}\left(X\right)$, $x\in X$, $h\in X\setminus\left\{ 0\right\} $.
We have:
\begin{align*}
 & \left(P_{s,t}\varphi\right)\left(x+h\right)-\left(P_{s,t}\varphi\right)\left(x\right)\\
= & \int_{X}\left[\varphi\left(y+U\left(t,s\right)h\right)-\varphi\left(y\right)\right]\mathcal{N}\left(U\left(t,s\right)x,Q\left(t,s\right)\right)\left(dy\right)\\
= & \int_{X}\left\{ \left\langle \nabla\varphi\left(y\right),U\left(t,s\right)h\right\rangle _{X}+o_{U\left(t,s\right)h\to0}\left(U\left(t,s\right)h\right)\right\} \mathcal{N}\left(U\left(t,s\right)x,Q\left(t,s\right)\right)\left(dy\right)\\
= & \left\langle \int_{X}\nabla\varphi\left(y\right)\mathcal{N}\left(U\left(t,s\right)x,Q\left(t,s\right)\right)\left(dy\right),U\left(t,s\right)h\right\rangle _{X}+o_{h\to0}\left(h\right)\\
= & \left\langle \left[\overrightarrow{P_{s,t}}\left(\nabla\varphi\right)\right]\left(x\right),U\left(t,s\right)h\right\rangle _{X}+o_{h\to0}\left(h\right)
\end{align*}
Thus:
\[
\frac{1}{\left|h\right|_{X}}\left\{ \left(P_{s,t}\varphi\right)\left(x+h\right)-\left(P_{s,t}\varphi\right)\left(x\right)-\left\langle U\left(t,s\right)^{*}\left[\overrightarrow{P_{s,t}}\left(\nabla\varphi\right)\right]\left(x\right),h\right\rangle _{X}\right\} \to0
\]
for $h\to0$, which proves $\text{\eqref{eq: formula grad C1}}$.

iii) Let $f\in C_{b}\left(X\right)$ and $t\in\mathbb{R}$. Observe
that, for every $s\leq t$:
\begin{align*}
\int_{X}\left|y\right|_{X}^{2}\mathcal{N}\left(U\left(t,s\right)x,Q\left(t,s\right)\right)\left(dy\right) & =\int_{X}\left|y+U\left(t,s\right)x\right|_{X}^{2}\mathcal{N}\left(0,Q\left(t,s\right)\right)\left(dy\right)\\
 & =\int_{X}\left|y\right|_{X}^{2}\mathcal{N}\left(0,Q\left(t,s\right)\right)\left(dy\right)+\left|U\left(t,s\right)x\right|_{X}^{2}\\
 & =trQ\left(t,s\right)+\left|U\left(t,s\right)x\right|_{X}^{2},
\end{align*}
and the latter sum converges to $trQ\left(t,-\infty\right)$, i.e.
to $\int_{X}\left|y\right|_{X}^{2}\mathcal{N}\left(0,Q\left(t,-\infty\right)\right)\left(dy\right)$,
for $s\to-\infty$, by Assumption $\text{\ref{assu: omega_0 neg}}$
and by Lemma $\text{\ref{lem: Q(t,-infty) trace class}}$. Since $Q\left(t,s\right)\to Q\left(t,-\infty\right)$
in $\mathcal{L}\left(X\right)$ (by definition of both quantities),
the convergence theorem for Gaussian measures cited in the final part
of Remark $\text{\ref{rem: mis gauss e trasf F}}$ implies that $\mathcal{N}\left(U\left(t,s\right)x,Q\left(t,s\right)\right)$
converges weakly to $\mathcal{N}\left(0,Q\left(t,-\infty\right)\right)$.
Thus, for every $f\in C_{b}\left(X\right)$:
\begin{align*}
\left[P_{s,t}f\right]\left(x\right) & =\int_{X}f\left(y\right)\mathcal{N}\left(U\left(t,s\right)x,Q\left(t,s\right)\right)\left(dy\right)\\
 & \to\int_{X}f\left(y\right)\mathcal{N}\left(0,Q\left(t,-\infty\right)\right)\\
 & \quad\text{for }s\to-\infty,
\end{align*}
which proves relation $\eqref{eq: conv P_s,t s a -infty}$.

Now we turn to the proof of $\text{\eqref{eq: conv P_s,t t a +infty}}$.
For every $s\leq t$, set for simplicity of notation:
\[
\nu_{t,s}:=\mathcal{N}\left(0,U\left(t,s\right)Q\left(s,-\infty\right)U\left(t,s\right)^{*}\right).
\]
Since $Q\left(t,-\infty\right)=Q\left(t,s\right)+U\left(t,s\right)Q\left(s,-\infty\right)U\left(t,s\right)^{*}$,
we find, with a computation similar to that used to prove point i),
that, for every $\xi\in X$:
\begin{align*}
 & \left(\mathcal{N}\left(0,Q\left(t,s\right)\right)\star\nu_{t,s}\right)^{\wedge}\left(\xi\right)=\widehat{\mathcal{N}\left(0,Q\left(t,s\right)\right)}\cdot\widehat{\nu_{t,s}}\left(\xi\right)\\
=\, & \exp\left(-\frac{1}{2}\left\langle \xi,Q\left(t,s\right)+U\left(t,s\right)Q\left(s,-\infty\right)U\left(t,s\right)^{*}\right\rangle _{X}\xi\right)\\
=\, & \widehat{\mathcal{N}\left(0,Q\left(t,-\infty\right)\right)}\left(\xi\right).
\end{align*}
Thus also the measures $\mathcal{N}\left(0,Q\left(t,s\right)\right)\star\nu_{t,s}$
and $\mathcal{N}\left(0,Q\left(t,-\infty\right)\right)$ agree. Therefore,
for every $f\in C_{b}\left(X\right)$ and $x\in X$:
\begin{align*}
 & \left[P_{s,t}f\right]\left(x\right)-\int_{X}f\left(y\right)\mathcal{N}\left(0,Q\left(t,-\infty\right)\right)\left(dy\right)\\
= & \int_{X}f\left(y\right)\mathcal{N}\left(U\left(t,s\right)x,Q\left(t,s\right)\right)\left(dy\right)-\int_{X}f\left(z\right)\mathcal{N}\left(0,Q\left(t,s\right)\right)\star\nu_{t,s}\left(dz\right)\\
= & \int_{X}\int_{X}\left[f\left(y+U\left(t,s\right)x\right)-f\left(y+u\right)\right]\nu_{t,s}\left(du\right)\mathcal{N}\left(0,Q\left(t,s\right)\right)\left(dy\right).
\end{align*}
If $f\in C_{b}^{\alpha}\left(X\right)$ for some fixed $\alpha\in\left(0,1\right)$,
then, recalling the definition of $\nu_{t,s}$:
\begin{align*}
 & \left|\left[P_{s,t}f\right]\left(x\right)-\int_{X}f\left(y\right)\mathcal{N}\left(0,Q\left(t,-\infty\right)\right)\left(dy\right)\right|\\
\leq & \int_{X}\int_{X}\left|f\left(y+U\left(t,s\right)x\right)-f\left(y+u\right)\right|\nu_{t,s}\left(du\right)\mathcal{N}\left(0,Q\left(t,s\right)\right)\left(dy\right)\\
\leq & \left[f\right]_{C^{\alpha}\left(X\right)}\int_{X}\left|U\left(t,s\right)x-u\right|_{X}^{\alpha}\nu_{t,s}\left(du\right)\\
\leq & \left[f\right]_{C^{\alpha}\left(X\right)}\left[\int_{X}\left|U\left(t,s\right)x-u\right|_{X}^{2}\nu_{t,s}\left(du\right)\right]^{\frac{\alpha}{2}}\\
= & \left[f\right]_{C^{\alpha}\left(X\right)}\left[\left|U\left(t,s\right)x\right|_{X}^{2}+tr\nu_{t,s}\right]^{\frac{\alpha}{2}}\\
\leq & \left[f\right]_{C^{\alpha}\left(X\right)}M_{\omega}^{\alpha}e^{\alpha\omega\left(t-s\right)}\left[\left|x\right|_{X}^{2}+trQ\left(s,-\infty\right)\right]^{\frac{\alpha}{2}}\\
\to & \ 0\quad\text{for }t\to\infty,
\end{align*}
where $\omega\in\left(\omega_{0},0\right)$ and $M_{\omega}$ is given
by Assumption $\text{\ref{assu: omega_0 neg}}$.
\end{proof}
\begin{proof}[Proof of Proposition $\text{\ref{prop: invarianza}}$.]

$ $

We begin by proving relation $\text{\eqref{eq: muinv}}$. Fix $s<t$.
Observe that this relation is equivalent to the fact the measures
$\mu_{t}$ and $P_{s,t}^{*}\mu_{s}$, defined on the Borel $\sigma$-algebra
$\mathcal{B}\left(X\right)$ generated by the open subsets of $X$,
coincide. By Remark $\text{\ref{rem: mis gauss e trasf F}}$, it is
sufficient that the equality $\widehat{\mu_{t}}=\widehat{P_{s,t}^{*}\mu_{s}}$
between Fourier transforms holds for the correspondent identity between
measures to be true.

Again by Remark $\text{\ref{rem: mis gauss e trasf F}}$, we have,
for every $\xi,y\in X$:
\begin{align*}
\left[P_{s,t}\left(e^{i\left\langle \xi,\cdot\right\rangle _{X}}\right)\right]\left(y\right) & =\int_{X}e^{i\left\langle \xi,U\left(t,s\right)y+z\right\rangle _{X}}\mathcal{N}\left(0,Q\left(t,s\right)\right)\left(dz\right)\\
 & =e^{i\left\langle \xi,U\left(t,s\right)y\right\rangle _{X}}\int_{X}e^{i\left\langle \xi,z\right\rangle _{X}}\mathcal{N}\left(0,Q\left(t,s\right)\right)\left(dz\right)\\
 & =e^{i\left\langle \xi,U\left(t,s\right)y\right\rangle _{X}}\widehat{\mathcal{N}\left(0,Q\left(t,s\right)\right)}\left(\xi\right)\\
 & =e^{i\left\langle \xi,U\left(t,s\right)y\right\rangle _{X}}e^{-\frac{1}{2}\left\langle \xi,Q\left(t,s\right)\xi\right\rangle _{X}}
\end{align*}
Thus, for every $\xi\in X$:
\begin{align*}
\widehat{P_{s,t}^{*}\mu_{s}}\left(\xi\right) & =\int_{X}e^{i\left\langle \xi,y\right\rangle _{X}}P_{s,t}^{*}\mu_{s}\left(dy\right)\\
 & =\int_{X}\left[P_{s,t}\left(e^{i\left\langle \xi,\cdot\right\rangle _{X}}\right)\right]\left(y\right)\mu_{s}\left(dy\right)\\
 & =e^{-\frac{1}{2}\left\langle \xi,Q\left(t,s\right)\xi\right\rangle _{X}}\int_{X}e^{i\left\langle U\left(t,s\right)^{*}\xi,y\right\rangle _{X}}\mu_{s}\left(dy\right)\\
 & =e^{-\frac{1}{2}\left\langle \xi,Q\left(t,s\right)\xi\right\rangle _{X}}\widehat{\mathcal{N}\left(0,Q\left(s,-\infty\right)\right)}\left(U\left(t,s\right)^{*}\xi\right)\\
 & =e^{-\frac{1}{2}\left\langle \xi,Q\left(t,s\right)\xi\right\rangle _{X}}e^{-\frac{1}{2}\left\langle U\left(t,s\right)^{*}\xi,Q\left(s,-\infty\right)U\left(t,s\right)^{*}\xi\right\rangle _{X}}\\
 & =\exp\left(-\frac{1}{2}\left\langle \xi,Q\left(t,s\right)\xi+U\left(t,s\right)Q\left(s,-\infty\right)U\left(t,s\right)^{*}\xi\right\rangle _{X}\right).
\end{align*}
Remembering the definitions of $Q\left(t,s\right)$ and $Q\left(s,-\infty\right)$
we immediately obtain that
\[
U\left(t,s\right)Q\left(s,-\infty\right)U\left(t,s\right)^{*}=Q\left(t,-\infty\right)-Q\left(t,s\right).
\]
Thus:
\begin{align*}
\widehat{P_{s,t}^{*}\mu_{s}}\left(\xi\right) & =\exp\left(-\frac{1}{2}\left\langle \xi,Q\left(t,-\infty\right)\xi\right\rangle _{X}\right)\\
 & =\widehat{N\left(0,Q\left(t,-\infty\right)\right)}\left(\xi\right)\\
 & =\widehat{\mu_{t}}\left(\left(\xi\right)\right).
\end{align*}

Relation $\text{\eqref{eq: muinv}}$ is thus proven.

Now observe that, as an immediate consequence of the definition of
$\overrightarrow{P_{s,t}}$ given in $\text{\eqref{eq: def O-U vett formula}}$,
we have
\[
\left\langle \overrightarrow{P_{s,t}}\Phi,e_{N}\right\rangle _{X}=P_{s,t}\left\langle \Phi,e_{N}\right\rangle _{X}\quad\forall N\in\mathbb{N},
\]

for every $N\in\mathbb{N}$. Thus relation $\text{\eqref{eq: muinv vettoriale}}$
is implied by $\text{\eqref{eq: muinv}}$, applied to the functions
$\phi=\left\langle \Phi,e_{N}\right\rangle _{X}$.
\end{proof}
\begin{proof}[Proof of Proposition $\text{\ref{prop: estensione P_s,t}}$.]

$ $

Based on relation $\text{\eqref{eq: muinv}}$, we show that, for every
$p\in[1,+\infty)$:
\begin{equation}
\left\Vert P_{s,t}\phi\right\Vert _{L^{p}\left(X,\mu_{s}\right)}\leq\left\Vert \phi\right\Vert _{L^{p}\left(X,\mu_{t}\right)}\quad\forall\phi\in C_{b}\left(X\right).\label{eq: Ps,t cont in Lp su C}
\end{equation}
Fix $\phi\in C_{b}\left(X\right)$ and $p\in[1,+\infty)$. By $\text{\eqref{eq: muinv}}$
we have
\begin{align*}
\int_{X}\left|\phi\left(x\right)\right|^{p}\mu_{t}\left(dx\right) & =\int_{X}\left[P_{s,t}\left|\phi\right|^{p}\right]\left(x\right)\mu_{s}\left(dx\right)\\
 & =\int_{X}\int_{X}\left|\phi\left(y+U\left(t,s\right)x\right)\right|^{p}N\left(0,Q\left(t,s\right)\right)\left(dy\right)\mu_{s}\left(dx\right)\\
 & \geq\int_{X}\left|P_{s,t}\phi\left(x\right)\right|^{p}\mu_{s}\left(dx\right),
\end{align*}
and thus $\text{\eqref{eq: Ps,t cont in Lp su C}}$ holds.

This allows us to define $P_{s,t}^{p}$ in a pretty standard way.
Consider $f\in L^{p}\left(X,\mu_{t}\right)$ and $\left(\phi_{n}\right)_{n}\subseteq C_{b}\left(X\right)$
such that $\phi_{n}\to f\text{ in }L^{p}\left(X,\mu_{t}\right)$.
By relation $\text{\eqref{eq: Ps,t cont in Lp su C}}$, the sequence
$\left(P_{s,t}\phi_{n}\right)_{n}$, which is obviously contained
in the space $L^{p}\left(X,\mu_{s}\right)$, is Cauchy with respect
to the norm of this space. Thus, it has a limit in $L^{p}\left(X,\mu_{s}\right)$,
which we call $P_{s,t}f$. The very same relation $\text{\eqref{eq: Ps,t cont in Lp su C}}$
ensures that the definition:
\begin{eqnarray*}
\text{} & P_{s,t}^{p}f:={\displaystyle \lim_{n\to\infty}}P_{s,t}\phi_{n}\text{ in }L^{p}\left(X,\mu_{s}\right)\\
 & \forall f\in L^{p}\left(X,\mu_{t}\right),\ \left(\phi_{n}\right)_{n}\subseteq C_{b}\left(X\right)\text{ s.t. }\phi_{n}\to f\text{ in }L^{p}\left(X,\mu_{t}\right)
\end{eqnarray*}
is well posed since it does not depend on the choice of $\left(\phi_{n}\right)_{n}$. 

Consider relations $\text{\eqref{eq: muinv}}$ and $\text{\eqref{eq: Ps,t cont in Lp su C}}$
for $\phi=\phi_{n}$: by passing to the limit for $n\to\infty$, we
obtain that they also hold for $\phi=f$ and $P_{s,t}=P_{s,t}^{p}$
(in order to pass to the limit in $\text{\eqref{eq: muinv}}$ it is
convenient to observe that $L^{p}\left(X,\nu\right)$ is continuously
embedded in $L^{1}\left(X,\nu\right)$ when $\nu\left(X\right)<+\infty$).
In particular $P_{s,t}^{p}:L^{p}\left(X,\mu_{t}\right)\to L^{p}\left(X,\mu_{s}\right)$
is a contraction.

The vectorial counterparts of these relations follow. Using the standard
properties of the Bochner integral we can see that the following inequality
holds for $\Phi\in C_{b}\left(X;X\right):$
\[
\left\Vert \overrightarrow{P_{s,t}}\Phi\right\Vert _{L^{p}\left(X,\mu_{s};X\right)}\leq\left\Vert \Phi\right\Vert _{L^{p}\left(X,\mu_{t};X\right)}
\]

in the same way as we did in order to prove $\text{\eqref{eq: Ps,t cont in Lp su C}}$.
Thus, with the same construction, we can extend $\overrightarrow{P_{s,t}}:C_{b}\left(X;X\right)\to C_{b}\left(X;X\right)$
to a contraction operator $\overrightarrow{P_{s,t}^{p}}:L^{p}\left(X,\mu_{t};X\right)\to L^{p}\left(X,\mu_{s};X\right)$.
Upon this construction, $\overrightarrow{P_{s,t}^{p}}F$ is defined
as the $L^{p}\left(X,\mu_{s};X\right)$-limit of the sequence $\left(\overrightarrow{P_{s,t}^{p}}\Phi_{n}\right)_{n}$
where $\left(\Phi_{n}\right)_{n}\subseteq C_{b}\left(X;X\right)$
converges to $F$ in $L^{p}\left(X,\mu_{t};X\right)$; thus we also
have $\overrightarrow{P_{s,t}^{p}}\Phi_{n}\to\overrightarrow{P_{s,t}^{p}}F$
in $L^{1}\left(X,\mu_{s};X\right)$, which implies that relation $\text{\eqref{eq: muinv vettoriale}}$
which has already been established for $\overrightarrow{P_{s,t}}$
and continuous functions, also holds for $\overrightarrow{P_{s,t}^{p}}$
and $L^{p}\left(X,\mu_{t};X\right)$ functions.
\end{proof}
\section{Characterization of the Cameron-Martin space}\label{appendice CM}

In this section we prove Proposition $\text{\ref{prop: caratt}}$.
Note that, in the statement, the operator $Q$ is assumed to be merely
bounded, self-adjoint and non-negative. In particular, no compactness
property is required for this argument to work.

The main tool is the following Proposition.
\begin{prop}
\label{prop: A_ext}Let $\left(X,\,\left\langle \cdot,\cdot\right\rangle \right)$
be a real Hilbert space, $D\left(A\right)$ a linear subspace of $X$
and $A:D\left(A\right)\to\overline{D\left(A\right)}$ a linear operator
(not necessarily bounded nor densely defined). Assume that $A$ is
symmetric in $D\left(A\right)$, namely:
\begin{equation}
\left\langle Ax,y\right\rangle =\left\langle Ay,x\right\rangle \quad\forall x,y\in D\left(A\right).\label{eq: autoagg A}
\end{equation}
Define
\[
E\left(A\right):=\left\{ x\in X\mid\text{the map }y\to\left\langle Ay,x\right\rangle \text{ is continuous in }D\left(A\right)\right\} .
\]
Then $E\left(A\right)$ is a linear subspace of $X$ containing $D\left(A\right)$,
and there exists a linear operator $A_{ext}:E\left(A\right)\to\overline{D\left(A\right)}$
which extends $A$ and satisfies:
\begin{equation}
\left\langle A_{ext}x,y\right\rangle =\left\langle Ay,x\right\rangle \quad\forall x\in E\left(A\right),y\in D\left(A\right).\label{eq: rel A Aext}
\end{equation}
\end{prop}

\begin{proof}
Define, for every $x\in X$, $f_{x}:D\left(A\right)\to\mathbb{R}$
as $f_{x}\left(y\right):=\left\langle Ay,x\right\rangle $. Clearly
$D\left(A\right)\subseteq E\left(A\right)$ because, for every $x\in D\left(A\right)$,
we have $f_{x}=\left\langle \cdot,Ax\right\rangle $ in virtue of
$\text{\eqref{eq: autoagg A}}$.

The set $E\left(A\right)$ is clearly a vector space. Indeed having
fixed $x_{1},x_{2}\in E\left(A\right)$ and $t,s\in\mathbb{R}$, it
is immediate to verify that $f_{tx_{1}+sx_{2}}=tf_{x_{1}}+sf_{x_{2}}$
in $D\left(A\right)$. Thus, since $f_{x_{1}}$ e $f_{x_{2}}$ are
continuous in $D\left(A\right)$, also $f_{tx_{1}+sx_{2}}$ is such,
namely $tx_{1}+sx_{2}\in E\left(A\right)$.

\vspace{2mm}

Now we define a linear extension of $A$ to $E\left(A\right)$.

For any $x\in E\left(A\right)$, let $\bar{f}_{x}:\overline{D\left(A\right)}\to\mathbb{R}$
be the continuous extension of $f_{x}$ to $\overline{D\left(A\right)}$,
i.e.
\[
\bar{f}_{x}\left(y\right):=\lim_{n\to\infty}f_{x}\left(y_{n}\right)
\]
for every $y\in\overline{D\left(A\right)}$ and $\left(y_{n}\right)_{n}\subseteq D\left(A\right)$
such that $y_{n}\to y$. Note that the above definition is well posed
due to the condition $x\in E\left(A\right)$. Indeed the limit exists
because the sequence $\left(f_{x}\left(y_{n}\right)\right)_{n}\subseteq\mathbb{R}$
is Cauchy, and is unique - namely, it does not depend on $\left(y_{n}\right)_{n}$.
Further, $\bar{f}_{x}$ is continuous in $\overline{D\left(A\right)}$,
which, as a closed linear subspace, inherits the Hilbert space structure
from $X$. Thus, by the Riesz representation theorem, there exists
a vector $A_{ext}x\in\overline{D\left(A\right)}$ such that $\bar{f}_{x}=\left\langle A_{ext}x,\cdot\right\rangle $.

Thus we have, for every $y\in D\left(A\right)$:
\begin{align*}
\left\langle Ay,x\right\rangle  & =f_{x}\left(y\right)=\bar{f}_{x}\left(y\right)=\left\langle A_{ext}x,y\right\rangle ,
\end{align*}
which proves relation $\text{\eqref{eq: rel A Aext}}$.

\vspace{2mm}

We now show that the map $x\to A_{ext}x$ is linear from $E\left(A\right)$
to $\overline{D\left(A\right)}$.

Let $x_{1},x_{2}\in E\left(A\right)$ and $t,s\in\mathbb{R}$. Then
$A_{ext}\left(tx_{1}+sx_{2}\right)\in\overline{D\left(A\right)}$
represents the continuous linear functional $\bar{f}_{tx_{1}+sx_{2}}$.
We have for every $y\in\overline{D\left(A\right)}$, chosen $y_{n}\to y$
with $\left(y_{n}\right)_{n}\subseteq D\left(A\right)$:
\begin{align*}
\left\langle A_{ext}\left(tx_{1}+sx_{2}\right),y\right\rangle  & =\bar{f}_{tx_{1}+sx_{2}}\left(\lim_{n}y_{n}\right)\\
 & =\lim_{n}f_{tx_{1}+sx_{2}}\left(y_{n}\right)\\
 & =t\lim_{n}f_{x_{1}}\left(y_{n}\right)+s\lim_{n}f_{x_{2}}\left(y_{n}\right)\\
 & =t\bar{f}_{x_{1}}\left(y\right)+s\bar{f}_{x_{2}}\left(y\right)\\
 & =\left\langle tA_{ext}x_{1},y\right\rangle +\left\langle sA_{ext}x_{2},y\right\rangle \\
 & =\left\langle tA_{ext}x_{1}+sA_{ext}x_{2},y\right\rangle .
\end{align*}
Therefore $A_{ext}\left(tx_{1}+sx_{2}\right)=tA_{ext}x_{1}+sA_{ext}x_{2}$
because $y$ is generic.

\vspace{2mm}

Finally we verify that $A_{ext}=A$ in $D\left(A\right)$. Fix $x\in D\left(A\right)$,
$y\in\overline{D\left(A\right)}$, and again $y_{n}\to y$ with $\left(y_{n}\right)_{n}\subseteq D\left(A\right)$.
Hence, remembering $\text{\eqref{eq: autoagg A}}$:
\[
\left\langle Ax,y\right\rangle =\lim_{n}\left\langle x,Ay_{n}\right\rangle =\lim_{n}f_{x}\left(y_{n}\right)=\bar{f}_{x}\left(y\right)=\left\langle A_{ext}x,y\right\rangle .
\]
Since $Ax\in\overline{D\left(A\right)}$ by our hypothesis on $A$,
we infer that $Ax=A_{ext}x$.
\end{proof}
\begin{cor}
\label{cor: caratt dominio aggiunto}Let $T\in\mathcal{L}\left(X\right)$
be a self-adjoint operator and let
\[
T^{-1}:Im\left(T\right)\to\overline{ImT}
\]
be its pseudoinverse. Let $E\left(T^{-1}\right)$ and $T_{ext}^{-1}$
as in Proposition $\text{\ref{prop: A_ext}}$. Then:

i) $KerT\subseteq E\left(T^{-1}\right)$ and $T_{ext}^{-1}\left(KerT\right)=\left\{ 0\right\} $,

ii) $E\left(T^{-1}\right)=ImT\oplus KerT$.
\end{cor}

\begin{proof}
Since $ImT^{-1}\subseteq\overline{ImT}$, Proposition $\text{\ref{prop: A_ext}}$
can be applied with $A=T^{-1}$ and $D\left(A\right)=ImT$, because
relation $\text{\eqref{eq: autoagg A}}$ is satisfied by this choice
of $A$ and $D\left(A\right)$. We deduce that there exists a linear
extension $T_{ext}^{-1}$ of $T^{-1}$ to the space $E\left(T^{-1}\right)\supseteq ImT$
that satisfies $\text{\eqref{eq: rel A Aext}}$.

i) Fix $x\in KerT$. We have $\overline{ImT}=KerT^{\perp}$ (because
$T$ is self adjoint), so the operator $\left\langle T^{-1}\left(\cdot\right),x\right\rangle $
is null and thus continuous in $ImT$, namely $x\in E\left(T^{-1}\right)$.

Further, relation $\text{\eqref{eq: rel A Aext}}$ implies, for every
$y\in ImT$:
\[
\left\langle T_{ext}^{-1}x,y\right\rangle =\left\langle T^{-1}y,x\right\rangle =0.
\]
Thus, since $Im\left(T_{ext}^{-1}\right)\subseteq\overline{ImT}$
by construction, $T_{ext}^{-1}x=0$.

ii) We prove $\left(\subseteq\right)$. Let $x\in E\left(T^{-1}\right)$.
It follows from $\text{\eqref{eq: rel A Aext}}$ that:
\[
\left\langle T_{ext}^{-1}x,Tz\right\rangle =\left\langle x,T^{-1}Tz\right\rangle \quad\forall z\in X.
\]
Therefore, remembering that $T^{-1}Tz=z$ if $z\in\overline{ImT}$,
we infer:
\begin{equation}
\left\langle TT_{ext}^{-1}x,z\right\rangle =\left\langle x,z\right\rangle \quad\forall z\in\overline{ImT}.\label{eq: proiezione}
\end{equation}
This means precisely that $TT_{ext}^{-1}x$ is the orthogonal projection
of $x$ onto the space $\overline{ImT}$, and thus $x\in ImT+KerT$.

As far as the inclusion $\left(\supseteq\right)$ is concerned, take
$x=T\xi+z$, with $z\in KerT$; then for every $y\in ImT$:
\[
\left\langle T^{-1}y,x\right\rangle =\left\langle T^{-1}y,T\xi\right\rangle +\left\langle T^{-1}y,z\right\rangle =\left\langle y,\xi\right\rangle .
\]
Such expression is continuous in $y$, thus $x\in E\left(T^{-1}\right)$.
\end{proof}
\begin{rem}
The utility of Corollary $\text{\ref{cor: caratt dominio aggiunto}}$
is apparent. The most common way to prove that a vector $v$ belongs
to the image of an operator relies on the fact that this image is
a closed set, then using a density argument. Point ii) in Corollary
$\text{\ref{cor: caratt dominio aggiunto}}$ instead applies to the
image of a bounded, self-adjoint operator without knowing if the image
is closed: it states that a certain object, namely the orthogonal
projection onto $\overline{ImT}$ of a vector $x\in E\left(T^{-1}\right)$,
belongs to $ImT$ without using any density argument. Actually the
projection is written explicitly as $TT_{ext}^{-1}x$ - and this shows
in its turn the importance of the preliminary Proposition $\text{\ref{prop: A_ext}}$.
\end{rem}

This machinery will be used to prove the inclusion $ImQ^{\frac{1}{2}}\supseteq H_{Q}$
in Proposition $\text{\ref{prop: caratt}}$.
\begin{proof}[Proof of Proposition $\text{\ref{prop: caratt}}$]
 Let us begin with the last statement. Take $x\in ImQ^{\frac{1}{2}}$;
we can assume that $x\neq0$, without loss of generality.

Since $x=Q^{\frac{1}{2}}Q^{-\frac{1}{2}}x$ and $Q^{\frac{1}{2}}$
is self-adjoint, we have for every $z\in X$ such that $\left|Q^{\frac{1}{2}}z\right|_{X}\leq1$:
\[
\left\langle x,z\right\rangle _{X}=\left\langle Q^{-\frac{1}{2}}x,Q^{\frac{1}{2}}z\right\rangle _{X}\leq\left|Q^{-\frac{1}{2}}x\right|_{X}.
\]
Thus, the quantity $\left|Q^{-\frac{1}{2}}x\right|_{X}$ is an upper
bound of the set $\mathcal{D}\left(x;Q\right)$.

Furthermore, recall that $Q^{-\frac{1}{2}}x\in\overline{ImQ^{\frac{1}{2}}}$,
by the previous considerations; thus, having taken $\left(z_{n}\right)_{n}\subseteq X$
such that $Q^{\frac{1}{2}}z_{n}\to Q^{-\frac{1}{2}}x$ in $X$ as
$n\to\infty$:
\[
\left\langle x,\frac{z_{n}}{\left|Q^{-\frac{1}{2}}z_{n}\right|_{X}}\right\rangle _{X}=\left\langle Q^{-\frac{1}{2}}x,\frac{Q^{\frac{1}{2}}z_{n}}{\left|Q^{-\frac{1}{2}}z_{n}\right|_{X}}\right\rangle _{X}\overset{n\to\infty}{\to}\left|Q^{-\frac{1}{2}}x\right|_{X}.
\]
Observe that the division by the norm of $Q^{-\frac{1}{2}}z_{n}$
is legitimate since we are assuming that such sequence converges to
a vector which is not null. The above convergence means that $\left|Q^{-\frac{1}{2}}x\right|_{X}\leq\sup\mathcal{D}\left(x;Q\right)$.
Hence, the equality holds.

Observe that this argument proves in particular that $ImQ^{\frac{1}{2}}\subseteq H_{Q}$.

\vspace{2mm}

Now let us turn to the opposite inclusion.

Fix $x\in H_{Q}$, i.e. such that $\sup\mathcal{D}\left(x,Q\right)<+\infty$.
Then it follows from the definition of $\mathcal{D}\left(x,Q\right)$
that:
\[
\left\langle \frac{Q^{-\frac{1}{2}}y}{\left|y\right|_{X}},x\right\rangle \leq\sup\mathcal{D}\left(x,Q\right)\quad\forall y\in ImQ^{\frac{1}{2}}\setminus\left\{ 0\right\} ,
\]
which implies that the linear functional $y\mapsto\left\langle Q^{-\frac{1}{2}}y,x\right\rangle $
is continuous in its domain $ImQ^{\frac{1}{2}}$. In other terms,
$x$ belongs to the domain $E\left(Q^{-\frac{1}{2}}\right)$ of the
adjoint operator $\left(Q^{-\frac{1}{2}}\right)^{*}$. Recall that
$Q^{-\frac{1}{2}}$ is not a bounded self-adjoint operator, but merely
a symmetric operator in its domain $ImQ^{\frac{1}{2}}$, so $E\left(Q^{-\frac{1}{2}}\right)$
does not necessarily coincide with $ImQ^{\frac{1}{2}}$, and is, in
general, bigger. Actually, by point ii) in Corollary $\text{\ref{cor: caratt dominio aggiunto}}$,
$E\left(Q^{-\frac{1}{2}}\right)=ImQ^{\frac{1}{2}}\oplus KerQ^{\frac{1}{2}}$,
and thus
\[
x=Q^{\frac{1}{2}}\xi+z,\quad\text{with }\xi\in X,\ Q^{\frac{1}{2}}z=0.
\]
Now we use again the fact that $\sup\mathcal{D}\left(x,Q\right)<+\infty$
to prove that $z=0$. Since in particular $nz\in Q^{-\frac{1}{2}}\left(\overline{B\left(0,1\right)}\right)$
for every $n\in\mathbb{N}$, we have:
\[
n\left|z\right|_{X}^{2}=\left\langle x,nz\right\rangle _{X}\leq\sup\mathcal{D}\left(x,Q\right)\quad\forall n\in\mathbb{N},
\]
which implies $z=0$. Thus $x\in ImQ^{\frac{1}{2}}$, which concludes
the proof.
\end{proof}


\begin{thebibliography}{10}
\bibitem{Acq - Ev op} \textsc{Acquistapace, P.}, \emph{Evolution
operators and strong solutions of abstract linear parabolic equations},
Differ. Integral Equ. 1(4), 1988, 433-457.

\bibitem{AFT}\textsc{Acquistapace, P.}, \textsc{Flandoli, F.} and \textsc{Terreni, B.},
\emph{Initial boundary value problems and optimal control for nonautonomous
parabolic systems},\emph{ }Siam J. Control and Optimization 29(1),
1991, 89-118.

\bibitem{Add13}\textsc{Addona, D.}, \emph{Nonautonomous Ornstein-Uhlenbeck
operators in weighted spaces of continuous functions}, Semigr. Forum
87 (2013).

\bibitem{AddAngLor17} \textsc{Addona, D.}, \textsc{Angiuli, L.} and \textsc{Lorenzi, L.}, \emph{Hypercontractivity, supercontractivity, ultraboundedness and stability in semilinear problems},
Advances Nonlin. Anal. 8 (2017).

\bibitem{ADDONA-Poinc}\textsc{Addona, D., Muratori, M.} and \textsc{Rossi, M.},
\emph{On the equivalence of Sobolev norms in Malliavin spaces.} J.
Funct. Anal. 283(7) (2022).

\bibitem{AngLorLun13}\textsc{Angiuli, L.} and \textsc{Lorenzi, L.}
and \textsc{Lunardi, A.}, Hypercontractivity and Asymptotic Behavior
in Nonautonomous Kolmogorov Equations. In: Communications in Partial
Differential Equations 38(12) (2013).

\bibitem{AngLor141} \textsc{Angiuli, L.} and \textsc{Lorenzi, L.}. \emph{Non autonomous parabolic problems with unbounded coefficients in unbounded domains},
Advances in Differential Equations, 20 (2014).

\bibitem{AngLor142}\textsc{Angiuli, L.} and \textsc{Lorenzi, L.},
\emph{On improvement of summability properties in nonautonomous Kolmogorov
equations}, Communcations in Pure ans Applied Mathematics (2014).

\bibitem{AngLor16}\textsc{Angiuli, L.} and \textsc{Lorenzi, L.},
\emph{On the estimates of the derivatives of solutions to nonautonomous
Kolmogorov equations and their consequences}, Riv. Math. Univ. Parma
(N.S.)7(2) (2016).

\bibitem{BigDef23}\textsc{Bignamini, D.} and \textsc{De Fazio, P.},
\emph{Log-Sobolev inequalities and hypercontractivity for Ornstein
-- Uhlenbeck evolution operators in infinite dimension, }J. Evol.
Equ. 24, 78 (2024).

\bibitem{BOG;GM}\textsc{Bogachev, V.}, Gaussian Measures. Mathematical
Surveys and Monographs, Vol. 62. American Mathematical Society (1991).

\bibitem{Cer21} \textsc{Cerrai, S.}, and \textsc{Lunardi, A.},
\emph{Smoothing effects and maximal Hölder regularity for non-autonomous
Kolmogorov equations in infinite dimension}, J. Differ. Equ. 434,
113245 (2025)

\bibitem{DAN} \textsc{Daners, D.}, Heat Kernel Estimates for Operators
with Boundary Conditions, Math. Nachr. 217 (2000), 13 - 41.

\bibitem{Def23}\textsc{De Fazio, P.}\emph{, On smoothing in non
autonomous ornstein-uhlenbeck equations in infinite dimensions, }arXiv:2212.05559
{[}math.AP{]} (2023)

\bibitem{DAP-LUN}\textsc{Da Prato, G.} and \textsc{Lunardi, A.},
\emph{Ornstein-Uhlenbeck operators with time periodic coefficients},
J. evol. equ. 7 (2007), 587-614.

\bibitem{DAP-ZAB}\textsc{Da Prato, G.} and \textsc{Zabczyk, J.},
Second Order Partial Differential Equations in Hilbert Spaces. Cambridge
University Press, Cambridge (2004).

\bibitem{DAP-ZAB2}\textsc{Da Prato, G.} and \textsc{Zabczyk, J.},
Stochastic Equations in Infinite Dimensions. 2nd ed. Cambridge University
Press, Cambridge (2014). 

\bibitem{GEISS-LUN_0}\textsc{Geissert, M.} and \textsc{Lunardi, A.},
\emph{Invariant measures and maximal $L^{2}$ regularity for nonautonomous
Ornstein--Uhlenbeck equations}, J. London Math. Soc. 77(3) (2008),
719-740.

\bibitem{GEISS-LUN}\textsc{Geissert, M.} and \textsc{Lunardi, A.},
\emph{Asymptotic behavior and hypercontractivity in non-autonomous
Ornstein--Uhlenbeck equations}, J. London Math. Soc. 79(1) (2009),
85-106.

\bibitem{GIL-TRUD}\textsc{Gilbarg, D.} and \textsc{Trudinger, N. S.},
Elliptic Partial Differential Equations of Second Order, Reprint of
the 1998 Edition. Springer-Verlag Berlin Heidelberg New York (1998).

\bibitem{KNA}\textsc{Kn\"able, F.}, \emph{Ornstein--Uhlenbeck equations
with time-dependent coefficients and Lévy noise in finite and infinite
dimensions}, J. Evol. Equ. 11 (2011), 959-993.

\bibitem{KunLorLun10}\textsc{Kunze, M.}, \textsc{Lorenzi, L.},
and \textsc{Lunardi, A.}, \emph{Nonautonomous Kolmogorov parabolic
equations with unbounded coefficients}, Transactions of the American
Mathematical Society 362(1) (2010), 169-198.

\bibitem{KunLorRha14}\textsc{Kunze, M.}, \textsc{Lorenzi, L.},
and \textsc{Rhandi, A.}, \emph{Kernel estimates for nonautonomous
Kolmogorov equations with potential term} (English summary), New prospects
in direct, inverse and control problems for evolution equations.

\bibitem{KunLorRha16}\textsc{Kunze, M.}, \textsc{Lorenzi, L.},
and \textsc{Rhandi, A.}, \emph{Kernel estimates for nonautonomous
Kolmogorov equations.}

\bibitem{LorLunZam10}\textsc{Lorenzi, L.}, \textsc{Lunardi, A.},
and \textsc{Zamboni, A}., \emph{Asymptotic behavior in time periodic
parabolic problems with unbounded coefficients} (English summary),
J. Differ. Equ. 249(12) (2010).

\bibitem{Lor11}\textsc{Lorenzi, L.}, \emph{Optimal H\"older regularity
for nonautonomous Kolmogorov equations}, Discrete Contin. Dyn. Syst.
Ser. S4(1) (2011), 169-191.

\bibitem{LorLunSch16}\textsc{Lorenzi, L.}, \textsc{Lunardi, A.},
\textsc{Schnaubelt, R.}, \emph{Strong convergence of solutions to
nonautonomous Kolmogorov equations}, Proc. Amer. Math. Soc.144(9)
(2016).

\bibitem{Luna int}\textsc{Lunardi, A.}, Interpolation Theory, 3d
edition. Scuola Normale Superiore, Pisa (2018).

\bibitem{OuyRoc}\textsc{Ouyang, S.}, and \textsc{R\"ockner, M.},
\emph{Time inhomogeneous generalized Mehler semigroups and skew convolution
equations,} Forum Math. 28(2) (2016), 339-376.

\bibitem{Pazy}\textsc{Pazy, A.}, Semigroups of Linear Operators
and Applications to Partial Differential Equations. Springer-Verlag
New York (1983).

\bibitem{ROL}\textsc{Schnaubelt, R.}, \emph{Asymptotic behaviour
of parabolic nonautonomous evolution equations,} chapter in Functional
Analytic Methods for Evolution Equations by Da Prato, G., Kunstmann,
P. C., Wies, L., Lasiecka, I., Lunardi A., Schnaubelt, R. Springer
Berlin, Heidelberg (2004).

\bibitem{Triebel}\textsc{Triebel, H.}, Interpolation Theory, Function
Spaces, Differential Operators. North Holland Amsterdam New York Oxford
(1978).

\end{thebibliography}
\end{document}